\documentclass[11pt]{article}
\usepackage{eurosym}
\usepackage{amsfonts}
\usepackage{amsmath}
\usepackage{geometry}
\usepackage{amssymb}
\usepackage{mathabx}
\usepackage{graphicx}
\usepackage{color}
\usepackage{dsfont}
\usepackage{bbm}
\usepackage[dvipsnames]{xcolor}
\usepackage{mathrsfs}
\usepackage[author={Oana}]{pdfcomment}
\usepackage[colorinlistoftodos]{todonotes}
\usepackage{titling}
\usepackage{lipsum}
\usepackage{soul}
\usepackage{amssymb}  
\usepackage{amsthm}   

\usepackage{calc}

\usepackage{accents}
\newcommand{\dbtilde}[1]{\accentset{\approx}{#1}}

\newtheorem{theorem}{Theorem}

\newtheorem{corollary}[theorem]{Corollary}

\newtheorem{definition}[theorem]{Definition}

\newtheorem{lemma}[theorem]{Lemma}

\newtheorem{proposition}[theorem]{Proposition}
\newtheorem{remark}[theorem]{Remark}

\newcommand{\T}{\mathbb{T}}

\newcommand{\cL}{\mathcal{L}}

\newcommand{\cT}{\mathcal{T}}
\newcommand{\cD}{\mathcal{D}}
\newcommand{\cW}{\mathcal{W}}

\newcommand{\cJ}{\mathcal{J}}

\newcommand*{\eea}{\end{array}}
\newcommand{\ckX}{\check{X}}
\newcommand{\cku}{\check{u}}
\newcommand{\ckh}{\check{h}}

\newcommand*{\bme}{\begin{multiequations}}
\newcommand*{\eme}{\end{multiequations}}

\newcommand{\alf}{\frac{1}{2}}

\renewcommand*{\Omega}{\varOmega} 
\newcommand{\sub}{\scriptscriptstyle} 
\renewcommand*{\vec}{\boldsymbol} 
\newcommand{\df}{\mathrm{d}} 
\newcommand{\dt}{\df t} 
\newcommand{\sdbt}{\sigma \df B_t} 
\newcommand{\sdbth}{\sigma_{\sub H} \df \boldsymbol{B}_t} 
\newcommand{\sdbtz}{\sigma_{\sub z} \df B_t} 
\newcommand{\dpt}{\df p_t^{\sigma}} 
\newcommand{\Df}{\mathrm{\mathbb{D}}} 

\newcommand{\bdot}{\boldsymbol{\cdot}} 
\newcommand{\grad}{\nabla} 
\newcommand{\gradh}{\nabla_{\sub H}}
\newcommand{\tp}{^{\scriptscriptstyle T}} 
\newcommand{\Exp}{\mathbb{E}} 
\newcommand{\tr}{\mathrm{tr}} 
\renewcommand*{\div}{\nabla\bdot} 
\newcommand{\divh}{\nabla_{\sub H}\boldsymbol{\cdot}}
\newcommand{\adv}{\bdot\nabla} 

\newcommand{\laplac}{\nabla^{2}} 
\newcommand{\real}{\mathbb{R}} 

\providecommand\bcdot{\boldsymbol{\cdot}}

\renewcommand*{\Omega}{\varOmega}
\renewcommand*{\Sigma}{\varSigma}

\newcommand{\nab}{\boldsymbol{\nabla}}

\begin{document}

\date{}
\title{\textbf{Analytical Properties for a Stochastic Rotating Shallow Water Model under Location Uncertainty}}
\author{Oana Lang \qquad Dan Crisan \qquad Etienne M\'emin }
\maketitle

\begin{abstract}
    The rotating shallow water model is a simplification of oceanic and atmospheric general circulation models that are used in many applications such as surge prediction, tsunami tracking and ocean modelling. In this paper we introduce a class of rotating shallow water models which are stochastically perturbed in order to incorporate model uncertainty into the underlying system. The stochasticity is chosen in a judicious way, by following the principles of  location uncertainty, as introduced in \cite{Memin2014}. We prove that the resulting equation is part of a class of stochastic partial differential equations that have unique maximal strong solutions. The methodology is based on the construction of an approximating sequence of models taking value in an appropriately chosen finite-dimensional Littlewood-Paley space. Finally, we show that a distinguished element of this class of stochastic partial differential equations has a global weak solution.    

\end{abstract}


\section{Introduction}
Numerical ocean models still exhibit a limiting “irreducible imprecision” compared to measured quantities in the turbulent regimes of the ocean dynamics. Air-sea interactions, small-scale turbulent parameterization, boundary processes, mixing and biogeochemical fluxes are currently very badly constrained in models, leading to large uncertainties for the prediction of climate and the modelling of the ocean biogeochemical cycles. Problematic issues also occur for the short-time forecasting of extreme events such as tropical cyclones and surges. The difficulty of parameterizing small-scale physical processes effects together with the need of probabilistic ensemble forecasting and uncertainty quantification calls for the introduction of new stochastic fluid dynamics models. To be deployed efficiently, such models must be generic and fit general physical conservation laws. 
Our efforts are in line with those advocated in \cite{Hasselmann-76},\cite{Palmer-2019}-\cite{Berner-2017} (and references therein) towards developing generic ``stochastic physics'' principles for fluid flows. Analytical solutions or accurate numerical simulations of these models depend greatly on the mathematical knowledge we have on them. This study is a step toward the development of physically meaningful (in a sense that will be described below) stochastic fluid dynamics models. 

 The introduction of  stochastic modeling principles for geophysical flows has been a focus of great attention  in climate study, meteorology and oceanography since the seminal works of \cite{Hasselmann-76} and \cite{Leith-75}. Most of the schemes proposed so far are built from pragmatic considerations or from homogenisation/averaging techniques coupled with linear or quasi-linear ansatz dynamics to encode the evolution of the small-scale random processes (see \cite{Berner-2017, Franzke-2015, Gottwald-2017, Majda-2008, Palmer-2019} and reference therein for extensive reviews on the subject). In particular, pragmatic schemes built from parameters' perturbation have  shown to bring some noticeable improvements in weather forecasting applications (see for instance \cite{Palmer-2019} for some examples). However, these schemes remain ad-hoc in their conception and lack of general solid grounds to be extended to other models or configurations just as does classical physics. They also  face some theoretical issues, for example the  control the variance brought by the random terms, which may, beyond numerical stability, deeply change the asymptotic nature of the underlying dynamical system, even for low noise magnitude (see \cite{Chapron-2018-QJRMS} for an example of this phenomena on the classical Lorenz 63 model).  
 
 A new class of general models based on stochastic transport has been recently proposed \cite{Bauer-2020-JPO,Memin2014,Resseguier-2017-GAFD-I}. This framework, referred to as modelling under location uncertainty (LU), has the advantage to be derived from physical conservation laws expressed through the stochastic transport of fluid parcels. As such, the LU models are directly extendable to classical approximations of geophysical dynamics. One of their strong assets is that they enjoy proper energy conservation and provide new approaches to subgrid modelling, expressed both in terms of fluctuation distributions, and spatial/temporal correlations. 
 
 In this study we will introduce and analyse a rotating shallow water model (RSW) built under the location uncertainty principle. We will begin by presenting the general principles undergoing the LU scheme and provide a thorough description of the derivation of the stochastic rotating shallow water under location uncertainty (LU-SRWS).\\[0.5mm]
 
 \noindent\textbf{Modelling under location uncertainty} \\
 The LU principle relies on a decomposition of the Lagrangian displacement
\begin{equation}\label{eq:dX}
	\df X_t = u({X}_t, t)\, \dt + \sigma (X_t, t)\, \df B_t,
\end{equation}
 in terms of a large-scale velocity component  $u$ that is both spatially and temporally correlated, and a highly oscillating  unresolved component  (also called noise) term $\sdbt$ that is correlated in space, but not in time. The noise is defined by a cylindrical Brownian motion $B_t$ \cite{Daprato-Zabczyk} and its spatial structure  is specified, within the bounded spatial domain $\Omega \subset \real^d\ (d = 2\ \text{or}\ 3)$,  through a deterministic Hilbert-Schmidt operator $\sigma: (L^2 (\Omega))^d \to (L^2 (\Omega))^d$, acting on square integrable vector-valued functions $f\in (L^2 (\Omega))^d$, with a bounded kernel  $\breve{\sigma}$ such that 
\begin{equation}\label{eq:corr}
	\sigma [f] (x, t)  = \int_{\Omega} \breve{\sigma} (x, y, t) f (y)\, \df y,\ \quad \forall f \in (L^2 (\Omega))^d.
\end{equation}
The random flow $\sdbt$ is a centered Gaussian process with the well-defined \emph{covariance tensor}:
\begin{align}\label{eq:cov}
	Q (x, y, t, s) &= \Exp \Big[ \big( \sigma (x, t)\, \df B_t \big) \big( \sigma (y, s)\, \df B_s \big)\tp \Big] \nonumber \\
	&= \delta(t-s)\, \dt \int_{\Omega}{\breve{\sigma} (x, z, t) \breve{\sigma}\tp (y, z, s)}\, \df z,
\end{align}
where $\Exp$ stands for the expectation, $\delta$ is the Kronecker symbol and $\bullet\tp$ denotes matrix or vector transpose. The magnitude of the noise is given by its quadratic variation, denoted here as $a$, and which is here given by the diagonal components of the covariance per unit of time:
\begin{equation}\label{eq:var0}
	a (x, t) \dt = Q (x, x, t, t).
\end{equation}
We remark that this  variance tensor has the same unit as the diffusion tensor ($\text{m}^2 \cdot \text{s}^{-1}$). The turbulent kinetic energy (TKE) is related to the  integral  over the whole volume of the variance tensor divided by the decorrelation time: $\alf \tr (a)/\tau$ .
As the covariance is a self-adjoint compact operator the noise can also be conveniently expressed, through its orthogonal eigenfunctions $\{  \Phi_n (\bullet, t) \}_{\sub n \in \mathbb{N}}$ weighted by the eigenvalues $\Lambda_n \geq 0$ with $\sum_{\sub n \in \mathbb{N}} \Lambda_n < \infty$, in terms of the spectral decomposition:
\begin{equation}\label{seq:KL}
	\sigma (x, t)\, \df B_t = \sum_{n \in \mathbb{N}} \Phi_n (x, t)\, {\df \beta_t^n},\ \quad a (x, t) = \sum_{n \in \mathbb{N}} \Phi_n (x, t) \Phi_n\tp (x, t),
\end{equation}
where $\beta^n$ denotes $n$ independent and identically distributed (i.i.d.) one-dimensional standard Brownian motions. The specification of those basis functions from data driven empirical covariance matrices enables us to construct specific noises, informed either by numerical or observational data. 

Let us note that the noise and the decomposition \eqref{eq:dX} have been here defined in terms of It\^{o} integrals. They can be defined by using  Stratonovich integrals (in place of It\^o integrals) with additional regularity assumptions, see \cite{Bauer2020jpo}. 

The LU dynamics models are then derived in similar way as in the deterministic setting through a stochastic version of the Reynolds transport theorem \cite{Bauer-2020-JPO,Memin2014, Resseguier-2017-GAFD-I} and the It\^{o}-Wentzell formula \cite{Kunita}.  To note, a modelling framework based also on stochastic transport has been  proposed in the geometric mechanics community \cite{CotterCrisanPan1, Holm-2015}. This framework is termed Stochastic Advection by Lie Transport (SALT). LU and SALT are derived from extended Newtonian principles and Hamiltonian derivation, respectively. They are the two faces of the same coin, exhibiting different conservation properties, namely energy preservation for the first and circulation
conservation, for the second.  \\

\noindent\textbf{Stochastic transport}\\
In a very similar manner as in the deterministic case, the stochastic flow dynamics in the LU framework is derived from a stochastic Reynolds transport theorem (SRTT) \cite{Memin2014} describing the rate of change of a random scalar $q$ transported by the stochastic flow \eqref{eq:dX} within a flow volume $\mathcal{V}$. For incompressible unresolved flows, $\div \sigma= 0$, the SRTT can be written as
\begin{subequations}
	\begin{align}
		&\df_t\, \Big( \int_{\mathcal{V}(t)} q (x, t)\, \df x \Big) = \int_{\mathcal{V}(t)} \big( \Df_t q + q \div (u - u_s) \big)\, \df \vec{x}, \label{eq:SRTT} \\
		&\Df_t q = \df_t q + (u - u_s) \adv q\, \dt + \sdbt \adv q - \alf \div (a \nabla q)\, \dt, \label{eq:STO} 
	\end{align}
\end{subequations}
where $\df_t q (x, t) = q (x, t + \dt)  - q (x, t)$ denotes the forward time-increment of $q$ at a fixed point $x$, $\Df_t$ is introduced as the stochastic transport operator in \cite{Resseguier-2017-GAFD-I} and $u_s = \alf \div a$ is referred to as the It\^{o}-Stokes drift (ISD) in \cite{Bauer-2020-JPO}. The transport operator plays the role of the material derivative in the stochastic setting. The ISD is defined by the variance tensor divergence and embodies the effect of statistical inhomogeneity of the unresolved flow on the large-scale component. As shown in  \cite{Bauer-2020-JPO}, it can be considered as a generalization of the Stokes drift associated to surface waves propagation  with the emergence of a similar vortex force and Coriolis correction. For a compressible noise term the ISD involves an additional term related to the noise divergence \cite{Resseguier-2017-GAFD-I}. In the definition of the stochastic transport operator in \eqref{eq:STO}, 
the last two terms describe, respectively, an energy backscattering from the unresolved scales to the large scales and an inhomogeneous diffusion of the large scales driven by the quadratic variation of the unresolved flow components. The diffusion term provides in this setting a matrix form of the  Boussinesq eddy viscosity assumption. This generalizes hence this assumption and provides to it a firm mathematical modelling ground in terms of the diffusion attached to time uncorrelated velocity components. It is not anymore defined by a loose analogy with the molecular dissipation mechanism. The backscattering term corresponds to an energy  source that is exactly compensated by the diffusion term.

For an isochoric flow with $\div (u - u_s) = 0$, one may immediately deduce from \eqref{eq:SRTT} the following transport equation of a scalar quantity:
\begin{equation}
	\Df_t q = 0,
\end{equation}
where the energy of such random scalar $q$ can be seen to be pathwise globally conserved through Itô integration by part and adequate boundary conditions:
\begin{equation}\label{eq:energy-balance}
	\df_t\, \Big( \int_{\sub\Omega} \alf q^2\, \df x \Big) = \Big( \underbrace{\alf\int_{\sub\Omega} q \div (a \grad q)\, \df x}_{\text{Energy loss by diffusion}} + \underbrace{\alf\int_{\sub\Omega} (\grad q)\tp a \grad q\, \df x}_{\text{Energy intake by noise}} \Big)\, \dt = 0.
\end{equation}
This energy conservation  for any realization can be interpreted in terms of an exact balance between the energy brought by the noise and the one dissipated by the diffusion term. Physically, the agitation generated by the small-scale component is globally exactly counterbalanced by the diffusion term. The noise quadratic variation plays a fundamental role in this balance expression.    \\

\noindent\textbf{Navier-Stokes equation under location uncertainty}\\
The SRTT and Newton's second principle formulated in a distributional sense \cite{Memin2014,Mikulevicius-Rozovskii-2004} enables us to derive the following (three-dimensional) stochastic Navier-Stokes equations in a rotating frame \cite{Bauer-2020-JPO}:
\begin{subequations}
	\begin{align}
		&{\text{\em Horizontal momentum equation}}:\nonumber\\
		&\Df_t u + f \times \big( u\, \dt + \sdbth \big) = - \frac{1}{\rho} \gradh \big( p\, \dt + \dpt \big) + \nu \laplac \big( u\, \dt + \sdbth\big), \label{eq:hmoment1}\\
		&{\text{\em Vertical momentum equation}}:\nonumber\\
		&\Df_t w = - \frac{1}{\rho} \partial_z \big( p\, \dt + \dpt \big) - g\, \dt + \nu \laplac \big( w\, \dt + \sdbtz \big), \label{eq:vmoment1}\\
		&{\text{\em  Mass equation}}:\nonumber\\
		&\Df_t \rho = 0, \label{eq:mass1} \\
		&{\text{\em Continuity equations}}:\nonumber\\
		&\divh \big( u - u_s \big) + \partial_z (w - w_s) = 0,\ \quad \divh \sdbth + \partial_z \sdbtz = 0. \label{eq:continu1} 
	\end{align}
\end{subequations}
In the horizontal momentum equation  $u= (u_x,u_y)\tp$ (resp. $\sdbth$) 
stands for the horizontal (2D) velocity component (resp. the unresolved random horizontal components); $\vec{f} = (2 \tilde{\Omega} \sin \Theta) k$ denotes the Coriolis parameter varying in latitude $\Theta$, with the Earth's angular rotation rate $\tilde{\Omega}$ and the vertical unit vector $k = [0, 0, 1]\tp$; $\rho$ is the fluid density; $\gradh = [\partial_x, \partial_y]\tp$ denotes the horizontal gradient; $p$ and $\dot{p}_t^{\sigma} = \dpt / \dt$ (informal definition) are the time-smooth and time-uncorrelated (martingale) components of the pressure field, respectively; $g$ is the Earth's gravity value and $\nu$ is the kinematic viscosity. In the following, the molecular friction term is assumed to be negligible and dropped from the equations. In the vertical momentum equation,  $w$ (resp. $\sdbtz$) is the vertical component of the three-dimensional large-scale flow (resp. the unresolved vertical random flow). The two continuity equations \eqref{eq:continu1} ensure volume conservation and mass conservation \eqref{eq:mass1}.\\

\noindent\textbf{Stochastic rotating shallow model under location uncertainty}\\
The shallow water model is usually obtained from scaling reasoning in terms of horizontal and vertical characteristic spatial scales with a ratio between them much smaller than one. Vertical integration to recover an  along-depth  mean 2D (compressible) velocity field and again scaling assumption (through Froude number  -- the ratio between the characteristic scale of current velocity and waves velocity  -- of one, corresponding to the so-called long wave approximation) allowing to neglect vertical acceleration and to get the predominance of the so called hydrostatic balance enable to derive the classical shallow water system of equations. 
Notably, the Shallow water system is an instance of a 2D compressible fluid dynamics model, driven by a divergent velocity arising from an along depth averaging of an incompressible 3D velocity field. As such it involves a compressible version of the continuity equation on the water surface height combined to a 2D momentum equation expressed through the balance between the momentum material derivative and 2D forces.
The full details of such derivation can be found in classical geophysical fluid dynamics textbook, e.g.  \cite{Vallis}. In the LU stochastic case, a derivation similar to the derivation of the shallow water equation can be followed. However, the amplitude of the small-scale random component must be appropriately scaled. For example, we cannot have a random component whose amplitude is much bigger than the velocity variable. This is in particular important for the vertical component where a wrong choice of the noise may destroy the hydrostatic balance.
As described in \cite{Brecht-2021-JAMES}, for noises with a small enough ratio between the vertical and horizontal components and Rossby number lower than one (i.e the effect of planetary motion  and hence the Coriolis correction is important and cannot be neglected) the LU stochastic representation of the shallow water equations (LU-RSW) for a flat topography reads:
 
\begin{subequations}\label{seq:RSWLU}
	\begin{align}
		&{(\text{\em Conservation of momentum})}\nonumber\\
		&\Df_t u + f \times u\, \dt = - g \grad h\, \dt, \label{eq:RSWLU-moment}\\
		&{(\text{\em Conservation of mass})}\nonumber\\
		&\Df_t h + h \div u\, \dt = 0, \label{eq:RSWLU-mass}\\
		&{(\text{\em  Random pressure / noise balance})}\nonumber\\
		&f \times \sdbt = - \frac{1}{\rho} \grad \dpt, \\
		&{(\text{\em  Incompressible constraints})}\nonumber\\
		&\div \sdbt = 0,\ \quad \div u_s = 0. \label{eq:RSWLU-incomp}
	\end{align}
\end{subequations}
In the above system the gradient operator involved has to be understood as a 2D horizontal gradient. The random balance provides a relation between the martingale pressure term and the  noise. It has the same form as the geostrophic balance but is related here to the time uncorrelated part of the pressure. It can be shown that this stochastic system conserves path-wise the global energy (see appendix A of \cite{Brecht-2021-JAMES}): 
\begin{equation} \label{eq:RSWLU-energy} 
	\df_t \int_{\Omega} \frac{\rho}{2} \big( h |\vec{u}|^2 + g h^2 \big)\, \df \vec{x} = 0.
\end{equation}
The LU-RSW shares thus exactly the same energy conservation property as the deterministic one. Beyond their formal resemblance this energy conservation property provides a strong physical link between the two systems. \\

 The objective of the paper is to present the mathematical analysis of this system. We answer the following questions: Is it well posed? What kind of solutions does it admits? Are they global in time or local? This is a set of prime classical questions that one can naturally ask about the new class of stochastic PDEs.
 
 Before describing further the type of mathematical results that have been obtained we need to make some important remarks. The LU-RSW defined above corresponds to a large-scale stochastic representation of a simplified ocean dynamics. Numerical implementation of this system requires to work on grid of finite resolution and to introduce sub-grid diffusion terms to prevent the pilling up of enstrophy at the resolution cutoff scale by cascading effects driven by transport equations. To counterbalanced this artefact, numerical dissipation needs to be artificially introduced. This will be done by viscosity terms, $\nu\Delta u$ and $\eta\Delta h$, in the momentum  and surface elevation equations, respectively. Within a LU interpretation, these two Laplacian diffusion terms can be thought as emanating from a subgrid unresolved isotropic Brownian velocity variables for which only the associated diffusion is represented. This corresponds indeed to the simplest subgrid eddy-diffusivity models. Such Laplacian diffusion can be extended to a general diffusion term of the form $\nabla\bcdot (a_{sg} \nabla h)$ together with a corresponding Ito-Stokes drift $u'_{s}= -1/2 \nabla\bcdot a_{sg}$ for non homogeneous unresolved (subgrid) noise.
 Nevertheless, in order to conserve mass,  periodic or specific boundary conditions need to be imposed on this unresolved noise term or on $h$: $\nabla\bcdot u'_{sg}=0$, $h u'_{sg}\bcdot n|_{\partial \Omega} =0$, $(a_{sg}\nabla h) \bcdot n |_{\partial \Omega} =0$, with $n$ the normal to the domain.
 
 In the following, we start by re-stating the system in an abstract form: \\

\noindent\textbf{A general Class of SPDEs} \\

\noindent We consider the following class of SPDEs
\begin{equation}\label{gencl}
\begin{aligned}
& du_t + \left(\nabla \cdot (u_t \otimes u_t) + \alpha u_t (\nabla \cdot u_t) - u_t^S\cdot \nabla u_t + f \times u_t + g\nabla h_t\right) dt + \sigma \cdot \nabla u_t dB_t = \nu \Delta u_tdt + \frac{1}{2} \nabla \cdot (a\nabla u_t)dt  \\
& dh_t + \left(\nabla \cdot (h_tu_t) + \beta h_t ( \nabla \cdot u_t) -u_t^S \cdot \nabla h_t \right) dt + \sigma \cdot \nabla h_t dB_t =  \eta \Delta h_tdt + \frac{1}{2}\nabla \cdot (a\nabla h_t)dt
\end{aligned}
\end{equation}
with $\alpha, \beta\in \mathbb{R}$ and, where  
\begin{itemize}
\item $u^S$ represents the It\^{o}-Stokes drift, $u^S:= \frac{1}{2} \nabla \cdot a$	
\item $\sigma$ is the correlation operator, with $\nabla \cdot \sigma dB_t = 0$
\item $a$ corresponds to the diagonal component of the covariance, $a = \sigma\sigma^T$. 
\end{itemize}
The stochastic viscous LU-SRSW model can by obtain by choosing $\alpha=-1$ and  $\beta=0$. The model can be written as follows: 
\begin{subequations}\label{eq:mainlu-noncompactform}
\begin{alignat}{2}
    & du + \mathcal{\widetilde{L}}_u u dt + fk\times udt + g\nabla h dt + \mathcal{L}_{\sigma}u dB_t = \nu \Delta udt + \frac{1}{2}\mathcal{L}_{\sigma}^2 u dt \\
    & dh + \mathcal{\widetilde{L}}_u h dt +\mathcal{L}_{\sigma}h dB_t = \eta \Delta hdt + \frac{1}{2}\mathcal{L}_{\sigma}^2 h dt.
\end{alignat}    
\end{subequations}
where
\begin{equation*}
    \footnotesize{\mathcal{L}_u u:= u \cdot \nabla u, \ \ \ \ \  \mathcal{\widetilde{L}}_u u := \mathcal{L}_u u - \mathcal{L}_{u^S} u, \ \ \ \ \ \mathcal{L}_{\sigma} u := \sigma \cdot \nabla u,}
\end{equation*}
\begin{equation*}
	\footnotesize{\mathcal{L}_{\sigma}^2 u := \nabla \cdot (a\nabla u)} 
\end{equation*}

In this paper, we prove that the system \eqref{gencl} admits local (in time) pathwise strong solutions continuous in $\mathcal{W}^{k,2}$-norm and square integrable in $\mathcal{W}^{k+1,2}$-norm. Moreover, we show the existence of a unique maximal solution.
In addition,
for the particular case $\alpha=\beta=-\frac{1}{2}$, we prove that
equation \eqref{gencl} has a global (in time) weak solution. 

In the following we will use the composite notation $X_t = (u_t, h_t)$. Then the LU-SRSW model \eqref{gencl} can be written in a more compact form: 
\begin{equation}\label{eq:mainlu}
dX_t + \mathcal{A}(X_t)dt + \mathcal{G}(X_t)dB_t = \frac{1}{2}\nabla \cdot (a \nabla X_t) + \gamma\Delta X_t dt
\end{equation}
where $\gamma =(\nu,\eta)$, $\delta = (\alpha,\beta)$ and
\begin{equation*}
\mathcal{A}\left( 
\begin{array}{c}
u \\ 
h%
\end{array}
\right) =\left( 
\begin{array}{c}
\nabla \cdot (u \otimes u) + \alpha u (\nabla \cdot u) - u^S\cdot \nabla u + f \times u + g\nabla h  \\ 
\nabla \cdot (hu) + \beta h ( \nabla \cdot u) -u^S \cdot \nabla h
\end{array}
\right)
\end{equation*}
and 
\begin{equation*}
\mathcal{G}\left( 
\begin{array}{c}
u \\ 
h%
\end{array}%
\right) =\left( 
\begin{array}{c}
\sigma \cdot \nabla u\\ 
\sigma \cdot \nabla h%
\end{array}%
\right) 
\end{equation*}%
are two \textit{deterministic} operators defined on $\mathcal{W}^{k,2}(\mathbb{T}^2)$ with values in $L^2(\mathbb{T}^2)$. Nonetheless, $\mathcal{A}(X_s)$ and $\mathcal{G}(X_s)$  need to be progressively measurable stochastic processes. Note that although it is enough to impose that $\mathcal{A}(X_s)$ is measurable in order for the integral $\int_0^t\mathcal{A}(X_s)ds$ to be well-defined, we need it to be also adapted in order to make sure that the stopped process $\mathcal{A}(X_{t\wedge\tau})$ is yet measurable. Therefore, $\mathcal{A}(X_s)$ is progressively measurable.
Note that the two diffusion coefficients $\nu$ and $\eta$ are not related to each other i.e. one of them can be very small  and the other one very big.
\begin{remark}\label{l1}
Note that if $a$ is constant, then the term $\frac{1}{2}\nabla \cdot (a\nabla X_t)$ generates a Laplacian. However, this operator does not add to the difussivity of the equation, as it is required to make the transfer from the Stratonovich noise to the It\^o  noise. This is the reason why it cannot be used in the proof of well-posedness, and why we still need to add a proper diffusion operator to the original LU-SRSW model. 
\end{remark}
\begin{remark}
Note also that one can write $(\mathcal{L}_u + \mathcal{D}_u)u = \nabla \cdot (u\otimes u)$. We prefer to use the notations introduced above though, as they better highlight the introduction of stochasticity in the advective term, which is essential for the physical properties of the LU models. Nonetheless, we sometimes use below also the second notation. 
\end{remark}
\begin{remark}
In the same spirit as the Remark \ref{l1} on the supplementary Laplacian diffusion in (\ref{eq:mainlu-noncompactform}a) and (\ref{eq:mainlu-noncompactform}b), the additional compressibility term $\alpha \mathcal{D}_u\Theta$,  (for $\Theta= u$ or $\Theta=h$) can also be interpreted in terms of an unresolved noise of a specific form. In that case we seek a noise covariation $a_\Theta$ such that:
\[
- \frac{1}{2} (\div a_\Theta)\bcdot \nabla \Theta - \frac{1}{2}\nab\bcdot(a_\Theta \nabla \Theta) = \alpha \Theta \nabla \bcdot u + \epsilon_\Theta \Delta \Theta.  
\]
We will see that imposing such form for the unresolved noise with $\alpha =1/2$ allows us to extend a local convergence result to a global one.  Let us note however, that despite an immediate mathematical interest, such model raises nevertheless immediate physical questions as without an additional global constraint, mass conservation cannot be ensured. This deficiency limits in practice long simulation run of the modified system.      
\end{remark}
\noindent\textbf{Literature review} \\
The deterministic nonlinear shallow water equations (also known as the \textit{Saint-Venant equations}) have been extensively studied in the literature. A significant difficulty in the well-posedness analysis of this model is generated by the interplay between its intrinsic nonlinearities, in the absence of any incompressibility conditions. In order to counterbalance the resulting chaotic effects, a viscous higher-order term is usually added to the inviscid system. Various shallow water models have been introduced in \cite{Zeitlinbook}, \cite{BreschsDesjardinsMetivier}. In \cite{LiuYin2} the authors show global existence and local well-posedness for the 2D viscous shallow-water system in the Sobolev space $H^{s-\alpha}(\mathbb{R}^2) \times H^s(\mathbb{R}^2)$ with $s>1$ and $\alpha \in [0,1)$. The methodology is based on Littlewood-Paley approximations and Bony paraproduct decompositions. This extends the result in \cite{LiuYin1} where local solutions for any initial data and global solutions for small initial data  have been obtained in $H^s \times H^s$ with $s>1$. A similar result adapted to Besov spaces was obtained in \cite{LiuYin0}. More recently, ill-posedness for the two-dimensional shallow water equations in critical Besov spaces has been shown in \cite{LiHongZhu}. Existence of global weak solutions and convergence to the strong solution of the viscous quasi-geostrophic equation, on the two-dimensional torus is shown in \cite{BreschsDesjardins1}. In \cite{BreschsDesjardins2} the authors construct a sequence of smooth approximate solutions for the shalow water model obtained in \cite{BreschsDesjardins1}. The approximated system is proven to be globally well-posed, with height bounded away from zero. Global existence of weak solutions is then obtained using the stability arguments from \cite{BreschsDesjardins1}. Sundbye in \cite{Sundbye} obtains global existence and uniqueness of strong solutions for the initial-boundary-value problem with Dirichlet boundary conditions and small forcing and initial data. In this work the solution is shown to be classical for a strictly positive time and a $C^0$ decay rate is provided. The proof is based on a priori energy estimates. Independently, Kloeden has shown in \cite{Kloeden} that the Dirichlet problem admits a global unique and spatially periodic classical solution. Both \cite{Sundbye} and \cite{Kloeden} are based on the energy method developed by Matsumura and Nishida in \cite{Matsumura}. Local existence and uniqueness of classical solutions for the Dirichlet problem associated with the non-rotating viscous shallow water model with initial conditions $(u_0,h_0) \in C^{2, \alpha} \times C^{1, \alpha}$ can be found in \cite{Bui}. The proof is based on the method of successive approximations and H\"{o}lder space estimates, in a Lagrangian framework. Existence and uniqueness of solutions for the two-dimensional viscous shallow water system under minimal regularity assumptions for the initial data and with height bounded away from zero was proven in \cite{ChenMiaoZhang}. 
The possibly stabilising effects of the rotation in the inviscid case is analysed in \cite{rotRSW1} and \cite{rotRSW2}. 
In \cite{CL3} the authors have shown existence and uniqueness of a local strong solution for a stochastic rotating shallow water model derived using the so-called Stochastic Advection by Lie Transport (SALT) approach. 

\subsection{Contributions of the paper}

Our paper presents some of the very first analytical results obtained for LU models. We obtain two qualitatively different sets of results.

\begin{itemize}
    \item We show that equation (\ref{eq:mainlu}) admits local (in time) strong solutions (both in analytical sense as well as in probabilistic sense). More precisely we show there exists a pair $(X,\tau)$, where $\tau$ is a strictly positive stopping time and $X(\cdot\wedge\tau)$ is a predictable process such that, for $T>0$,
\begin{equation*}
   X(\cdot\wedge\tau)\in L^2\left( \Omega, C\left([0,T];\mathcal{W}^{k,2}(\mathbb{T}^2)^3\right) \right)
\end{equation*}
\begin{equation*}
   X\mathbbm{1}_{t\leq\tau} \in L^2\left( \Omega, L^2\left(0, T;\mathcal{W}^{k+1,2}(\mathbb{T}^2)^3\right)\right). 
\end{equation*}

Moreover, we show the existence of a unique \emph{maximal} solution $(X, \mathcal{T})$, in the sense of Definition \ref{solutionssrsw}c, see Section \ref{maximalsolution}. 

As opposed to \cite{CL3}, in this paper we construct the approximation sequence using Littlewood-Paley mollifiers. This is particularly suitable when looking at the ``energy dynamics'', the Littlewood-Paley decomposition allowing for a mathematically rigorous characterisation of the concept of \textit{turbulence} from physics. 
More precisely, the energy flux through a 3D sphere of radius $N=2^j$, in the Fourier (frequency) space, for an incompressible fluid, is given by
\begin{equation*}
    \Gamma_N(u):= \displaystyle\int_{\mathbb{T}^3} \mathcal{J}_N (u \otimes u) \cdot \nabla \mathcal{J}_N u dx
\end{equation*}
where $u$ is the velocity of the fluid and $\mathcal{J}_N u$ is the projection of $u$ onto its Fourier modes which have modulus smaller than $N$ (see the Appendix for details). In \cite{Friedlander} it has been shown explicitly that using this decomposition, the energy flux through the sphere of radius $N$ is controlled mainly by frequencies of order $N$:
\begin{equation*}
    |\Gamma_{N=2^j}(u)| \leq C\displaystyle\sum_{i=1}^{\infty}2^{-2/3|j-i|}2^i\|\tilde{\mathcal{J}}_{2^i} u\|_{L^3}^3
\end{equation*}
where $\tilde{\mathcal{J}}_{2^i}:= \mathcal{J}_{2^{i+1}} - \mathcal{J}_{2^i}$ and $\tilde{\mathcal{J}}_{2^i}$ corresponds to the $i$-th Littlewood-Paley piece of $u$. As highlighted in \cite{Friedlander}, this shows that the inter-scale energy transfer is controlled primarily by local interactions. 
This is more of interest when we do not have the operator $\mathcal{D}_u$ in the model (and the fluid is still compressible), since in this case the cancellations above do not hold automatically and one needs to pay extra attention to the way energy dissipates at different scales. In our case though, for the approximating sequence, the flux of energy is automatically zero at each dyadic scale.

Using the Littlewood-Paley mollifiers,  we succeed in showing directly that the sequence of approximating solutions is Cauchy when $k$ is large enough ($k\geq 2$), therefore there is no need to use a tightness argument. The reason is that we now have enough regularity in order for the truncation function to converge properly. Nonetheless, for $k<2$ we need tightness arguments as in this case the approximation does not contain a truncation and this does not allow us to prove the Cauchy property straightaway. Note that our approach for constructing a strong solution is different from the one used in \cite{Ziane} or \cite{VicolHoltz} in the sense that although the solution belongs to a higher order Sobolev space, it suffices to prove the Cauchy property in $L^2$. The advantage is that we do not need to first show the existence of a smoother solution using some extra tightness arguments (as in \cite{VicolHoltz} for instance). Also, the calculations are simpler in $L^2$. Nonetheless, the uniform estimates corresponding to the approximating solution are proven in the higher order Sobolev space of existence. As compared to \cite{Ziane}, we use a truncation function instead of working with stopping times only. This keeps the proof shorter overall, as we do not need to show separately an almost sure convergence of the corresponding stopping times.

\item For the particular case $\alpha=\beta=-1/2$, we show that the system
equation (\ref{eq:mainlu}) admits a global analytically weak but probabilistically strong solution $X$ (see Definition \ref{solutionssrsw}d) which belongs to the space
\begin{equation*}
   X(\cdot\wedge\tau)\in L^2\left( \Omega, C\left([0,T];L^{2}(\mathbb{T}^2)^3\right) \right)
\end{equation*}
\begin{equation*}
   X\mathbbm{1}_{t\leq\tau} \in L^2\left( \Omega, L^2\left(0, T;\mathcal{W}^{1,2}(\mathbb{T}^2)^3\right)\right)\ \ \ \hbox{for} \ \ T > 0.  
\end{equation*}

The global character of the solution is due to a couple of term-cancellations which are specific for this model and which are directly linked with the conservation of energy in $L^2(\mathbb{T}^2)$. More precisely, we formally have that 

\begin{equation}\label{cancel1}
\langle X_t, u_t \cdot \nabla X_t\rangle + \frac{1}{2}\langle X_t, (\nabla \cdot u_t)X_t\rangle= 0.
\end{equation}
and
\begin{equation}\label{cancel2}
   \langle X_t, u_t^{S} \cdot \nabla X_t \rangle + \langle X_t, \mathcal{L}_{\sigma}^2(X_t)\rangle + \langle \mathcal{G}(X_t), \mathcal{G}(X_t)\rangle = 0.
\end{equation}
Note that in the proofs below we use these cancellation properties corresponding to the approximating solution $X^N$, instead of $X$. This is due to the fact that in the original equation we do not have enough regularity to apply the It\^{o} formula directly, so we need to work with $X^N$ first and then $X$ ``inherits'' the required properties. The cancellations \eqref{cancel1} and \eqref{cancel2} do not hold in higher order Sobolev spaces $\mathcal{W}^{k,2}(\mathbb{T}^2)$ with $k\geq 1$ and therefore the solution is only local (in the sense of Definition \ref{solutionssrsw}a). 

These properties are essential to ensure the existing of global solution (albeit weak), as we work with a compressible system, and it is known that compressibility is a strong assuption which usually creates many technical difficulties. However, through this special structure, the LU-SRSW model retains the important features of a compressible fluid, while still allowing for a clean mathematical analysis.

We are able to prove strong continuity in time of the solution due to the fact that we work in a two-dimensional space. The key ingredient is Ladyzhenskaya's inequality, which in 3D has a slightly different form and therefore the solution would live in $C_w([0,T];L^2(\mathbb{T}^2))$ only in that case. In 2D, weak continuity in time would not be sufficient for proving the global bound \eqref{originalestim} from \eqref{estimapprox1} for the reason that we could not pass to the limit in the strong topology of $C([0,T];L^2(\mathbb{T}^2))$ without imposing further conditions on the solution (i.e. such that suitable interpolation results can be used). 

\end{itemize}

\noindent \textbf{Structure of the paper.} In Section \ref{preliminaries} we introduce some notations, definitions of solutions, and preliminary results. In Section \ref{mainresults} we introduce the main results and we prove the existence of a maximal solution. In section \ref{sect:globalweak} we prove the existence of a global weak solution for a particular system which belongs to the class \ref{gencl}. In Section \ref{uniqueness} we show that the local strong solution is pathwise unique. In Section \ref{sect:localstrong} the existence of a local strong solution is proven. In the Appendix we give details about the Littlewood-Paley projection, the a priori estimates, and the Cauchy property. 
\section{Preliminaries and notations}\label{preliminaries}
\begin{itemize}
\item We denote by $\mathbb{T}^{2}=\mathbb{R}^{2}/\mathbb{Z}^{2}$ the two-dimensional
torus. Let $H$ be a Hilbert space on $\mathbb{T}^2$, endowed with an orthonormal basis $(e_{j})_{j\in\mathbb{N}}$, and denote its dual by $H^{\star}$. In our case $H=L^2(\mathbb{T}^2)$ for the construction of the weak solution and $H=\mathcal{W}^{k,2}$ with $k\geq 1$ for the construction of the strong solution. 
\item Let $(\Omega, \mathcal{F}, (\mathcal{F}_t)_{t\geq 0}, \mathbb{P})$ be a filtered probability space and $(B^i)_{i \in \mathbb{N}}$ a sequence of independent one-dimensional Brownian motions adapted to the complete and right-continuous filtration $(\mathcal{F}_t)_{t\geq 0}$.  That is, $\mathcal{S}: = (\Omega, \mathcal{F}, (\mathcal{F}_t)_{t\geq 0}, \mathbb{P}, (B^i)_{i\in\mathbb{N}})$ is a fixed \textit{stochastic basis}.

\item $C([0,T]; H)$ is the space of continuous functions $f:[0,T]\rightarrow H$ that is $f(t_j)\rightarrow f(t_0)$ in $H$ for any sequence $(t_j)_j \subset [0,T]$ such that $t_j \rightarrow t_0$.
\item $C_w([0,T]; H)$ is the space of weakly continuous functions $f:[0,T]\rightarrow H$ that is $f(t_j)\rightharpoonup f(t_0)$ in $H$  for any sequence $(t_j)_j \subset [0,T]$ such that $t_j \rightarrow t_0$, where $\rightharpoonup$ denotes the convergence with respect to the weak norm topology. This is equivalent to saying that the functions $t\rightarrow \langle \varphi, f(t)\rangle_{H^{\star}\times H}$ belong to the space $C([0,T])$ for any $\varphi \in H^{\star}$. A sequence $(f_n)_n$ converges to $f$ in $C_w([0,T]; H)$ if $\displaystyle\sup_{t\in[0,T]} |\langle \varphi, f_n \rangle_{H^{\star} \times H} - \langle \varphi, f \rangle_{H^{\star} \times H}| \rightarrow 0 $ for all $\varphi \in H^{\star}$.

\item Let  $\beta \in (0,1), p\in [2,\infty)$ and let $H$ be a Hilbert space. Then the fractional Sobolev space $\mathcal{W}^{\beta,p}(0,T; H)$ is endowed with the norm 
\begin{equation*}
\|f\|_{\mathcal{W}^{\beta, p}(0, T ; H)}^p:=\int_{0}^{T}\left\|f_{t}\right\|_{H}^{p} d t+\int_{0}^{T} \int_{0}^{T} \frac{\left\|f_{t}-f_{s}\right\|_{H}^{p}}{|t-s|^{1+\beta p}} d t d s.
\end{equation*} 
In our case $H=\mathcal{W}^{-1,2}(\mathbb{T}^2)$ for the weak solution and $H=L^2(\mathbb{T}^2)$ for the strong solution.
\item Let $f\in C\left([0,T];\mathcal{W}^{k,2}(\mathbb{T}^2)\right) \cap L^2\left(0,T; \mathcal{W}^{k+1,2}(\mathbb{T}^2)\right)$ with $k \geq 0$. We introduce the the space $\mathcal{S}_{T,p,(k,2)}$ of such functions, endowed with the norm
\begin{equation*}
    \|f\|_{\mathcal{S}_{T,p,(k,2)}} := \displaystyle\sup_{t\in[0,T]}\|f_t\|_{k,2}^p + \left(\displaystyle\int_0^T \|f_t\|_{k+1,2}^2 dt\right)^{p/2} \ \ \ \hbox{for any} \ \ p > 0. 
\end{equation*}
\item $z\times u := (-u^2, u^1)$ for a two-dimensional vector $u=(u^1,u^2)$ and a unit vector $z$. 
\item $C$ is a generic constant and can differ from line to line.
\end{itemize}

\begin{definition}\label{solutionssrsw} 
Let $\mathcal{S} = \left(\Omega, \mathcal{F}, (\mathcal{F}_t), \mathbb{P}, (B_t)_t\right)$ be a fixed stochastic basis and $k\geq 0$. 
\begin{itemize}
	\item[a.] A pathwise \underline{local solution} of the LU-SRSW system is given
by a pair $(X,\tau )$ where $\tau :\Omega \rightarrow \lbrack 0,\infty ]$ is
a strictly positive bounded stopping time and $X:\Omega \times
\lbrack 0,\infty )\times \mathbb{T}^{2}\rightarrow \mathbb{R}^{3}$ is such that $%
X_{\cdot\wedge \tau }$ is $\mathcal{F}%
_{t}$-adapted for any $t\geq 0$, with
\begin{equation*}
   X_{\cdot\wedge\tau}\in L^2\left( \Omega, C\left([0,T];\mathcal{W}^{k,2}(\mathbb{T}^2)^3\right) \right)
\end{equation*}
\begin{equation*}
   X\mathbbm{1}_{t\leq\tau} \in L^2\left( \Omega, L^2\left(0, T;\mathcal{W}^{k+1,2}(\mathbb{T}^2)^3\right)\right) 
\end{equation*}
and the LU-SRSW system is satisfied locally i.e.%
\begin{equation*}
X_{t\wedge \tau }=X_{0}+\int_{0}^{_{t\wedge \tau }}\mathcal{A}(X_{s})
ds+\int_{0}^{_{t\wedge \tau }}\mathcal{G}(X_{s}) dB_{s}+\frac{1}{2}\displaystyle\int_0^{t\wedge\tau}\nabla \cdot (a \nabla X_s)ds+\gamma \int_{0}^{_{t\wedge \tau }}\Delta X_{s}ds
\end{equation*}%
holds $\mathbb{P}$-almost surely, as an identity in $L^{2}(\mathbb{T}^{2})^{3}.$

\item[b.] If $\tau =\infty $ then the solution is called 
\underline{global}.

\item[c.] A pathwise \underline{maximal solution} of the SRSW system is given by a pair $(X, \mathcal{T})$ where $\mathcal{T}: \Omega \rightarrow [0, \infty]$ is a non-negative stopping time and   
$X: \Omega \times [0, \mathcal{T}) \times \mathbb{T}^2 \rightarrow \mathbb{R}^3$ is a process for which there exists an increasing sequence of stopping times $(\tau^n)_n$ with the following properties: 

\begin{itemize}
\item[i.] $\mathcal{T}=\lim_{n\rightarrow \infty }\tau ^{n}$ and $\mathbb{P}(%
\mathcal{T}>0)=1$
\item[ii.] $(X,\tau ^{n})$ is a pathwise local solution of the LU-SRSW system
for every $n\in \mathbb{N}$
\item[iii.] if $\mathcal{T}<\infty $ then 
\begin{equation*}
\displaystyle\limsup_{t \rightarrow \mathcal{T}} \| X_{t}\|_{k,2}=\infty .
\end{equation*}
\end{itemize}

\item[d.] A \underline{global weak (in PDE sense) solution} of the LU-SRSW system is given by an $(\mathcal{F}_t)_t$-adapted process $X:\Omega \times [0,\infty) \times \mathbb{T}^2 \rightarrow \mathbb{R}^3$ which belongs to the space
\begin{equation*}
    L^2\left( \Omega, C\left([0,T];L^2(\mathbb{T}^2)^3 \right)\right) \bigcap L^2\left( \Omega, L^2\left(0,T;\mathcal{W}^{1,2}(\mathbb{T}^2)^3 \right)\right)
\end{equation*}
and such that the following identity holds for any test function $\varphi \in \mathcal{W}^{1,2}(\mathbb{T}^2)$: 
\begin{equation*}
    \langle X_t, \varphi\rangle = \langle X_0,\varphi\rangle + \displaystyle\int_0^t \langle X_s,\mathcal{A}^{\star}\varphi\rangle ds + \displaystyle\int_0^t \langle X_s,\mathcal{G}^{\star}\varphi \rangle dB_s - \frac{1}{2}\displaystyle\int_0^t \langle  a\nabla X_s, \nabla \varphi \rangle ds - \gamma\displaystyle\int_0^t \langle \nabla X_s, \nabla \varphi\rangle ds
\end{equation*}
where $\langle ,\rangle $denotes the inner product in $L^2(\mathbb{T}^2)^3$ and $\mathcal{A}^{\star}, \mathcal{G}^{\star}$ are the dual operators corresponding to $\mathcal{A}$ and $\mathcal{G}$ respectively. 
\end{itemize}
\end{definition}
\vspace{2mm}
\noindent\textbf{Assumptions on the stochastic fields:} We impose the following condition on the correlation operator $\sigma$:
\begin{equation*}
    \|\sigma\|_{k+1,\infty} < \infty.
\end{equation*}
This condition ensures that for any $f\in\mathcal{W}^{2,2}(\mathbb{T}^2)$ there exist some constants $C$ such that
\begin{equation*}
\|\sigma \cdot \nabla f\|_2^2 \leq C\|f\|_{1,2}^2    
\end{equation*}
\begin{equation*}
    \|\nabla \cdot (a\nabla f)\|_2^2 \leq C \|f\|_{2,2}^2
\end{equation*}
where $a = \sigma \sigma^{T}$.
\section{Main results}\label{mainresults}

\noindent In order to show that there exists a solution for the rotating shallow water model \eqref{eq:mainlu} we introduce the following (smooth) truncation function
$f_R: \mathbb{R}_{+} \rightarrow [0,1]$ equal to $1$ on $[0,R]$,
equal to $0$ on $[R+1, \infty)$, and decreasing on $[R, R+1]$, with 
$f_{R}(X_{t}):=f_{R}\left( \|X_t\|_{k,2}\right)$,
where 
$\|X_t\|_{k,2}:=\|u_t\|_{k,2}+\|h_{t}\|_{k,2}$
for arbitrary $R>0$ and $k\geq 0$. 
When $k=0$ the truncation is not needed due to the cancellation properties described above, and for this reason we consider that in this case $f_R \equiv 1$.
When $k > 0$, the truncation function enables us to construct a global strong solution for a truncated version of the original model \eqref{eq:mainlu}. Then we show that the truncation can be lifted up to a positive stopping time. Moreover, the solution is shown to be maximal.
The following two theorems are the main results of the paper. 
\begin{theorem}
\label{mainthmlu} 
Let $\mathcal{S} = (\Omega, \mathcal{F}, (\mathcal{F}_t)_t, 
\mathbb{P}, (B_t)_t)$ be a fixed stochastic basis and $X_0\in \mathcal{W}^{k,2}(\mathbb{T}^2)^3$. Then the stochastic rotating shallow water system %
\eqref{eq:mainlu} admits a unique pathwise maximal
strong solution $(X,\mathcal{T})$ such that, for $k\geq 1$,
\begin{equation*}
   X_{\cdot\wedge\tau_n}\in L^2\left( \Omega, C\left([0,T];\mathcal{W}^{k,2}(\mathbb{T}^2)^3\right) \right)
\end{equation*}
\begin{equation*}
   X\mathbbm{1}_{t\leq\tau_n} \in L^2\left( \Omega, L^2\left(0, T;\mathcal{W}^{k+1,2}(\mathbb{T}^2)^3\right)\right) 
\end{equation*}
for any stopping time $\tau_n$ such that $(X, \tau_n)$ is a local solution and $\displaystyle\lim_{n\rightarrow \infty} \tau_n = \mathcal{T}$.
\end{theorem} 

\begin{theorem}\label{thmlu2}
Let $\mathcal{S} = (\Omega, \mathcal{F}, (\mathcal{F}_t)_t, 
\mathbb{P}, (B_t)_t)$ be a fixed stochastic basis and $X_0\in L^2(\mathbb{T}^2)^3$. Let $\alpha=\beta=-\frac{1}{2}$ in \eqref{gencl}. Then the LU stochastic rotating shallow water system \eqref{eq:mainlu} admits an $L^2$-valued continuous global weak solution such that
\begin{equation}\label{originalestim}
    \displaystyle\sup_{s\leq T}\|X_s\|_2^2 + \displaystyle\int_0^T \|X_s\|_{1,2}^2 ds \leq \|X_0\|_2^2 \ \ \ \ \ \ \mathbb{P}-a.s.
\end{equation}
\end{theorem}

\section{Proof of Theorem \ref{mainthmlu} - Existence of a maximal strong solution}

In this section we prove Theorem \ref{mainthmlu} that is we show that the system \eqref{gencl} admits a maximal strong solution. We first construct a local solution using a truncated Littlewood-Paley approximation and we then show that the resulting solution is actually maximal. 

\subsection{Construction of a local strong solution}\label{sect:localstrong}

We recall in the Appendix the basic properties of the Littlewood-Paley projections.
Here $X^{N,R}=(u^{N,R}, h^{N,R})$ where $u^{N,R}$ and $h^{N,R}$ are solutions of the truncated approximating system, that is a truncated version of \eqref{gencl}. 
\begin{equation}\label{truncatedu}
      \begin{aligned}
    & du^{N,R}  +  \mathcal{J}_N(fk \times  \mathcal{J}_Nu^{N,R} )dt +  f_R(X^{N,R})\left(\mathcal{J}_{N}(\mathcal{{L}}_{\mathcal{J}_{N}u^{N,R}} + (1+\alpha)\mathcal{D}_{\mathcal{J}_Nu^{N,R}})\mathcal{J}_{N}u^{N,R}  \right)dt\\
      & +\mathcal{J}_{N}(\mathcal{{L}}_{\mathcal{J}_{N}u^{S,N,R}} \mathcal{J}_{N}u^{N,R}) + g\mathcal{J}_{N}(\nabla \mathcal{J}_{N}h^{N,R}) dt + \mathcal{J}_{N}\left(\mathcal{L}_{\sigma}\mathcal{J}_{N}u^{N,R}\right)dB_t  \\
      & = \nu \Delta u^{N,R} dt + \frac{1}{2}\mathcal{J}_{N}\left(\mathcal{L}_{\sigma}^2\mathcal{J}_{N}u^{N,R}\right)dt 
      \end{aligned}  
\end{equation}
\begin{equation}\label{truncatedh}   
\begin{aligned}
 dh^{N,R} +  & f_R(X^{N,R})\left( \mathcal{J}_{N} (\mathcal{{L}}_{\mathcal{J}_{N} u^{N,R}} + (1+\beta)\mathcal{D}_{\mathcal{J}_{N}u^{N,R}})\mathcal{J}_{N}h^{N,R} ) \right)dt\\
    & + \mathcal{J}_{N}(\mathcal{{L}}_{\mathcal{J}_{N}u^{S,N,R}} \mathcal{J}_{N}h^{N,R})+ \mathcal{J}_{N}\left(\mathcal{L}_{\sigma}\mathcal{J}_{N}h^{N,R}\right)dB_t \\
    & = \eta \Delta h^{N,R}dt + \frac{1}{2}\mathcal{J}_{N}\left(\mathcal{L}_{\sigma}^2\mathcal{J}_{N}h^{N,R}\right)dt.
   \end{aligned}
\end{equation}
\begin{remark}
Note that in the above equations,  the set of Brownian motions driving each element of the sequence $X^{N,R}$ are the same, hence the absence of the superscript $N$ in their notation. This helps the analysis of the Cauchy property. If we choose a different set of Brownian motions for each element of the sequence, then we will need to assume their convergence in a norm that will imply the convergence of the stochastic integrals. 
\end{remark}
This approximating system is well-posed by a straightforward extension of Theorem 1 Chapter 3 in \cite{Skorokhod} to the case when the systems is driven by a countable set of Brownian motions.
The truncation function $f_R: \mathbb{R}_{+} \rightarrow [0,1]$ is chosen  to be $1$ on $[0,R]$,  $0$ on $[R+1, \infty)$, and decreasing on $[R, R+1]$, 
$f_{R}(X_{t}):=f_{R}\left( \|X_t\|_{k,2}\right)$
where 
$\|X_t\|_{k,2}:=\|u_t\|_{k,2}+\|h_{t}\|_{k,2}$
for arbitrary $R>0$ and $k\geq 1$. Similarly, $f_R(u_t):=f_R(\|u_t\|_{k,2})$ and $f_R(h_t):=f_R(\|h_t\|_{k,2})$.  
To shorten the formulae, we will use the following abbreviated notation 
\begin{equation*}
    X^{N,R}:= \check{X}^N, \ \ \ u^{N,R}:= \check{u}^N, \ \ \ h^{N,R}:= \check{h}^N,
\end{equation*}
and then the approximating truncated system can be rewritten in the following more compact form: 
\begin{equation}\label{compactapproxstrong}
    d\check{X}_t^{N} + \check{\mathcal{A}}^{N}(\check{X}_t^{N})dt + \check{\mathcal{G}}^{N}(\check{X}_t^{N})dB_t = \gamma\Delta \check{X}_t^{N}
\end{equation}
where 
\begin{equation*}
\begin{aligned}
&\check{\mathcal{A}}^{N}(\check{X}^{N}) := 
&\left( 
\begin{array}{c}
 (\check{\mathcal{A}}^{N}(\check{u}^N))^1\\ 
(\check{\mathcal{A}}^{N}(\check{h}^N))^2
\end{array}
\right)
\end{aligned}
\end{equation*}
with
\begin{equation*}
\begin{aligned}
    (\check{\mathcal{A}}^{N}(\check{u}^N))^1 &:= \mathcal{J}_N(fk\times \mathcal{J}_Nu^{N,R})+ f_R(u^{N,R})\left(\mathcal{J}_N (\mathcal{{L}}_{\mathcal{J}_Nu^{N,R}}  + (1+\alpha)\mathcal{D}_{\mathcal{J}_Nu^{N,R}} )\mathcal{J}_Nu^{N,R}\right)\\
    &+ \mathcal{J}_N(\mathcal{{L}}_{\mathcal{J}_Nu^{S,N,R}}\mathcal{J}_Nu^{N,R})+ g\mathcal{J}_N\nabla (\mathcal{J}_Nh^{N,R}) -\frac{1}{2} \mathcal{J}_{N}\left(\mathcal{L}_{\sigma}^2\mathcal{J}_{N}u^{N,R}\right) 
\end{aligned}
\end{equation*}
\begin{equation*}
    \begin{aligned}
    (\check{\mathcal{A}}^{N}(\check{h}^N))^2 & := f_R(h^{N,R})\left(\mathcal{J}_N(\mathcal{{L}}_{\mathcal{J}_Nu^{N,R}}  + (1+\beta)\mathcal{D}_{\mathcal{J}_Nu^{N,R}})\mathcal{J}_Nh^{N,R} \right)\\
    & + \mathcal{J}_N(\mathcal{{L}}_{\mathcal{J}_Nu^{S,N,R}}\mathcal{J}_Nh^{N,R})-\frac{1}{2} \mathcal{J}_{N}\left(\mathcal{L}_{\sigma}^2\mathcal{J}_{N}h^{N,R}\right)
\end{aligned}
\end{equation*}
and 
\begin{equation*}
\check{\mathcal{G}}^{N}(\ckX^{N})
:=\left( 
\begin{array}{c}
\mathcal{J}_N(\cL_{\sigma} \mathcal{J}_Nu^{N,R}) \\ 
\mathcal{J}_N(\cL_{\sigma} \mathcal{J}_Nh^{N,R})%
\end{array}%
\right). 
\end{equation*}%
We show below that the approximating truncated system \eqref{compactapproxstrong}
converges to a truncated form of the original system \eqref{eq:mainlu}
which is given by
\begin{equation}\label{eq:truncatedlu}
    d\check{X}_t + \check{\mathcal{A}}(\check{X}_t)dt + \check{\mathcal{G}}(\check{X}_t)dB_t = \gamma\Delta \check{X}_t
\end{equation}
where 
\begin{equation*}
\begin{aligned}
&\check{\mathcal{A}}(\check{X}) := 
&\left( 
\begin{array}{c}
 (\check{\mathcal{A}}(\check{u}))^1\\ 
(\check{\mathcal{A}}(\check{h}))^2
\end{array}
\right)
\end{aligned}
\end{equation*}
with
\begin{equation*}
\begin{aligned}
    (\check{\mathcal{A}}(\check{u}))^1 &:= fk\times u^{R}+ f_R(u^{R})\left((\mathcal{{L}}_{u^{R}}  + (1+\alpha)\mathcal{D}_{u^{R}} )u^{R}\right) + (\mathcal{{L}}_{u^{S,R}}u^{R})+ g\nabla (h^{R}) -\frac{1}{2} \mathcal{L}_{\sigma}^2u^{R}
\end{aligned}
\end{equation*}
\begin{equation*}
    \begin{aligned}
    (\check{\mathcal{A}}(\check{h}))^2 & := f_R(h^{R})\left((\mathcal{{L}}_{u^{R}}  + (1+\beta)\mathcal{D}_{u^{R}})h^{R} \right) + (\mathcal{{L}}_{u^{S,R}}h^{R})-\frac{1}{2}\mathcal{L}_{\sigma}^2h^{R}
\end{aligned}
\end{equation*}
and 
\begin{equation*}
\check{\mathcal{G}}(\ckX)
:=\left( 
\begin{array}{c}
\cL_{\sigma} u^{R} \\ 
\cL_{\sigma} h^{R}
\end{array}%
\right). 
\end{equation*}%


\begin{theorem}\label{mainthm:strongsln}
Assume that for any $T>0$
\begin{equation}\label{cauchyrel}
\displaystyle\lim_{N\rightarrow \infty}\displaystyle\sup_{M\geq N}\mathbb{E}\left[\displaystyle\sup_{t\in [0, T]}\|\check{X}_t^N - \check{X}_t^M\|_2^p\right] +\mathbb{E}\left[\left(\int_{0}^{T}\|\check{X}_s^N - \check{X}_s^M\|_{1,2}^2ds\right)^{p/2}\right]= 0
\end{equation}
and there exists a constant $C=C(R)$ independent of $N$ such that 
\begin{equation}\label{approxslnestim}
    \displaystyle\sup_{N \geq 1}\mathbb{E}\left[ \displaystyle\sup_{t\in[0,T]}\|\check{X}_t^N\|_{k,2}^p + \left(\displaystyle\int_0^t \|\check{X}_s^N\|_{k+1,2}^2ds\right)^{p/2}\right ] \leq C
\end{equation}
for any $p>0$ and $k\geq 1$.
Then there exists a predictable and progressively measurable process $\check{X}$ with paths in the space  
$$C([0,T]; \cW^{k,2}(\mathbb{T}^2)) \bigcap L^2(0,T; \mathcal{W}^{k+1,2}(\mathbb{T}^2)) $$
satisfying 
\begin{equation}\label{origslnestim}
  \mathbb{E}\left[ \displaystyle\sup_{t\in[0,T]}\|\check{X}_t\|_{k,2}^p + \left(\displaystyle\int_0^t \|\check{X}_s\|_{k+1,2}^2ds\right)^{p/2}\right ]  \leq C.
\end{equation}
such that 
\begin{equation}\label{ckhXlim}
    \displaystyle\lim_{N\rightarrow\infty}\mathbb{E}\left[\displaystyle \sup_{t\in [0,T]}\|\ckX^N _t- \ckX_t\|_{2}^p\right] +
\displaystyle\lim_{N\rightarrow\infty}\mathbb{E}\left[\left(\displaystyle \int_0^T\|\ckX^N_t- \ckX_t\|_{1,2}^2dt\right)^{p/2}\right] = 0.
\end{equation}
\end{theorem}

\noindent The complete proofs for \eqref{cauchyrel} and \eqref{approxslnestim} are provided in Appendix. 

\begin{proof}[\textbf{Proof of Theorem \ref{mainthm:strongsln}}.]
The existence of the limiting process $\ckX$ follows immediately from Cauchy property \eqref{cauchyrel}. The bound 
(\ref{origslnestim}) follows from Fatou's lemma. 
From (\ref{cauchyrel}), (\ref{approxslnestim}), (\ref{origslnestim}) we can deduce, by using a standard interpolation argument, that for any $\epsilon\in (0,1)$, we have that
\begin{equation}\label{hc}
    \displaystyle\lim_{N\rightarrow\infty}\mathbb{E}\left[\displaystyle\sup_{t\in[0,T]}\|\ckX_t^N-\ckX_t\|_{k-\epsilon,2}^p\right] =0 
\end{equation}
that is the convergence of $(\ckX^N)_N$ holds in $\mathcal{W}^{k-\epsilon,2}(\mathbb{T}^2)$. We use this result to pass to the limit in the approximating equation. The convergence of all the linear term is immediate so we oly show teh convergence of the nonlinear terms. Explicitly, $\ckX^N=(\cku^N,\ckh^N)$ satisfies the system given by \eqref{truncatedu}-\eqref{truncatedh}. The convergence of the linear terms is straightforward, hence we write in more detail the convergence of the nonlinear terms. We want to show that
    \begin{equation}\label{conv}
        \begin{aligned}
            \displaystyle\int_0^t\langle f_R(\ckX^{N})\left(\mathcal{J}_{N}(\mathcal{{L}}_{\mathcal{J}_{N}\cku^{N}} + (1+\delta)\mathcal{D}_{\mathcal{J}_N\cku^{N}})\mathcal{J}_{N}\ckX^{N}\right) - f_R(\ckX) (\mathcal{{L}}_{\cku} + (1+\delta)\mathcal{D}_{\cku})\ckX, \varphi \rangle ds \xrightarrow{N\rightarrow\infty} 0
        \end{aligned}
    \end{equation}
in $L^2(\mathbb{P})$. For the advective term we will use that\footnote{Note that we can write the approximating term without the projection operator $\mathcal{J}_N$ in front of it}
\begin{equation}\label{decomp}
    \begin{aligned}
        |\langle f_R(\ckX^{N})\mathcal{{L}}_{\cku^{N}} \ckX^{N} - f_R(\ckX) \mathcal{{L}}_{\cku} \ckX, \varphi \rangle|& \leq f_R(\ckX^N)|\langle (\cku^N-\cku)\cdot \nabla\ckX^N,\varphi\rangle| \\
        & + f_R(X)|\langle \cku \cdot \nabla(\ckX^N-\ckX),\varphi\rangle| \\
        & + |f_R(\ckX^N)-f_R(\ckX)||\langle \cku^N \cdot \nabla \ckX^N, \varphi\rangle|.
    \end{aligned}
\end{equation}
and show the convergence of each of the three terms on the right hand side of the above inequality. The arguments are identical to that contained in  \cite{CL3} pp. 14-15, where we use that fact that we have also convergence in $\mathcal{W}^{k-\epsilon,2}(\mathbb{T}^2)$ and therefore we do not need to move derivatives onto the test functions $\varphi$. For the divergence term in \eqref{conv} we can use a decomposition similar to \eqref{decomp} to obtain that\footnote{Observe that, as opposed to \cite{CL3} here we have convergence directly in strong probabilistic sense, that is in the probability space $\Omega$. This means that we do not need an extra probability space and a Yamada-Watanabe argument. This is due to the Cauchy property \eqref{cauchyrel} and the uniform a priory estimates \eqref{approxslnestim}.}
\begin{equation*}
    \begin{aligned}
       \displaystyle\int_0^t\langle f_R(\ckX^{N})\left( (1+\delta)\mathcal{D}_{\mathcal{J}_N\cku^{N}})\ckX^{N}\right) - f_R(\ckX) ( (1+\delta)\mathcal{D}_{\cku})\ckX, \varphi \rangle ds \xrightarrow{N\rightarrow\infty} 0 \ \ \ \ \hbox{in} \ \ \ L^2(\mathbb{P}). 
    \end{aligned}
\end{equation*}

We show in Section \ref{normcontinuity} that the $\|\cdot\|_{k,2}$ norm is continuous, that is $\|\check{X}\|_{k,2} \in C([0,T]; \mathbb{R})$. Then
$\check{X} \in C([0,T]; \cW^{k,2}(\mathbb{T}^2))$.
\end{proof}
\begin{corollary}
There exists a stopping time $\tau$ and a predictable and progressively measurable process $$X(\cdot) \in C([0,\tau]; \mathcal{W}^{k,2}(\mathbb{T}^2)) \bigcap L^2(0,\tau; \mathcal{W}^{k+1,2}(\mathbb{T}^2)) $$ such that for any $p>0$ and $k\geq 1$
\begin{equation}\label{origslnestimlast}
  \mathbb{E}\left[ \displaystyle\sup_{t\in[0,\tau]}\|X_t\|_{k,2}^p + \left(\displaystyle\int_0^{\tau} \|X_t\|_{k+1,2}^2dt\right)^{p/2}\right ]  \leq C.
\end{equation}
and for any $\epsilon>0$ 
\begin{equation}
    \mathbb{E}[\displaystyle\sup_{t\in[0,\tau]}\|X_t^N-X_t\|_{k-\epsilon,2}^2] \xrightarrow[N\rightarrow\infty]{} 0. 
\end{equation}
\end{corollary}

The following proposition is immediate. 
\begin{proposition}\label{prop:truncated}
Let $X_0=X_0^R \in  \cW^{k,2}(\T^2)^3$ ($k\geq 1$), $X^R : \Omega \times [0, \infty) \times \T^2 \rightarrow \mathbb{R}^3$ the global solution of the truncated model \eqref{eq:truncatedlu} ($R>0$). Then the pair $(X,\tau^R),$ with $X: \Omega \times [0, \tau^R) \times \T^2 \rightarrow \mathbb{R}^3$ such that $X_t=X_t^R$ for $t\in[0,\tau^R]$ is a local solution of the original model \eqref{eq:mainlu}.
\end{proposition}

\subsection{Maximal solution}\label{sect:maximalsolution}
We introduce below the result which states that the local strong solution constructed in the previous section is actually maximal. Note that the property of being a maximal solution is not influenced directly by the different form of the model, but by the structure of the a priori estimates.

\begin{proposition}\label{prop:maxsollu}
Given $X_0\in \mathcal{W}^{k,2}(\mathbb{T}^2)^3$ and $R>0$, there exists a
unique maximal solution $\left(X,\mathcal{T}\right) $ of the original model \eqref{eq:mainlu} such that 
\begin{equation}\label{limsup}
\displaystyle\limsup_{t\rightarrow \mathcal{T}}\|
X_t\|_{k,2}=\infty
\end{equation}%
whenever $\mathcal{T}<\infty $.
\end{proposition}
\begin{proof}\footnote{A similar proof for $k=1$ can be found in \cite{CL3}.}
\textbf{Existence:} If we choose $R=n$ in Proposition \ref{prop:truncated} then $\left( X^{n},\tau ^{n}\right) $ is a local
solution of the system \eqref{gencl}. Also, $X^{n+1}$ satisfies the truncated system \eqref{eq:truncatedlu} for $R=n$ on the
interval $[0,\tau ^{n}]$. By the local uniqueness, it follows that 
\begin{equation}
X^{n+1}|_{[0,\tau ^{n}]}=X^{n}|_{[0,\tau ^{n}]}.  \label{consistent}
\end{equation}%
Define $\mathcal{T}:=\displaystyle\lim_{n\rightarrow \infty }\tau ^{n}$ 
and 
\begin{equation}
X|_{[0,\tau ^{n}]}:=X^{n}|_{[0,\tau ^{n}]}.  \label{maximalsolution}
\end{equation}%
We need to show that (\ref{limsup}) holds. 
If $\displaystyle\lim_{n\rightarrow \infty }\tau ^{n}=\mathcal{T}<\infty $ 
then
\begin{equation*}
\displaystyle\limsup_{t\rightarrow \mathcal{T}}\left\vert \left\vert
X_t\right\vert \right\vert _{1,2}\ge    \displaystyle\limsup_{n\rightarrow\infty}\|X_{\tau_n}\|_{1,2} =  \displaystyle\limsup_{n\rightarrow\infty} n = \infty.
\end{equation*}
\textbf{Uniqueness}. 
Assume that $\left( \bar{X}{,\mathcal{\bar{T}}}\right) $ is another solution with 
$\left( \bar{X},\bar\tau_n \right) $, $n\ge 1$ being the corresponding sequence of local solutions converging to the maximal solution. 
By the uniqueness of the truncated
equation it follows that $\bar{a}=a$ on $[0,\bar{\tau}^{n}\wedge \tau ^{n}]$%
. By taking the limit as $n\rightarrow \infty $ it follows that $\bar{X}=X$
on $[0,\mathcal{T}\wedge \mathcal{\bar{T}})$. We prove next that $\mathcal{T}%
=\mathcal{\bar{T}}$, $\mathbb{P}-a.s.$. Assume that 
\begin{equation*}
\mathbb{P}\left( \omega \in \Omega ,\mathcal{T}\left( \omega \right) \neq 
\mathcal{\bar{T}}\left( \omega \right) \right) >0.
\end{equation*}%
Note that 
\begin{equation*}
\mathbb{P}\left( \omega \in \Omega ,\mathcal{T}\left( \omega \right) \neq 
\mathcal{\bar{T}}\left( \omega \right) \right) \leq \mathbb{P}\left( \Xi
^{1}\right) +\mathbb{P}\left( \Xi ^{2}\right) 
\end{equation*}%
where 
\begin{eqnarray*}
\Xi ^{1} &=&\left\{ \omega \in \Omega ,\mathcal{T}\left( \omega \right)
<\infty ,\mathcal{T}\left( \omega \right) <\mathcal{\bar{T}}\left( \omega
\right) \right\}  \\
\Xi ^{2} &=&\left\{ \omega \in \Omega ,\mathcal{\bar{T}}\left( \omega
\right) <\infty ,\mathcal{\bar{T}}\left( \omega \right) <\mathcal{T}\left(
\omega \right) \right\} .
\end{eqnarray*}%
We prove that $\mathbb{P}\left( \Xi ^{1}\right) =\mathbb{P}\left( \Xi
^{2}\right) =0$. The two sets are symmetric so we show this only for the
first one. From the definition of the local solution, observe that%
\begin{equation*}
\mathbb{E}\left[ \sup_{s\in \lbrack 0,\bar{\tau}^{n}\left( \omega \right)
]}\Vert \bar{X}_t\Vert_{k,2}\right] <\infty 
\end{equation*}%
hence%
\begin{equation*}
\mathbb{P}\left( \omega \in \Omega ,\sup_{s\in \lbrack 0,\bar{\tau}%
^{n}\left( \omega \right) ]}\Vert \bar{X}_t\Vert_{k,2}=\infty \right) =0.
\end{equation*}%
However, if $\mathcal{T}\left( \omega \right) <\infty $ and $\mathcal{T}%
\left( \omega \right) <\bar{\tau}^{n}\left( \omega \right) $ then 
\begin{equation*}
\infty =\sup_{s\in \lbrack 0,\mathcal{T}\left( \omega \right) )}\Vert X_s\Vert
_{k,2}=\sup_{s\in \lbrack 0,\mathcal{T}\left( \omega \right) )}\Vert \bar{X_s
}\Vert_{k,2}\leq \sup_{s\in \lbrack 0,\bar{\tau}^{n}\left( \omega \right)
]}\Vert \bar{X}_s\Vert_{k,2}.
\end{equation*}%
It follows that $\mathbb{P}\left( \omega \in \Omega ,\mathcal{T}\left(
\omega \right) <\infty ,\mathcal{T}\left( \omega \right) <\bar{\tau}%
^{n}\left( \omega \right) \right) =0$ and therefore  
\begin{equation*}
\mathbb{P}\left( \Xi ^{1}\right) =\lim_{n\rightarrow \infty }\mathbb{P}%
\left( \omega \in \Omega ,\mathcal{T}\left( \omega \right) <\infty ,\mathcal{%
T}\left( \omega \right) <\bar{\tau}^{n}\left( \omega \right) \right) =0
\end{equation*}
which gives the claim.
\end{proof}

\begin{corollary}\label{stoppedestimlu}
Let $(X, \cT)$ be the maximal solution constructed in Proposition \ref{prop:maxsollu}. Then the process  $t\rightarrow X_{t\wedge \tau^R}$ is such that
\begin{equation*}
   X_{\cdot\wedge\tau^R}\in L^2\left( \Omega, C\left([0,T];\mathcal{W}^{k,2}(\mathbb{T}^2)^3\right) \right)
\end{equation*}
\begin{equation*}
   X\mathbbm{1}_{t\leq\tau^R} \in L^2\left( \Omega, L^2\left(0, T;\mathcal{W}^{k+1,2}(\mathbb{T}^2)^3\right)\right) 
\end{equation*}
for any $T>0$. In particular, 
\begin{equation}\label{difnorm}
\mathbb{E}[\|X\|_{\tau^R\wedge T,k+1,2}^2] <\infty
\end{equation}
for any $R,T>0$.
\end{corollary}
\begin{proposition}\label{highernormcontrol}
Let $(X, \cT)$ be the maximal solution constructed in Proposition \ref{prop:maxsollu}. Then 
\begin{equation*}
    \mathbb{P}\left( \cT = \hat{\cT}  \right) = 1. 
\end{equation*}
\end{proposition}
\begin{proof}[\textbf{Proof.}]
 One can observe that $\hat{\tau}^M \leq \tau^M$ for all $M\geq 0$, and therefore $\hat{\cT} \leq \cT$ $\mathbb{P}-a.s.$. We show that $\cT \leq \hat{\cT}$ $\mathbb{P}-a.s.$. On the set $\{ \hat{\cT} = \infty\}$ the inequality is trivially true, so we only need to show that
 \begin{equation*}
     \mathbb{P}\left( \{\omega \in \Omega: \hat{\cT}(\omega) < \infty, \cT(\omega) \leq \hat{\cT}(\omega)\} \right) =1. 
 \end{equation*}
 Note that
 \begin{equation*}
     \begin{aligned}
    \{ \omega \in \Omega: \displaystyle \|X(\omega)\|_{\tau^N,k+1,2} < \infty \} = \displaystyle\bigcup_M  \{ \omega \in \Omega:  \|X(\omega)\|_{\tau^N,k+1,2} < M \} = \displaystyle\bigcup_M \{ \omega \in \Omega: \tau^N(\omega) <  \hat{\tau}^M(\omega) \}
     \end{aligned}
 \end{equation*}
 and
 \begin{equation*}
     \begin{aligned}
       \displaystyle\bigcup_M \{ \omega \in \Omega: \tau^N(\omega) < \hat{\tau}^M(\omega)\} \subset \{ \omega \in \Omega: \hat{\cT}(\omega)<\infty,\tau^N(\omega) < \hat{\cT}(\omega) \}.
     \end{aligned}
 \end{equation*}
From Proposition \ref{stoppedestimlu} we deduce that 
\begin{equation*}
    \begin{aligned}
    \mathbb{P}\left( \{ \omega \in \Omega:  \|X(\omega)\|_{\tau^N\wedge T,k+1,2} < \infty\} \right) =1, \ \ \ \forall N\in\mathbb{N}, \ \ \ T>0.
    \end{aligned}
\end{equation*}
and since
$\limsup_{t\rightarrow \hat{\cT}} \|X(\omega)\|_{t,k+1,2}=\infty$ on the set $\{\hat{\cT}<\infty\}$, we deduce that 
\[
    \mathbb{P}\left( \{\omega \in \Omega: \hat{\cT}(\omega)<\infty,  \tau^N(\omega) \wedge T< \hat{\cT}(\omega)\}\right)=1,
    \ \ \ \forall N\in\mathbb{N}, \ \ \ T>0,
\]
therefore
\begin{equation*}
    \begin{aligned}
    \mathbb{P}\left( \{\omega \in \Omega: \hat{\cT}(\omega)<\infty,  \tau^N(\omega) <\hat{\cT}(\omega)\}\right)=
    \mathbb{P}\left( \bigcap_{L>1}\{\omega \in \Omega: \hat{\cT}(\omega)<\infty,  \tau^N(\omega)\wedge L < \hat{\cT}(\omega)\}\right) = 1. 
    \end{aligned}
\end{equation*}
Then we have
\begin{equation*}
    \begin{aligned}
    \mathbb{P}\left( \{\omega \in \Omega: \hat{\cT}(\omega)<\infty, \cT(\omega) \leq \hat{\cT}(\omega)\}\right) &\ge \mathbb{P}\left( \displaystyle\bigcap_N \{ \omega \in \Omega: \hat{\cT}(\omega)<\infty, \tau^N(\omega) < \hat{\cT}(\omega)\} \right) 
     =1. 
    \end{aligned}
\end{equation*}
\end{proof}
\section{Pathwise Uniqueness}\label{uniqueness}
\begin{theorem}\label{prop:uniquenesstruncatedlu} 
Let $(X^{R,1}, \tau_M^{R,1})$, $(X^{R,2}, \tau_M^{R,2})$ be two solutions
of the truncated LU-SRSW system,
which take values in the space 
$L^2\left( \Omega, C\left([0,T],\cW^{k,2}(\T^2)^3\right)\right) \cap L^2\left(\Omega, L^2\left(0,T;\cW^{k+1,2}(\T^2)^3\right)\right)$
and start from the initial conditions $X_0^1$, $X_0^2 \in \mathcal{W}^{k,2}(\mathbb{T}^2)^3$, respectively. 
Then there exists a constant $C=C(M)$ independent of $R$ such that 
\begin{equation*}
    \mathbb{E}\left[\|\bar{X}_{t\wedge\bar{\tau}_M^R}^R\|_{k,2}^2\right] \leq Ce^{Ct}\|\bar{X}_0\|_{k,2}^2,
\end{equation*}
where $\bar{X}^R := X^{R,1} - X^{R,2}$ and $\bar{\tau}_M^R := \tau_M^{R,1}\wedge \tau_M^{R,2}.$
In particular, the truncated LU-SRSW system has a unique solution in the space
\begin{equation*}
L^2\left( \Omega, C([0,T],\cW^{k,2}(\T^2)^3)\right) \cap L^2\left(\Omega, L^2(0,T;\cW^{k+1,2}(\T^2)^3)\right).
\end{equation*}
$$\hbox{where} \ \ \tau_M^{R,i}= \displaystyle\inf_t\{ \displaystyle\sup_{s\in[0,t]}\|X_s^{R,i}\|_{k,2}^2 + \displaystyle\int_0^t\|X_s^{R,i}\|_{k+1,2}^2ds \geq M \}.$$
\end{theorem}

\noindent\textbf{\textit{Proof}}
\noindent 
\noindent We show that
\begin{equation}\label{unique1lu}
    \begin{aligned}
    d\|\bar{X}_t^R\|_{k,2}^2 \leq C\left( \|X^{R,1}\|_{k,2}^4 + \|X^{R,2}\|_{k,2}^4\right) \|\bar{X}_t^R\|_{k,2}^2dt + dD_t
    \end{aligned}
\end{equation}
where $D_t$ is a local martingale given by 
\begin{equation*}
    dD_t := 2\left( \displaystyle\sum_{|\alpha|\leq k}\langle \partial^{\alpha}\bar{u}_t^R, \partial^{\alpha}\left(\mathcal{L}_{\sigma}\bar{u}_t^R\right)\rangle +  \langle\partial^{\alpha} \bar{h}_t^R, \partial^{\alpha}\left(\mathcal{L}_{\sigma} \bar{h}_t^R\right)\rangle  \right)
dB_t.
\end{equation*}
Then
\begin{equation*}
    \begin{aligned}
    \mathbb{E}\left[ e^{-C\displaystyle\int_0^{t\wedge\bar{\tau}_M^R}\|Z_s\|ds} \|\bar{X}_{t\wedge\bar{\tau}_M^R}^R\|_{k,2}^2\right] 
    & \leq \|\bar{X}_0\|_{k,2}^2 + \mathbb{E}\left[\displaystyle\int_0^{t\wedge\bar{\tau}_M^R}e^{-C\displaystyle\int_0^{s\wedge\bar{\tau}_M^R}\|Z_r\|dr}dB_{s}\right]
    \end{aligned}
\end{equation*}
where
\begin{equation*}
    \|Z\|:= \|X^{R,1}\|_{k,2}^4 + \|X^{R,2}\|_{k,2}^4.
\end{equation*}
That is
\begin{equation*}
    \begin{aligned}
   \mathbb{E}\left[\|\bar{X}_{t\wedge\bar{\tau}_M^R}^R\|_{k,2}^2\right] & 
    \leq e^{CM^4t}\|\bar{X}_0\|_{k,2}^2
    \end{aligned}
\end{equation*}
since the stopped process $D_{t\wedge\bar{\tau}_M^R}$ is a martingale and therefore its expectation is zero.
We will now prove that \eqref{unique1lu} holds, using Lemma \ref{generallemmauniquenesslu}. 

\noindent We can write
\begin{subequations}
\begin{alignat}{2}  
& d\bar{u}_t^R = \left(\tilde{Q}(\bar{u}^R) - fk\times \bar{u}_t^R + \nu\Delta\bar{u}_t^R + g\nabla%
\bar{h}_t^R + \frac{1}{2}\nabla \cdot (a\nabla\bar{u}_t^R) \right)dt + \mathcal{L}_{\sigma}\bar{u}_t^RdB_t  \\
&  d\bar{h}_t^R = \left( \tilde{Q}(\bar{h}^R, \bar{u}^R) + \eta\Delta \bar{h}_t^R + \frac{1}{2}\nabla \cdot (a\nabla\bar{h}_t^R) \right)dt +  \mathcal{L}_{\sigma}\bar{h}_t^RdB_t  . 
\end{alignat}
\end{subequations}
where 
$$Q(\bar{u}^R):=f_R(X_R^2)\nabla \cdot (u_t^{R,2}(u_t^{R,2})^T) -
f_R(X_R^1)\nabla \cdot (u_t^{R,1}(u_t^{R,1})^T) +f_R(X_R^2)\alpha u_t^{R,2}(\nabla \cdot u_t^{R,2}) - f_R(X_R^1)\alpha u_t^{R,1}(\nabla \cdot u_t^{R,1})$$
$$Q(\bar{u}^R):= \displaystyle\sum_j Q((\bar{u}^{R})_j,\bar{u}^R)$$
$$Q(\bar{h}^R, \bar{u}^R) := f_R(X_R^2)\nabla
\cdot (h_t^{R,2}u_t^{R,2}) - f_R(X_R^1)\nabla \cdot
(h_t^{R,1}u_t^{R,1})+f_R(X_R^2)\beta h_t^{R,2}(\nabla \cdot u_t^{R,2}) - f_R(X_R^1)\beta h_t^{R,1}(\nabla \cdot u_t^{R,1})$$\\
$$\tilde{Q}:= Q^S-Q$$
By the It\^{o} formula 
\begin{equation*}
\begin{aligned}
d\|\partial^{\theta}\bar{X}_t^R\|_{2}^2  &+ 2\gamma\|\partial^{\theta+1} \bar{X}%
_t^R\|_2^2dt  \\
&= -2 \left(\langle \partial^{\theta} \bar{u}_t^R,
\partial^{\theta}\left(fk\times\bar{v}_t^R + g\nabla \bar{p}_t^R\right)\rangle 
 + \langle \partial^{\theta} \bar{u}_t^R, \partial^{\theta}\tilde{Q}(\bar{u}^R) \rangle  + \langle \partial^{\theta} \bar{h}_t^R,\partial^{\theta} \tilde{Q}(\bar{h}^R, \bar{u}^R) \rangle\right) dt \\
 &+  \langle \partial^{\theta}\mathcal{L}_{\sigma}\bar{u}_t^R, \partial^{\theta}\mathcal{L}_{\sigma} \bar{u}_t^R \rangle dt +  \langle \partial^{\theta}\mathcal{L}_{\sigma}\bar{h}_t^R, \partial^{\theta}\mathcal{L%
}_{\sigma}\bar{h}_t^R\rangle dt + dD_t 
\end{aligned}
\end{equation*}
for a multi-index $\theta$ with $|\theta|\leq k$
All the terms which do not contain a stochastic integral are controlled as functions of $C(\epsilon,R,\delta)\|Z\|\|\bar{X}\|_{k,2}^2 + C(\epsilon,\delta) \|\bar{X}\|_{k+1,2}^2$ using Lemma \ref{generallemmauniquenesslu}. We choose $C(\epsilon,\delta) < \gamma $ such that all terms which are controlled by $C(\epsilon,\delta)\|\bar{X}\|_{k+1,2}^2$ on the right hand side cancel out the term $2\gamma\|\bar{X}\|_{k+1,2}^2$ on the left hand side. Then \eqref{generallemmauniquenesslu} holds as requested and therefore the two solutions are indistinguishable as processes with paths in $L^2\left( \Omega, C\left([0,T],\cW^{k,2}(\T^2)^3\right)\right) \cap L^2\left(\Omega, L^2\left(0,T;\cW^{k+1,2}(\T^2)^3\right)\right)$. 

 For the weak solution $(k=0)$ the stopping we have $\displaystyle\lim_{R\rightarrow\infty}\tau_M^{R,1} = \displaystyle\lim_{R\rightarrow\infty}\tau_M^{R,1} = \infty$ and the "weight" $Z$ is given by
\begin{equation*}
    \|Z\| := \left( \|X^1\|_{2}^2 + \|X^2\|_{2}^2 \right)^2
\end{equation*}
therefore it is bounded in $L^2(\Omega, L^2(0,T))$.

\section{Proof of Theorem \ref{thmlu2} - Existence of a global weak solution}\label{sect:globalweak}
For $\alpha=\beta = -\frac{1}{2}$ the main system \eqref{gencl} can be rewritten as
\begin{equation*}
    dX_t + \tilde{\mathcal{A}}(X_t)dt + \mathcal{G}(X_t)dB_t = \gamma\Delta X_t
\end{equation*}
In this case the truncation function $f_R\equiv 1$ and then the approximating system can be rewritten as\footnote{We will use the same notation $\tilde{\mathcal{A}^N}$ and $\mathcal{G}^N$, but taking into account that $f_R\equiv 1$, we can discard the truncation function when writing the approximating system.}
\begin{equation*}
    dX_t^N + \tilde{\mathcal{A}}^N(X_t^N)dt + \mathcal{G}^N(X_t^N)dB_t^N = \gamma\Delta X_t^N
\end{equation*}
where\footnote{Note that $\tilde{\mathcal{A}}$ is a particular form of $\mathcal{A}$, for $\alpha=\beta = -\frac{1}{2}$.}
\begin{equation*}
\tilde{\mathcal{A}}^N(X^N) :=\left( 
\begin{array}{c}
 fk\times \mathcal{J}_Nu^N+ \mathcal{J}_N(\mathcal{\widetilde{L}}_{\mathcal{J}_Nu^N} \mathcal{J}_Nu^N) + \mathcal{J}_N(\mathcal{D}_{\mathcal{J}_Nu^N} \mathcal{J}_Nu^N) + \mathcal{J}_Ng\nabla (\mathcal{J}_Nh^N )-\frac{1}{2}\mathcal{J}_{N}\left(\mathcal{L}_{\sigma}^2\mathcal{J}_{N}u^{N}\right)  \\ 
\mathcal{J}_N(\mathcal{\widetilde{L}}_{\mathcal{J}_Nu^N} \mathcal{J}_Nh^N) + \mathcal{J}_N(\mathcal{D}_{\mathcal{J}_Nu^N}\mathcal{J}_Nh^N )- \frac{1}{2} \mathcal{J}_{N}\left(\mathcal{L}_{\sigma}^2\mathcal{J}_{N}h^{N}\right)
\end{array}
\right)
\end{equation*}
and 
\begin{equation*}
\mathcal{G}^N(X^N) :=\left( 
\begin{array}{c}
\mathcal{J}_N(\cL_{\sigma} \mathcal{J}_Nu^N) \\ 
\mathcal{J}_N(\cL_{\sigma} \mathcal{J}_Nh^N)%
\end{array}%
\right). 
\end{equation*}
\begin{remark}
Note that when writing the approximation system above we could drop the projection operator $\mathcal{J}_N$ as $X^N$ is already a solution of a system which lives in the projected space (see the Appendix for details). However, we prefer to keep this notation in order better highlight the properties of the approximating sequence.
\end{remark}
\begin{proof}[\textbf{Proof of Theorem \ref{thmlu2}}.]
From Theorem \ref{globalapprox1} and Proposition \ref{prop:tightness} we deduce that we can choose a weakly (i.e. in distribution) convergent subsequence $(X^{N_j})_j$. By the Skorokhod representation theorem, there exists a probability space $\left( \hat{\Omega}, \hat{\mathcal{F}}, \hat{\mathbb{P}}\right)$ and a sequence of processes $(\hat{X}^{N_j})_j$ which has the same distribution as $(X^{N_j})_j$ and converges $\hat{\mathbb{P}}$-almost surely in the Skorokhod topology of the paths space. That is, there exists $\hat{X}$ such that $\lim_{j\rightarrow \infty}\hat{X}^{N_j} = \hat{X}$ $\hat{\mathbb{P}}-a.s.$ in $C([0,T]; \mathcal{W}^{-1,2}(\mathbb{T}^2)^3)$. Since $\hat{X}^{N_j}$ and $X^N$ have the same distribution they also satisfy the same SPDE (and the same a priori estimates) and for any test function $\varphi \in C^{\infty}(\mathbb{T}^2)$ we can write (we write $\hat{X}^N$ instead of $\hat{X}^{N_j}$)
\begin{equation*}
\begin{aligned}
     \langle \hat{u}_t^N,\varphi \rangle = \langle \hat{u}_0^N, \varphi\rangle &-\displaystyle\int_0^t\left(\langle \hat{u}^N_s, \mathcal{\widetilde{L}}_{\hat{u}^N_s}^{*}\varphi \rangle + \langle \hat{u}^N_s, \mathcal{D}_{\hat{u}^N_s}^{\star}\varphi\rangle + \langle fk\times \hat{u}^N_s, \varphi\rangle + g\langle \hat{h}^N_s, \nabla\varphi \rangle - \nu \langle \hat{u}^N_s, \Delta \varphi \rangle \right) ds \\
     & + \frac{1}{2}\displaystyle\int_0^t\langle \hat{u}^N_s, (\mathcal{L}_{\sigma}^2)^{\star}\varphi \rangle ds   - \displaystyle\int_0^t \langle \hat{u}^N_s, (\mathcal{L}_{\sigma})^{\star}\varphi \rangle dB_s^N  
\end{aligned}
\end{equation*}
\begin{equation*}
\begin{aligned}
     \langle \hat{h}^N_t, \varphi \rangle = \langle \hat{h}_0^N, \varphi\rangle & -  \displaystyle\int_0^t\left(\langle \hat{h}^N_s, (\mathcal{\widetilde{L}}_{\hat{u}^N_s})^{\star}\varphi\rangle + \langle \hat{h}^N_s, \mathcal{D}_{\hat{u}^N_s}^{\star}\varphi \rangle - \eta \langle  \hat{h}^N_s, \Delta\varphi  \rangle - \frac{1}{2}\langle \hat{h}^N_s, (\mathcal{L}_{\sigma}^2)^{\star}\varphi \rangle\right)ds \\
     & - \displaystyle\int_0^t \langle \hat{h}^N_s, (\mathcal{L}_{\sigma})^{\star}\varphi \rangle dB_s. 
\end{aligned}    
\end{equation*}
We can easily pass to the limit in the linear terms. For the nonlinear term: 
\begin{equation*}
    \begin{aligned}
   & \displaystyle\int_0^t | \langle \hat{u}_s^N, \tilde{\mathcal{L}}_{\hat{u}_s^N}^{\star}\varphi\rangle - \langle \hat{u}_s, \tilde{\mathcal{L}}_{\hat{u}_s}^{\star}\varphi\rangle|ds  \leq \displaystyle\int_0^t |\langle \hat{u}_s^N-\hat{u}_s,\tilde{\mathcal{L}}_{\hat{u}_s^N}^{\star}\varphi \rangle|ds + \displaystyle\int_0^t |\langle \hat{u}, \tilde{\mathcal{L}}_{\hat{u}_s^N}^{\star}\varphi - \tilde{\mathcal{L}}_{\hat{u}_s}^{\star}\varphi\rangle| ds\\
    & \leq \displaystyle\int_0^t\|\hat{u}_s^N-\hat{u}_s\|_2\|\tilde{\mathcal{L}}_{\hat{u}_s^N}^{\star}\varphi\|_2ds +  \displaystyle\int_0^t |\langle \hat{u} \cdot \nabla\varphi, \hat{u}_s^N-\hat{u}_s\rangle|ds\\
  & \leq \left(\displaystyle\int_0^t  \|\hat{u}_s^N-\hat{u}_s\|_{2}^{2}\|\tilde{\mathcal{L}}_{\hat{u}_s^N}^{\star}\varphi\|_2^2ds\right)^{1/2} + \left( \displaystyle\int_0^t \|\hat{u}_s \cdot \nabla \varphi\|_2^2\|\hat{u}_s^N-\hat{u}\|_2^2ds\right)^{1/2}\\
  & \leq \left( \displaystyle\sup_{s\in [0,T]}\|\tilde{\mathcal{L}}_{\hat{u}_s^N}^{\star}\varphi\|_2^2\right)^{1/2}\left(\displaystyle\int_0^T\|\hat{u}_s^N-\hat{u}_s\|_{2}^{2}ds\right)^{1/2} + \left(\displaystyle\int_0^t\|\hat{u}_s\cdot\nabla\varphi\|_2^2\|\hat{u}_s^N-\hat{u}\|_2^2ds \right)^{1/2}\\
  & \leq C\left( \displaystyle\int_0^t\|\hat{u}_s^N-\hat{u}_s\|_{-1,2}^2\|\hat{u}_s^N-\hat{u}_s\|_{1,2}^2ds\right)^{1/4} + C(\|\varphi\|_{2,2})\left( \displaystyle\sup_{s\in[0,T]}\|\hat{u}_s^N-\hat{u}_s\|_{-1,2}^2\displaystyle\int_0^T\|\hat{u}_s\|_{1,2}^2(\|\hat{u}_s\|_{1,2}^2 + \|\hat{u}_s^N\|_{1,2}^2)ds\right)^{1/2}\\
  & \leq C\left( \displaystyle\sup_{s\in[0,T]}\|\hat{u}_s^N-\hat{u}_s\|_{-1,2}^2\left(\displaystyle\int_0^T \|\hat{u}_s^N\|_{1,2}^2ds+ \displaystyle\int_0^T\|\hat{u}_s\|_{1,2}^2ds\right)\right)^{1/4} \\
 & +  C(\|\varphi\|_{2,2})\left( \displaystyle\sup_{s\in[0,T]}\|\hat{u}_s^N-\hat{u}_s\|_{-1,2}^2\displaystyle\int_0^T\|\hat{u}_s\|_{1,2}^2(\|\hat{u}_s\|_{1,2}^2 + \|\hat{u}_s^N\|_{1,2}^2)ds\right)^{1/2}\\
      & \xrightarrow[N\rightarrow \infty]{} 0 
    \end{aligned}
\end{equation*}
The convergence of the stochastic integrals holds true as a direct application of \cite{Jacubowski95}: see similar calculations for similar stochastic integrals with transport noise in \cite{CL3} Proposition 3.12. 
We obtain a martingale solution and combining this with the pathwise uniqueness property, we conclude that the solution is pathwise. The term in the equation for $h$ converge with similar arguments. Then the limit $\tilde{X}$ satisfies the following system, for any test functions $\varphi \in \mathcal{W}^{1,2}(\mathbb{T}^2)$: \\
\begin{equation*}
\begin{aligned}
     \langle \hat{u}_t,\varphi \rangle = \langle \hat{u}_0, \varphi\rangle &-\displaystyle\int_0^t\left(\langle \hat{u}_s, \mathcal{\widetilde{L}}_{\hat{u}_s}^{*}\varphi \rangle + \langle \hat{u}_s, \mathcal{D}_{\hat{u}_s}^{\star}\varphi\rangle + \langle fk\times \hat{u}_s, \varphi\rangle + g\langle \hat{h}_s, \nabla\varphi \rangle - \nu \langle \nabla\hat{u}_s, \nabla \varphi \rangle\right)ds\\
     & - \frac{1}{2}\displaystyle\int_0^t\langle \hat{u}_s, (\mathcal{L}_{\sigma}^2)^{\star}\varphi \rangle ds \\
     & - \displaystyle\int_0^t \langle \hat{u}_s, (\mathcal{L}_{\sigma})^{\star}\varphi \rangle dB_s  
\end{aligned}
\end{equation*}

\begin{equation*}
\begin{aligned}
     \langle \hat{h}_t, \varphi \rangle = \langle \hat{h}_0, \varphi\rangle & -  \displaystyle\int_0^t\left(\langle \hat{h}_s, (\mathcal{\widetilde{L}}_{\hat{u}_s})^{\star}\varphi\rangle + \langle \hat{h}_s, \mathcal{D}_{u_s}^{\star}\varphi \rangle - \eta \langle  \nabla\hat{h}_s, \nabla\varphi  \rangle - \frac{1}{2}\langle \hat{h}_s, (\mathcal{L}_{\sigma}^2)^{\star}\varphi \rangle\right)ds \\
     & - \displaystyle\int_0^t \langle \hat{h}_s, (\mathcal{L}_{\sigma})^{\star}\varphi \rangle dB_s. 
\end{aligned}    
\end{equation*}
We have so far that $X\in C_w\left( [0,T], L^2(\mathbb{T}^2)^3\right)$, but due to Proposition \ref{prop:normcont} it is actually in $ C\left( [0,T], \mathcal{W}^{k,2}(\mathbb{T}^2)^3\right)$ for any $k \geq 0$. 
Since the solution is continuous, the control in \eqref{originalestim} can be obtained from the control in \eqref{estimapprox1} using Fatou's lemma:
\begin{equation*}
    \begin{aligned}
      \displaystyle\sup_{s\leq T}\|X_s\|_2^2 + \displaystyle\int_0^T \|X_s\|_{1,2}^2 ds & \leq \displaystyle\liminf_{N\rightarrow\infty}\left(  \displaystyle\sup_{s\leq T}\|X_s^N\|_2^2 + \displaystyle\int_0^T \|X_s^N\|_{1,2}^2 ds\right) < \infty.
    \end{aligned}
\end{equation*}

Note that for the time integral we can apply Fatou directly (we actually used this in the convergence of the nonlinear term above), but we need continuity to ensure that we can pass to the limit also in the supremum norm. 
\end{proof}
\begin{theorem}\label{globalapprox1}
For each $N$, there exists a unique $L^2$-valued \underline{continuous} \textbf{untruncated} solution $X_t^N$ and a constant $C$ independent of $N$ such that 
\begin{equation}\label{estimapprox1}
    \displaystyle\sup_N  \left(\displaystyle\sup_{s\leq T} \|X_s^N\|_2^2 + \displaystyle\int_0^T\|X_s^N\|_{1,2}^2ds \right)\leq C \|X_0\|_2^2.
\end{equation}
\end{theorem}
\begin{proof}[\textbf{Proof of Theorem \ref{globalapprox1}.}]
As described above, the (untruncated) approximating system can be written as
\begin{equation*}
dX_t^N + \mathcal{A}^N(X_t^N)dt + \mathcal{G}^N(X_t^N)dB_t^N = \frac{1}{2} \mathcal{J}_N (\nabla \cdot (a \nabla \mathcal{J}_N X_t^N)) + \nu \Delta X_t^N dt
\end{equation*}
that is 
\begin{equation*}\label{originalumlfduntruncated}
      \begin{aligned}
    \hspace{-6mm}  du^{N} & +  fk \times  \mathcal{J}_Nu^{N} dt +  \mathcal{J}_{N}(\mathcal{\widetilde{L}}_{\mathcal{J}_{N}u^{N}} \mathcal{J}_{N}u^{N}) dt + \mathcal{J}_{N}\left(\mathcal{D}_{\mathcal{J}_{N}u^{N}}\mathcal{J}_{N}u^{N}\right)dt \\
      & + g\mathcal{J}_{N}(\nabla \mathcal{J}_{N}h^{N}) dt + \mathcal{J}_{N}\left(\mathcal{L}_{\sigma}\mathcal{J}_{N}u^{N}\right)dB_t^N \\
      & = \nu \Delta u^N dt + \frac{1}{2}\mathcal{J}_{N}\left(\mathcal{L}_{\sigma}^2\mathcal{J}_{N}u^{N}\right)dt 
      \end{aligned}  
\end{equation*}
\vspace{2mm}
\begin{equation*}   
\begin{aligned}
 \hspace{-20mm}  dh^{N} +  & \mathcal{J}_{N} (\mathcal{\widetilde{L}}_{\mathcal{J}_{N} u^{N}} \mathcal{J}_{N}h^{N})dt + \mathcal{J}_{N} (\mathcal{D}_{\mathcal{J}_{N}u^{N}}\mathcal{J}_{N}h^{N} )dt \\
    & + \mathcal{J}_{N}\left(\mathcal{L}_{\sigma}\mathcal{J}_{N}h^{N}\right)dB_t^N = \eta \Delta h^N dt + \frac{1}{2}\mathcal{J}_{N}\left(\mathcal{L}_{\sigma}^2\mathcal{J}_{N}h^{N}\right)dt.
   \end{aligned}
\end{equation*}
By the It\^{o} formula 
\begin{equation*}
    \begin{aligned}
        d\|X_t^N\|_2^2 + 2\gamma\|X_t^N\|_{1,2}^2dt &= -2\langle X_t^N,\mathcal{A}^N(X_t^N) \rangle dt +  \langle X_t^N, \mathcal{J}_N(\mathcal{L}_{\sigma}^2(\mathcal{J}_NX_t^N))\rangle dt \\
        & -2 \langle X_t^N, \mathcal{G}^N(X_t^N) \rangle dB_t^N + \langle \mathcal{G}^N(X_t^N), \mathcal{G}^N(X_t^N)\rangle dt.
    \end{aligned}
\end{equation*}
Note that the nonlinear terms cancel out in the first scalar product, more precisely
\begin{equation*}
\langle \mathcal{J}_NX_t^N, \mathcal{J}_Nu_t^N \cdot \nabla(\mathcal{J}_NX_t^N)\rangle + \frac{1}{2}\langle \mathcal{J}_NX_t^N, (\nabla \cdot \mathcal{J_N}u_t^N)\mathcal{J_N}X_t^N\rangle = 0
\end{equation*}
 by integration by parts.
Similarly, using that $u^S = \frac{1}{2}\nabla \cdot a$ and $a = \sigma\sigma^T$,
\begin{equation*}
   \langle \mathcal{J}_N X_t^N, \mathcal{J}_N u_t^{S,N} \cdot \nabla(\mathcal{J}_NX_t^N \rangle + \langle X_t^N, \mathcal{J}_N(\mathcal{L}_{\sigma}^2(\mathcal{J}_NX_t^N))\rangle + \langle \mathcal{G}^N(X_t^N), \mathcal{G}^N(X_t^N)\rangle = 0.
\end{equation*}
This is the reason why we do not need a truncation in the weak solution case.
The only term which does not vanish on the right hand side is $g\langle u^N, \nabla h^N \rangle$. However, this can be controlled by $\gamma\|X^N\|_{1,2}^2$ and therefore we have
\begin{equation*}
    \|X_t^N\|_2^2 + \gamma\displaystyle\int_0^t\|X_s^N\|_{1,2}^2ds \leq C\|X_0\|_2^2
\end{equation*}
which gives the required bound. Note that since $\nabla \cdot \sigma dB_t = 0$, the control holds pathwise. 
\end{proof}
\subsection{Tightness}
\begin{proposition}\label{prop:tightness}
The approximating sequence $(X^N)_N$ is relatively compact in the space $C([0,T];\mathcal{W}^{-1,2}(\mathbb{T}^2)^3)$. 
\end{proposition}
\begin{proof}[\textbf{Proof}]By a standard Arzela-Ascoli theorem the following embedding is true: $$L^{\infty}(0,T; L^2(\mathbb{T}^2)) \bigcap \mathcal{W}^{\beta,p}(0,T;\mathcal{W}^{-1,2}(\mathbb{T}^2)) \hookrightarrow C([0,T];\mathcal{W}^{-1,2}(\mathbb{T}^2)).$$ 
Let $B^K:= B^1(0,K) \cap B^2(0,K)$ with $B^1(0,K)\in L^{\infty}(0,T;L^2(\mathbb{T}^2)^3)$ and 
$B^2(0,K)\in \mathcal{W}^{\beta,p}(0,T;\mathcal{W}^{-1,2}(\mathbb{T}^2)^3)$. We want $L:=\displaystyle\lim_{K\rightarrow\infty}\displaystyle\sup_N\mathbb{P}\left(X^N \notin B^K \right)=0$. 
By Lemma \ref{fracttightness} and Lemma \ref{negativesobnorm} we have  $$L\leq \displaystyle\lim_{K\rightarrow\infty}\displaystyle\sup_N\left(\frac{\mathbb{E}\left[\displaystyle\sup_{s\in[0,T]}\|X^N\|_2^2\right]}{K^2}+\frac{\mathbb{E}\left[\|X^N\|_{\mathcal{W}^{\beta,p}(0,T;\mathcal{W}^{-1,2}(\mathbb{T}^2))}^p\right]}{K^p}\right)=0,$$ 
which proves the assumption. 
\end{proof}
\begin{lemma}\label{fracttightness}
Let $X^N$ be the solutions of the approximating sequence. Then 
\begin{equation*}
	X^N \in L^p\left(\Omega; \mathcal{W}^{\beta,p}\left(0,T;\mathcal{W}^{-1,2}(\mathbb{T}^2)^3\right)\right)	
\end{equation*}
with $p\in(2,\infty)$ and $\beta \in [0,\frac{1}{2})$ such that $p\beta \geq 1$, and there exists a constant $C=C(\beta,p,T,R)$ independent of $N$ such that 
\begin{equation*}
\mathbb{E}\left[ \|X^N\|_{\mathcal{W}^{\beta,p}(0,T;\mathcal{W}^{-1,2}(\mathbb{T}^2)^3)}^p\right] \leq C.	
\end{equation*}
\end{lemma}
\begin{proof}[\textbf{Proof.}]
Recall that
\begin{equation}\label{eq:wpnorm}
\|X_t^N\|_{\mathcal{W}^{\beta, p}(0, T ; \mathcal{W}^{-1,2}(\mathbb{T}^2))}^p:=\int_{0}^{T}\left\|X_{t}^N\right\|_{\mathcal{W}^{-1,2}(\mathbb{T}^2)}^{p} d t+\int_{0}^{T} \int_{0}^{T} \frac{\left\|X_{t}^N-X_{s}^N\right\|_{\mathcal{W}^{-1,2}(\mathbb{T}^2)}^{p}}{|t-s|^{1+\beta p}} d t d s.
\end{equation} 
We have to show that there exists a constant $C=C\left(
T\right) $ independent of $N$ such that
\begin{equation}\label{eq:negsobnorm}
\mathbb{E}\left[ \left\|X_{t}^{N}-X_{s}^{N}\right\| _{\mathcal{W}^{-1,2}(\mathbb{T}^2)}^{p}%
\right] \leq C |t-s|^{p/2}.
\end{equation}
Using Lemma \ref{negativesobnorm} below we have that for any test function $\varphi \in \mathcal{W}^{1,2}(\mathbb{T}^2)$, using Ladyzhenskaya inequality (for 2D)
\begin{equation*}
    \begin{aligned}
    | \langle \varphi,\tilde{\mathcal{A}}^N(X_r^N) \rangle |
    & \leq C \displaystyle \|\varphi\|_{1,2}\|X_r^N\|_4^2 \\
    & \leq C   \|\varphi\|_{1,2}\|X_r^N\|_2\|\nabla X_r^N\|_2 \\
    \end{aligned} 
\end{equation*}
that is
\begin{equation*}
    \|\tilde{\mathcal{A}}^N(X_r^N)\|_{-1,2} \leq C \|X_r^N\|_2\|\nabla X_r^N\|_2 
\end{equation*}
and therefore using Hölder and Cauchy-Schwartz inequalities
\begin{equation*}
    \begin{aligned}
     \mathbb{E}\left[\left | \displaystyle\int_s^t \|\tilde{\mathcal{A}}^N(X_r^N)\|_{-1,2}dr\right |^p \right]   & \leq C\mathbb{E}\left[ \left | \displaystyle\int_s^t \|X_r^N\|_2\|\nabla X_r^N\|_2 dr\right |^{p}\right] \\
     & \leq C\mathbb{E}\left[ \left( \displaystyle\int_s^t \|X_r^N\|_2^2 dr\right)^{p/2}\left( \displaystyle\int_s^t \|\nabla X_r^N\|_2^2 dr\right)^{p/2} \right]\\
     & \leq  C\mathbb{E}\left[(t-s)^{p/2}\displaystyle\sup_{r\in[s,t]} \|X_r^N\|_2^p\left(\displaystyle\int_0^T \|\nabla X_r^N\|_2^2dr\right)^{p/2} \right] \\
     & \leq C(t-s)^{p/2}\sqrt{\mathbb{E}\left[\displaystyle\sup_{r\in[s,t]} \|X_r^N\|_2^{2p}\right]\mathbb{E}\left[\left(\displaystyle\int_0^T \|\nabla X_r^N\|_2^2dr\right)^{p} \right]} \\
     & \leq \tilde{C}(t-s)^{p/2} 
    \end{aligned}
\end{equation*}
provided we can control $sup$ up to any power $p$. The same argument is used to control the term which contains the Laplacian. 
Finally, for the stochastic integral we use that
\begin{equation*}
    \begin{aligned}
     |\langle \varphi,\mathcal{G}^N(X_r^N) \rangle|
    & \leq C(\|\sigma\|_{1,\infty}) \displaystyle \|\varphi\|_{2}\|X_r^N\|_2 \\
    \end{aligned} 
\end{equation*}
and
\begin{equation*}
    \|\mathcal{G}^N(X_r^N)\|_{-1,2} \leq C(\sigma)\|X_r^N\|_2 
\end{equation*}
Then using the Burkholder-Davis-Gundy inequality we have
\begin{equation*}
    \begin{aligned}
    \mathbb{E}\left[\left | \displaystyle\int_s^t \| \mathcal{G}^N(X_r^N) \|_{-1,2} dB_r\right | ^p \right] & \leq C(p) \mathbb{E}\left[\left( \displaystyle\int_s^t \| \mathcal{G}^N(X_r^N) \|_{-1,2}^2 dr\right )^{p/2} \right]\\
    & \leq C(p) \mathbb{E}\left[\left( \displaystyle\int_s^t C(\sigma)\|X_r^N\|_{2}^2 dr\right )^{p/2} \right] \\
    & \leq C(p,\sigma)(t-s)^{p/2}\mathbb{E}\left[ \displaystyle\sup_{r\in[s,t]} \|X_r^N\|_2^p\right]\\
     & \leq \tilde{C}(t-s)^{p/2}.
    \end{aligned} 
\end{equation*}
Note that the first integral in \eqref{eq:wpnorm} is controlled similarly due to the same Lemma \ref{negativesobnorm}. 
Therefore \eqref{eq:negsobnorm} holds. 
\end{proof}
\begin{remark}
Note that we can do the same for test functions $\varphi \in \mathcal{W}^{k,2}(\mathbb{T}^2)$ and this gives us a control in $W^{-k,2}(\mathbb{T}^2)$. Therefore the solution is in $C([0,T]; \mathcal{W}^{-k,2}(\mathbb{T}^2)^3)$ with $k>1$.
\end{remark}

\subsection{Continuity of the map $t \rightarrow \|X_t\|_{k,2}$ with $k\geq 0$ }\label{normcontinuity}

The following results proves the continuity of the norm processes both for the week and for the strong solution. In both cases, the arguments are the same, the only difference is the use of different sets of estimates of the approximating
sequence

\begin{proposition}\label{prop:normcont}
The map $t \rightarrow \|X_t\|_{k,2}$ is strongly continuous for any $k\geq 0$. 
\end{proposition}
\begin{proof}[\textbf{Proof of Proposition \ref{prop:normcont}.}]
$\left.\right.$\\
Case $k\geq 2$ (strong solution). We use here the estimates in Lemma \ref{le:approxslnestim} in the Appendix. With arguments similar to the ones from the proof of Lemma \ref{fracttightness} we show that there exists $C$ such that
\begin{equation*}\label{eq:normcont}
\mathbb{E}\left[\left( \|X_t\|_{k,2}^2 - \|X_s\|_{k,2}^2\right)^4 \right] \leq C|t-s|^2
\end{equation*}
More precisely, for the approximating sequence $(\ckX^N)_N$ we have 
\begin{equation*}
    \begin{aligned}
       & \mathbb{E}\left[\left( \|\ckX_t^N\|_{k,2}^2 - \|\ckX_s^N\|_{k,2}^2\right)^4\right] \\
       & \leq C \displaystyle\sum_{|\theta|\leq k}\mathbb{E}\left[\left(\displaystyle\int_s^t \left(-2\langle \partial^{\theta} \ckX_r^N, \partial^{\theta} \check{\mathcal{A}}^N(\ckX_r^N)\rangle + \langle \partial^{\theta}\ckX_r^N, \partial^{\theta}(\gamma\Delta \ckX_r^N)\rangle + \langle \partial^{\theta}\check{\mathcal{G}}^N(\ckX_r^N),\partial^{\theta}\check{\mathcal{G}}^N(\ckX_r^N)\rangle \right)dr \right)^4\right]\\
& + C \displaystyle\sum_{|\theta|\leq k}\mathbb{E}\left[\left( -2 \displaystyle\int_s^t\langle\partial^{\theta}\ckX_r^N,\partial^{\theta}\check{\mathcal{G}}^N(\ckX_r^N)\rangle dB_r  \right)^4\right] \\
& \leq C_1\mathbb{E}\left[\left(\displaystyle\int_s^t\|\ckX_r^N\|_{k,2}^2 + \epsilon_1(\gamma) \|\ckX_r^N\|_{k+1,2}^2 + \epsilon_2(\gamma)\|\ckX_r^N\|_{k+1,2}^2 \right)^2\right]+ C_2\mathbb{E}\left[\left(\displaystyle\int_s^t|\langle\partial^{\theta}\ckX_r^N,\partial^{\theta}\check{\mathcal{G}}^N(\ckX_r^N)\rangle|^2 dr \right)^2\right] \\
& \leq \tilde{C}_1(t-s)^2\mathbb{E}\left[\displaystyle\sup_{r\in[s,t]}\|\ckX_r^N\|_{k,2}^4\right] + \tilde{C}_2(t-s)^2\mathbb{E}\left[\displaystyle\sup_{r\in[s,t]}\|\ckX_r^N\|_{k,2}^4\right] \\
       &\leq C(t-s)^2
    \end{aligned}
\end{equation*}
Then
\begin{equation*}
    \begin{aligned}
       & \mathbb{E}\left[\left( \|\ckX_t\|_{k,2}^2 - \|\ckX_s\|_{k,2}^2\right)^4\right] \leq \displaystyle\liminf_{N\rightarrow\infty}\mathbb{E}\left[\left( \|\ckX_t^N\|_{k,2}^2 - \|\ckX_s^N\|_{k,2}^2\right)^4\right] 
       \leq C(t-s)^2.
    \end{aligned}
\end{equation*}
Case $k=0$ (weak solution). Similar calculations hold due to Lemma \ref{negativesobnorm} below. 
\end{proof}
\begin{remark}
For the estimates in the proof of Proposition \ref{prop:normcont} we actually use a control in the norm $\|\cdot\|_{T,k,2}$ but given the fact that we know that the solution $X$ exists and its $\mathcal{W}^{k+1,2}$ norm is $L^2$-integrable we can choose $epsilon$ in Lemma \eqref{approxslnestim}. in the Appendix so that all terms which involve $\displaystyle\int_s^t \|X\|_{k+1,2}^2dr$ disappear, and therefore the control is then true in the norm $\|\cdot\|_{k,2}$ as required. 
\end{remark}
\begin{lemma}\label{negativesobnorm}
There exists two constants $C_1$ and $C_2$ such that 
\begin{subequations}
\begin{alignat}{2}
   & \displaystyle\int_0^T\|\tilde{\mathcal{A}}^N(X_s^N)\|_{-1,2}^2 ds \leq C_1 \\
   & \displaystyle\int_0^T\|\mathcal{G}^N(X_s^N)\|_{-1,2}^2 ds \leq C_2.
\end{alignat}
\end{subequations}
\end{lemma}
\begin{proof}[\textbf{Proof.}]
We proceed below with the analysis for the nonlinear terms, as the one for the linear term follows directly with the same type of arguments. 
Using the fact that $(\mathcal{L}_u+\mathcal{D}_u)u = \nabla \cdot \left(u\otimes u\right)$ and Ladyzhenskaya's inequality we write
\begin{equation*}
    \begin{aligned}
        |\langle \varphi, \nabla \cdot (u\otimes u)\rangle| & = |\langle \nabla\varphi,  u\otimes u \rangle| \\
        & \leq C \|\varphi\|_{1,2}\|u\|_4^2\\
        & \leq C\|\varphi\|_{1,2}\|u\|_2\|\nabla u\|_2.
    \end{aligned}
\end{equation*}
Therefore 
\begin{equation*}
    \|\nabla \cdot (u \otimes u)\|_{-1,2}^2 \leq C\|u\|_2\|\nabla u\|_2
\end{equation*}
and
\begin{equation*}
    \begin{aligned}
        \displaystyle\int_0^T\|\nabla \cdot (u \otimes u)\|_{-1,2}^2 ds  & \leq C\displaystyle\sup_{s\leq T}\|u\|_2^2\displaystyle\int_0^T\|\nabla u\|_2^2ds \leq \|X_0\|_2^2.
    \end{aligned}
\end{equation*}
The nonlinear part in the equation for $h$ is done similarly, using the fact that $u \cdot \nabla h + (\nabla \cdot u)h = \nabla \cdot (hu)$. The control for $\displaystyle\int_0^T\|\tilde{\mathcal{A}}^N(X^N)\|_{-1,2}ds$ holds as requested. For the control of $\displaystyle\int_0^T\|\mathcal{G}^N(X^N)\|_{-1,2}ds$ we proceed in the same manner, using the fact that $\nabla \cdot \sigma = 0$ and therefore we can write $\sigma \cdot \nabla X = \nabla \cdot (\sigma X).$
\end{proof}

\section{Appendix A: Littlewood-Paley approximations}
In this section we introduce the Littlewood-Paley decomposition which is used to construct the approximating sequence of solutions. For more details on this topic in the periodic domain see \cite{Dai}, \cite{Triebelbook}. 

Consider the orthonormal basis $(e_{\lambda})_{\lambda \in \mathbb{Z}^2}$ of $L^2(\mathbb{T}^2)$ with
  $  e_{\lambda}(x) := e^{i\lambda \cdot x} \ \ \ \hbox{for any} \ \ x \in \mathbb{T}^2. $
Any function $f \in L^2(\mathbb{T}^2)$ can be written as 
\begin{equation}\label{fourierseries}
    f(x) = \displaystyle\sum_{\lambda \in \mathbb{Z}^2} \hat{f}(\lambda) e^{i \lambda \cdot x} \ \ \ \hbox{for any} \ \ x \in \mathbb{T}^2,
\end{equation}
where $\hat{f}(\lambda)$ are the Fourier coefficients (or \textit{Fourier modes}) of $f$ given by \footnote{Note that $\hat{f}:\mathbb{Z}^2\rightarrow\mathbb{R}$ is the Fourier transform of $f$. In \eqref{fourierseries} we have the Fourier series of $f$ in $ x \in \mathbb{T}^2$.}
\begin{equation}\label{fouriermodes}
    \hat{f}(\lambda) = \frac{1}{(2\pi)^2}\displaystyle\int_{\mathbb{T}^2} f(y)e^{-i\lambda \cdot y} dy.
\end{equation}
Consider the sequence of subsets $(B_j)_j \subset \mathbb{Z}^2$ given by
\begin{equation*}
B_{j}:= \{ \lambda=(\lambda_1,\lambda_2) \in \mathbb{Z}^2 | |\lambda_m| \leq 2^j, m=1,2 \}.	
\end{equation*}
For any dyadic integer $N=2^j$ one can define $\mathcal{J}_Nf : \mathbb{T}^2 \rightarrow \mathbb{C}$ such that
\begin{equation}\label{dyadicsum}
\begin{aligned}
    \mathcal{J}_Nf(x)&:= \displaystyle\sum_{j=1}^{N}\left(\displaystyle\sum_{\lambda\in B_{j} \smallsetminus B_{j-1}}\hat{f}(\lambda)e^{i\lambda \cdot x}\right) +  \displaystyle\sum_{\lambda\in B_{0} }\hat{f}(\lambda)e^{i\lambda \cdot x}.\\
 \end{aligned}   
\end{equation}
Then $J_jf : \mathbb{T}^2 \rightarrow \mathbb{C}$ given by 
\begin{subequations}\label{dyadicblock0}
\begin{alignat}{2}
& J_0f(x) := \displaystyle\sum_{\lambda\in B_{0} }\hat{f}(\lambda)e^{i\lambda \cdot x} \\
& J_jf(x) := \displaystyle\sum_{\lambda\in B_{j} \smallsetminus B_{j-1}}\hat{f}(\lambda)e^{i\lambda \cdot x} \ \ \  \hbox{for} \ \ j\geq 1
\end{alignat}
\end{subequations}
are the \textit{Littlewood-Paley blocks}
and we can write 
\begin{equation}\label{sumN}
    \mathcal{J}_Nf(x) = \displaystyle\sum_{j=0}^N J_jf(x).
\end{equation}
The \textit{Littlewood-Paley decomposition} of $f \in L^2(\mathbb{T}^2)$ is given by
\begin{equation}\label{LPdecomposition}
    f(x) = \displaystyle\sum_{j=0}^{\infty} J_jf(x).
\end{equation}
The Littlewood-Paley decomposition \eqref{sumN} is based on a decomposition of the function $f$ into a countable sum of functions for which the corresponding Fourier transforms are supported on 'blocks' of size $2^j$. $\mathcal{J}_N$ corresponds to a localisation of the function $f$ to frequencies of order $N$. These blocks provide a decomposition of the frequency space associated with $f$. From \eqref{dyadicsum} we can see that
\begin{equation}
    \mathcal{J}_Nf(x) = \displaystyle\sum_{\lambda \in B_N} e^{i\lambda\cdot x}\hat{f}(\lambda)
\end{equation}
and the corresponding decomposition can also be written as a convolution:
\begin{equation}
    \mathcal{J}_Nf(x) = \left( \displaystyle\sum_{\lambda \in B_N} e^{i\lambda\cdot x} \right) \star f (x).
\end{equation}
An important feature which is more explicitly described using the LP decomposition is the regularity of the functions: this can be characterised in terms of the decay properties of its corresponding dyadic blocks with respect to the summation index. Such a representation allows one to obtain a better control on the derivatives, since the dyadic blocks are orthogonal to one another. More precisely, by Plancherel's identity\footnote{See e.g. Lemma 2.11 in \cite{Dai}}, the Sobolev norm of $f$ can be expressed as 
\begin{equation*}
    \|f\|_{\mathcal{W}^{k,2}(\mathbb{T}^2)} \equiv \|2^{jk}\|J_j f\|_{L^2(\mathbb{T}^2)}\|_{\ell^2(\mathbb{Z}^2)}.
\end{equation*}
By generalising this for blocks from $L^p(\mathbb{T}^2)$ one can introduce the Besov spaces see e.g. \cite{Dai} Definition 2.9. 
Let $\Lambda^s = (-\Delta)^{s/2}, s\geq0$, be the fractional laplacian operator on $\mathbb{T}^2$. 
We introduce below the main properties of the Littlewood-Paley blocks, based on Lemma 2.7 and Proposition 2.8 in \cite{Dai}.\footnote{For complete proofs see \cite{Dai}.} The \textit{Bernstein type inequalities} can be proven using the fact that $\widehat{\Lambda^sf}(\lambda) = |\lambda|^s\hat{f}(\lambda),$ with $\lambda\in\mathbb{Z}^2$. 
\begin{lemma}
Let $f\in\mathcal{W}^{k,2}(\mathbb{T}^2), k\geq 0$ and $J_jf$, $\mathcal{J}_Nf$ as above, with $j\geq 0, N>0$. Then 
\begin{itemize}
    \item [i.] There exist some constants $C_1, C_2$ such that 
    \begin{equation}
        \|J_jf\|_{L^2(\mathbb{T}^2)} \leq C_1\|f\|_{L^2(\mathbb{T}^2)}
    \end{equation}
    \begin{equation}
        \|\mathcal{J}_Nf\|_{L^2(\mathbb{T}^2)} \leq C_2\|f\|_{L^2(\mathbb{T}^2)}
    \end{equation}
\item [ii.] For any $j_1\neq j_2$ one has that
\begin{equation*}
    J_{j_1}J_{j_2}f = 0.
\end{equation*}
\item [iii.] \textit{(Bernstein inequality 1)} There exist some constants $C, \tilde{C}>0$ such that 
\begin{equation}
    \|J_j\Lambda^sf\|_{L^2(\mathbb{T}^2)} \leq C2^{sj}\|J_jf\|_{L^2(\mathbb{T}^2)} \leq \tilde{C}2^{sj}\|f\|_{L^2({\mathbb{T}^2})}
\end{equation}
\item [iv.] \textit{(Bernstein inequality 2)} There exist some constants $0<C_1<C_2$ such that 
\begin{equation}
    C_1 2^{sj}\|J_jf\|_{L^2(\mathbb{T}^2)} \leq \|J_j\Lambda^sf\|_{L^2(\mathbb{T}^2)} \leq C_22^{sj}\|J_jf\|_{L^2(\mathbb{T}^2)}. 
\end{equation}
\end{itemize}
\end{lemma}
\begin{remark} \label{importantLPremark}
Let $\alpha$ be a multi-index of length $|\alpha|=k$.
For $s = 2\alpha $  and $N=2^j$ we have 
\begin{equation*}
\|J_j(\partial^{\alpha}f)\|_{L^2(\mathbb{T}^2)} \leq CN^{2\alpha}\|J_jf\|_{L^2(\mathbb{T}^2)} \leq \tilde{C}N^{2\alpha}\|f\|_{L^2(\mathbb{T}^2)}.
\end{equation*}
Using \eqref{sumN} and \eqref{LPdecomposition} we can write 
\begin{equation*}
\|\partial^{\alpha}f\|_{L^2{\mathbb{T}^2}}^2 \leq 
\displaystyle\sum_{j\geq 1}\|J_j(\partial^{\alpha}f)\|_2^2.
\end{equation*}
Then 
\begin{equation*}
\|f\|_{\mathcal{W}^{k,2}(\mathbb{T}^2)}^2 \leq 
\displaystyle\sum_{j\geq 1}CN ^{2k} \|J_jf\|_{L^2(\mathbb{T}^2)}^2	
\end{equation*}
and
\begin{equation*}
\begin{aligned}
\|f-\mathcal{J}_{N}f\|_{\mathcal{W}^{k,2}(\mathbb{T}^2)}^2 & \leq \displaystyle\sum_{j\geq N+1}\|J_jf\|_{\mathcal{W}^{k,2}(\mathbb{T}^2)}^2 \leq \frac{1}{N}\|f\|_{\mathcal{W}^{k+1,2}(\mathbb{T}^2)}^2.
\end{aligned}	
\end{equation*}
\end{remark}
\section{Appendix B: A priori estimates and Cauchy properties}
In the next lemma we give the proof of \eqref{approxslnestim}.
\begin{lemma}\label{le:approxslnestim}
For any $p>0$ and $R>0$ there exists a constant $C$ independent of $N$ such that 
\begin{equation}\label{approxslnestimappendix}
    \displaystyle\sup_{N \geq 1}\mathbb{E}\left[ \displaystyle\sup_{t\in[0,T]}\|\ckX_t^N\|_{k,2}^p + \left(\displaystyle\int_0^T \|\ckX_t^N\|_{k+1,2}^2dt\right)^{p/2}\right ] \leq C(R).
\end{equation}
\end{lemma}
\begin{proof}
Note that by Jensen's inequality we have for $p \leq 2 $ \footnote{For $p > 2$ see Remark \ref{premark}} we have
\begin{equation*}
\mathbb{E}\left[ \displaystyle\sup_{t\in[0,T]}\|\ckX_t^N\|_{k,2}^p + \left( \displaystyle\int_0^T\|\ckX_t^N\|_{k+1,2}^2dt\right)^{p/2}\right] \leq \left(\mathbb{E} \left[ \displaystyle\sup_{t\in[0,T]}\|\ckX_t^N\|_{k,2}^2 + \displaystyle\int_0^T\|\ckX_t^N\|_{k+1,2}^2dt\right]\right)^{p/2}
\end{equation*}
i.e.
\begin{equation*}
    \mathbb{E}\left[\|\ckX^N\|_{\mathcal{S}_{T,p,(k,2)}}\right] \leq \left(\mathbb{E}\left[\|\ckX^N\|_{\mathcal{S}_{T,2,(k,2)}}\right]\right)^{p/2}.
\end{equation*}

Let $\theta$ be a multi-index of length $k$. By the It\^{o} formula 
\begin{equation*}
    \begin{aligned}
        d\|\partial^{\theta}\ckX_t^N\|_2^2 + 2\gamma\|\partial^{\theta+1 }\ckX_t^N\|_{2}^2dt &= -2\langle \partial^{\theta}\ckX_t^N,\partial^{\theta}\check{\mathcal{A}}^N(\ckX_t^N) \rangle dt\\  
        & -2 \langle \partial^{\theta}\ckX_t^N, \partial^{\theta}\check{\mathcal{G}}^N(\ckX_t^N) \rangle dB_t + \langle \partial^{\theta}\check{\mathcal{G}}^N(\ckX_t^N), \partial^{\theta}\check{\mathcal{G}}^N(\ckX_t^N)\rangle dt.
    \end{aligned}
\end{equation*}
After taking the absolute value, integrating in time and taking the supremum over $s\in[0,t]$ with $0\leq s\leq t\leq T$, and then taking the expectation with respect to $\mathbb{P}$, we have that
\begin{equation*}
    \begin{aligned}
        \mathbb{E}\left[\displaystyle\sup_{s\in[0,t]}\|\ckX_s^N\|_{k,2}^2 + 2\gamma\displaystyle\int_0^t \|\ckX_s^N\|_{k+1,2}^2ds\right] & \leq C \mathbb{E}\left[\|\ckX_0\|_{k,2}^2\right]  + C \mathbb{E}\left[ \displaystyle\int_0^t|\langle \partial^{\theta}\ckX_s^N,\partial^{\theta}\check{\mathcal{A}}^N(\ckX_s^N) \rangle|ds\right]\\
        & + C \mathbb{E}\left[ \left|\displaystyle\int_0^t \langle \partial^{\theta}\ckX_s^N, \partial^{\theta}\check{\mathcal{G}}^N(\ckX_s^N) \rangle  dB_s\right|\right] \\
        & + C\mathbb{E}\left[\displaystyle\int_0^t |\langle \partial^{\theta}\check{\mathcal{G}}^N(\ckX_s^N), \partial^{\theta}\check{\mathcal{G}}^N(\ckX_s^N)\rangle|ds \right]
    \end{aligned}
\end{equation*}
We control each of these terms separately using the results below and then apply the Gronwall lemma to obtain a bound for $\mathbb{E}\left[\|\ckX^N\|_{\mathcal{S}_{T,2,(k,2)}}\right]$. Given the definition of $\check{\mathcal{A}}, \check{\mathcal{G}}, \mathcal{L}, \mathcal{D}$, we need to estimate the following terms\footnote{The constants $C, C_i, i=1,2$ differ from line to line.}: 
The rotation term: 
\begin{equation*}
\begin{aligned}
    T_1 & := |\langle \partial^{\theta}\cku^N,\partial^{\theta}\mathcal{J}_N(fk\times\mathcal{J}_Nu^N) \rangle|  = |\langle \partial^{\theta+1}\mathcal{J}_N\cku^N,\partial^{\theta-1}(fk\times\mathcal{J}_N\cku^N) \rangle| \\
    & \leq C_1(R,\nu)\|\mathcal{J}_N\cku^N\|_{k+1,2}^2 + C_2(R,\nu)\|\mathcal{J}_N\cku^N\|_{k,2}^2  \\
    & \leq C_1(\nu)\frac{R}{2}\|\mathcal{J}_N\cku^N\|_{k+1,2}^2 + C_2(\nu)\frac{R}{2}\|\mathcal{J}_N\cku^N\|_{k,2}^2
\end{aligned}
\end{equation*}
The advection term, which contains also the It\^{o}-Stokes drift: 
\begin{equation*}
\begin{aligned}
    T_2 &:= |\langle \partial^{\theta}\cku^N, f_R(\cku^N)\partial^{\theta}\mathcal{J}_N({\mathcal{L}}_{\mathcal{J}_N\cku^N}\mathcal{J}_N\cku^N) + \partial^{\theta}\mathcal{J}_N({\mathcal{L}}_{\mathcal{J}_N\cku^{S,N}}\mathcal{J}_N\cku^N)  \rangle| 
\end{aligned}    
\end{equation*}
    Taking into account the fact that $\tilde{\mathcal{L}}_uu = (u-u^S)\cdot \nabla u$ we need to estimate two quite similar terms:
\begin{equation*}
    \begin{aligned}
    T_{21} &: = |\langle \partial^{\theta}\cku^N, f_R(\cku^N)\partial^{\theta}\mathcal{J}_N(\mathcal{L}_{\mathcal{J}_N\cku^N}\mathcal{J}_N\cku^N) \rangle| \\
    & = |\langle \partial^{\theta+1}\cku^N, f_R(\cku^N)\partial^{\theta - 1}\mathcal{J}_N(\mathcal{L}_{\mathcal{J}_N\cku^N}\mathcal{J}_N\cku^N) \rangle| \\
    & = Cf_R(\cku^N)\left |\displaystyle\sum_{\beta\leq\theta-1}\langle \partial^{\theta+1}\cku^N, f_R(\cku^N) \partial^{\beta}(\mathcal{J}_N\cku^N)\partial^{\theta-1-\beta}(\nabla\mathcal{J}_N\cku^N)\rangle \right | \\
    & \leq Cf_R(\cku^N)\|\mathcal{J}_N\cku^N\|_{k+1,2}\displaystyle\sum_{\beta\leq\theta-1}\|\partial^{\beta}\mathcal{J}_N\cku^N\|_4 \|\partial^{\theta-1-\beta}(\nabla\mathcal{J}_N\cku^N)\|_4 \\
    & \leq C(R)\|\mathcal{J}_N\cku^N\|_{k+1,2}\|\mathcal{J}_N\cku^N\|_{k,2} \\
    & \leq C_1(R,\nu)\|\mathcal{J}_N\cku^N\|_{k+1,2}^2 + C_2(R,\nu)\|\mathcal{J}_N\cku^N\|_{k,2}^2 \\
    & \leq C_1(\nu)\frac{R}{2}\|\mathcal{J}_N\cku^N\|_{k+1,2}^2 + C_2(\nu)\frac{R}{2}\|\mathcal{J}_N\cku^N\|_{k,2}^2.
    \end{aligned}
\end{equation*}
With exactly the same arguments but without moving one derivative on the left term we have 
\begin{equation*}
    \begin{aligned}
    T_{22} &: = |-\langle \partial^{\theta}\cku^N, \partial^{\theta}\mathcal{J}_N(\mathcal{L}_{\mathcal{J}_N\cku^{S,N}}\mathcal{J}_N\cku^N) \rangle| \\    
    & \leq C\|\mathcal{J}_N\cku^N\|_{k,2}\|\mathcal{J}_N\cku^{S,N}\|_{k,2}\|\mathcal{J}_N\cku^N\|_{k,2} \\
    & \leq C(\|a\|_{k+1,2})\|\mathcal{J}_N\cku^N\|_{k,2}^2
    \end{aligned}
\end{equation*}

So we have
\begin{equation*}
    T_2 \leq C_1(\nu)\frac{R}{2}\|\mathcal{J}_N\cku^N\|_{k+1,2}^2 + C_2(\nu)\frac{R}{2}\|\mathcal{J}_N\cku^N\|_{k,2}^2 + C(a)\|\mathcal{J}_N\cku^N\|_{k,2}^2 
\end{equation*}
For the third term we have 

\begin{equation*}
\begin{aligned}
    T_3 &:= (1+\alpha) |\langle \partial^{\theta}\cku^N, f_R(\cku^N)\partial^{\theta}\mathcal{J}_N(\mathcal{D}_{\mathcal{J}_N\cku^N}\mathcal{J}_N\cku^N)\rangle| \\
    & = (1+\alpha)|\langle \partial^{\theta+1}\cku^N, f_R(\cku^N)\partial^{\theta -1}\mathcal{J}_N(\mathcal{D}_{\mathcal{J}_N\cku^N}\mathcal{J}_N\cku^N)\rangle|\\ 
    & = (1+\alpha)Cf_R(\cku^N)\left |\displaystyle\sum_{\beta\leq\theta-1}\langle \partial^{\theta+1}\cku^N, f_R(\cku^N) \partial^{\beta}(\nabla \cdot \mathcal{J}_N\cku^N)\partial^{\theta-1-\beta}(\mathcal{J}_N\cku^N)\rangle \right | \\
    & \leq (1+\alpha)Cf_R(\cku^N)\|\mathcal{J}_N\cku^N\|_{k+1,2}\displaystyle\sum_{\beta\leq\theta-1}\|\partial^{\beta}(\nabla \cdot \mathcal{J}_N\cku^N)\|_4 \|\partial^{\theta-1-\beta}(\mathcal{J}_N\cku^N)\|_4 \\
    & \leq (1+\alpha)C(R)\|\mathcal{J}_N\cku^N\|_{k+1,2}\|\mathcal{J}_N\cku^N\|_{k,2} \\
    & \leq (1+\alpha)\left(C_1(R,\nu)\|\mathcal{J}_N\cku^N\|_{k+1,2}^2 + C_2(R,\nu)\|\mathcal{J}_N\cku^N\|_{k,2}^2\right) \\
    & \leq (1+\alpha)\left(C_1(\nu)\frac{R}{2}\|\mathcal{J}_N\cku^N\|_{k+1,2}^2 + C_2(\nu)\frac{R}{2}\|\mathcal{J}_N\cku^N\|_{k,2}^2\right).
\end{aligned}
\end{equation*}
For the fourth term we have 
\begin{equation*}
\begin{aligned}
    T_4 &:= |\langle \partial^{\theta}\cku^N, \partial^{\theta}\mathcal{J}_N(\nabla (\mathcal{J}_N \ckh^N))\rangle| \\
    & \leq C \|\mathcal{J}_N\cku^N\|_{k,2}\|\mathcal{J}_N\ckh^N\|_{k+1,2}\\
    & \leq C_1(\eta)\|\mathcal{J}_N\cku^N\|_{k,2}^2 +  C_2(\eta)\|\mathcal{J}_N\ckh^N\|_{k+1,2}^2 
\end{aligned}
\end{equation*}
We combine $T_5$ 
\begin{equation*}
\begin{aligned}
T_5 &:= |\langle \partial^{\theta}\cku^N, \partial^{\theta}\mathcal{J}_N\left(\mathcal{L}_{\sigma}^2\mathcal{J}_N\cku^N\right)\rangle|
\end{aligned}
\end{equation*}
and $T_8$
\begin{equation*}
\begin{aligned}
T_8 &:= |\langle \partial^{\theta}\ckh^N, \partial^{\theta}\mathcal{J}_N\left(\mathcal{L}_{\sigma}^2\mathcal{J}_N\ckh^N\right)\rangle|
\end{aligned}
\end{equation*}
with the quadratic variation term to obtain
\begin{equation*}
    \begin{aligned}
       \left| \langle \partial^{\theta}\ckX^N, \partial^{\theta}\mathcal{J}_N\left(\mathcal{L}_{\sigma}^2\mathcal{J}_N\ckX^N\right)\rangle + \langle \partial^{\theta}\mathcal{G}^N(\ckX_s^N), \partial^{\theta}\mathcal{G}^N(\ckX_s^N)\rangle\right| \leq C\|\ckX^N\|_{k,2}^2.
    \end{aligned}
\end{equation*}
The control for $T_6$ is similar to the one for $T_2$ and we have 
\begin{equation*}
\begin{aligned}
    T_6 &:= |\langle \partial^{\theta}\ckh^N, f_R(\ckh^N)\partial^{\theta}\mathcal{J}_N({\mathcal{L}}_{\mathcal{J}_N\cku^N}\mathcal{J}_N\ckh^N) + \partial^{\theta}\mathcal{J}_N({\mathcal{L}}_{\mathcal{J}_N\cku^{S,N}}\mathcal{J}_N\ckh^N)  \rangle| \\
    & \leq C_1(R,\eta)\|\mathcal{J}_N\ckh^N\|_{k+1,2}^2 + C_2(R,\eta)\|\mathcal{J}_N\cku^N\|_{k,2}^2 + C(a)\|\mathcal{J}_N\ckh^N\|_{k,2}^2 \\
    & \leq C_1(\eta)\frac{R}{2}\|\mathcal{J}_N\ckh^N\|_{k+1,2}^2 + C_2(\eta)\frac{R}{2}\|\mathcal{J}_N\cku^N\|_{k,2}^2 + C(a)\|\mathcal{J}_N\ckh^N\|_{k,2}^2.
\end{aligned}
\end{equation*}
The control for $T_7$ is similar to the one for $T_3$ and we have
\begin{equation*}
\begin{aligned}
    T_7 &:= (1+\beta)|\langle \partial^{\theta}\ckh^N, f_R(\ckh^N)\partial^{\theta}\mathcal{J}_N(\mathcal{D}_{\mathcal{J}_N\cku^N}\mathcal{J}_N\ckh^N)\rangle| \\
    & \leq (1+\beta)C(R)\|\mathcal{J}_N\ckh^N\|_{k+1,2}\|\mathcal{J}_N\cku^N\|_{k,2} \\
    & \leq (1+\beta)\left(C_1(R,\eta)\|\mathcal{J}_N\ckh^N\|_{k+1,2}^2 + C_2(R,\eta)\|\mathcal{J}_N\cku^N\|_{k,2}^2 \right)\\
    & \leq (1+\beta)\left(C_1(\eta)\frac{R}{2}\|\mathcal{J}_N\ckh^N\|_{k+1,2}^2 + C_2(\eta)\frac{R}{2}\|\mathcal{J}_N\cku^N\|_{k,2}^2\right).
\end{aligned}
\end{equation*}
For the stochastic integrals we use the fact that
\begin{equation*}
    \begin{aligned}
        |\langle \partial^{\theta}\ckX_r^N, \partial^{\theta}\check{\mathcal{G}}^N(\ckX_r^N)\rangle| 
        & \leq C\|\ckX_r^N\|_{k,2}^2.
    \end{aligned}
\end{equation*}
This is true since $\nabla \cdot \sigma = 0$ and there exists $C$ such that $\|\sigma\|_{k,2} \leq C$, by the initial assumptions. 
Then we apply the Burkholder-Davis-Gundy inequality to obtain
\begin{equation*}
    \begin{aligned}
        \mathbb{E}\left[ \displaystyle\sup_{s\in[0,t]}\left|\displaystyle\int_0^s \langle \partial^{\theta}\ckX_r^N, \partial^{\theta}\check{\mathcal{G}}^N(\ckX_r^N) \rangle dB_r\right|\right] & \leq C\mathbb{E}\left[ \displaystyle\int_0^s |\langle \partial^{\theta}\ckX_r^N, \partial^{\theta}\check{\mathcal{G}}^N(\ckX_r^N) \rangle|^2ds\right]^{1/2} \\
        & \leq C\mathbb{E}\left[ \displaystyle\sup_{r\in[0,s]}\|\ckX_r^N\|_{k,2}^2 \right] + C(T). 
    \end{aligned}
\end{equation*}
Note that all these bounds are independent of $N$. Summing up we have
\begin{equation*}
    \begin{aligned}
     \mathbb{E} & \left[  \displaystyle\sup_{s\in[0,t]} \|\ckX_s^N\|_{k,2}^2  + 2\gamma\displaystyle\int_0^t\|\ckX_s^N\|_{k+1,2}^2ds \right]  \leq \mathbb{E}\left[\|\ckX_0\|_{k,2}^2\right] \\
        & + C_1(R,\nu)\mathbb{E}\left[\displaystyle\int_0^t  \|\mathcal{J}_N\cku_s^N\|_{k+1,2}^2 ds\right]+ C_2(R,\nu)\mathbb{E}\left[\displaystyle\sup_{s\in[0,t]}\|\mathcal{J}_N\cku_s^N\|_{k,2}^2 \right]+ C(a)\mathbb{E}\left[\displaystyle\sup_{s\in[0,t]}\|\mathcal{J}_N\cku_s^N\|_{k,2}^2 ds \right]\\
        & + C_1(R,\nu)\mathbb{E}\left[\displaystyle\int_0^t \|\mathcal{J}_N\cku_s^N\|_{k+1,2}^2 ds\right]+ C_2(R,\nu)\mathbb{E}\left[\displaystyle\sup_{s\in[0,t]}\|\mathcal{J}_N\cku_s^N\|_{k,2}^2 \right]\\
        & + C_1(\eta)\mathbb{E}\left[ \displaystyle\sup_{s\in[0,t]} \|\mathcal{J}_N\cku_s^N\|_{k,2}^2 \right]+  C_2(\eta)\mathbb{E}\left[\displaystyle\int_0^t\|\mathcal{J}_N\ckh_s^N\|_{k+1,2}^2 ds \right]\\
        & + C_1(\|a\|_{k+1,2},\nu)\mathbb{E}\left[ \displaystyle\int_0^t \|\mathcal{J}_N\cku_s^N\|_{k+1,2}^2ds\right] + C_2(\|a\|_{k+1,2},\nu)\mathbb{E}\left[\displaystyle\sup_{s\in[0,t]}\|\mathcal{J}_N\cku_s^N\|_{k,2}^2 \right]\\
        & + C_1(R,\eta)\mathbb{E}\left[\displaystyle\int_0^t \|\mathcal{J}_N\ckh_s^N\|_{k+1,2}^2 ds\right]+ C_2(R,\eta)\mathbb{E}\left[\displaystyle\sup_{s\in[0,t]}\|\mathcal{J}_N\cku_s^N\|_{k,2}^2 \right]+ C(a)\mathbb{E}\left[\displaystyle\sup_{s\in[0,t]}\|\mathcal{J}_N\ckh_s^N\|_{k,2}^2 ds\right] \\
        & + C_1(R,\eta)\mathbb{E}\left[\displaystyle\int_0^t \|\mathcal{J}_N\ckh_s^N\|_{k+1,2}^2 ds\right]+ C_2(R,\eta)\mathbb{E}\left[\displaystyle\sup_{s\in[0,t]}\|\mathcal{J}_N\cku_s^N\|_{k,2}^2 ds \right]\\
        & + C_1(\|a\|_{k+1,2},\eta)\mathbb{E}\left[\displaystyle\int_0^t \|\mathcal{J}_N\ckh_s^N\|_{k+1,2}^2ds \right]+ C_2(\|a\|_{k+1,2},\eta)\mathbb{E}\left[\displaystyle\sup_{s\in[0,t]}\|\mathcal{J}_N\ckh_s^N\|_{k,2}^2  \right] \\
        & + C_1(\gamma,\delta)\mathbb{E}\left[ \displaystyle\sup_{r\in[0,s]}\|\ckX_r^N\|_{k,2}^2\right] + C_2(\gamma,\delta)\mathbb{E}\left[\displaystyle\int_0^s\|\ckX_r^N\|_{k+1,2}^2dr \right] \\
        & + C(\gamma)\mathbb{E}\left[ \displaystyle\int_0^t \|\ckX_s^N\|_{k+1,2}^2 ds\right] + C(T)
    \end{aligned}
\end{equation*}
with $\delta = (\alpha,\beta)$.
By choosing the constants to depend appropriately on $\nu$ and $\eta$ we have
\begin{equation*}
    \begin{aligned}
        \mathbb{E}  & \left[  \displaystyle\sup_{s\in[0,t]} \|\ckX_s^N\|_{k,2}^2  + 2\gamma\displaystyle\int_0^t\|\ckX_s^N\|_{k+1,2}^2ds \right]   \leq \|\ckX_0\|_{k,2}^2 \\
        & + \tilde{C}_1(R,\nu)\mathbb{E}\left[ \displaystyle\int_0^t \|\mathcal{J}_N\cku^N\|_{k+1,2}^2ds\right] + \tilde{C}_1(R,\eta)\mathbb{E}\left[ \displaystyle\int_0^t \|\mathcal{J}_N\ckh^N\|_{k+1,2}^2ds\right] \\
        & + \tilde{C}_2(R,\nu)\mathbb{E}\left[ \displaystyle\sup_{s\in[0,t]}\|\mathcal{J}_N\cku^N\|_{k,2}^2\right] + \tilde{C}_2(R,\eta)\mathbb{E}\left[ \displaystyle\sup_{s\in[0,t]}\|\mathcal{J}_N\ckh^N\|_{k,2}^2\right] \\
        & + \tilde{C}_1(\gamma)\mathbb{E}\left[ \displaystyle\int_0^t \|\mathcal{J}_N\ckX^N\|_{k+1,2}^2ds\right] + \tilde{C}_2(\gamma)\mathbb{E}\left[ \displaystyle\sup_{s\in[0,t]}\|\mathcal{J}_N\ckX^N\|_{k,2}^2\right] + C(T)\\
        & \leq \|\ckX_0\|_{k,2}^2+ C(T) +  \dbtilde{C}_1(R,\gamma)\mathbb{E}\left[ \displaystyle\int_0^t\|\mathcal{J}_N\ckX^N\|_{k+1,2}^2ds\right] + \dbtilde{C}_2(R,\gamma)\mathbb{E}\left[ \displaystyle\sup_{s\in[0,t]}\|\mathcal{J}_N\ckX^N\|_{k,2}^2\right].
    \end{aligned}
\end{equation*}
By Gronwall lemma and using the fact that $\mathcal{J}_N\ckX^N = \ckX^N$ (since $\ckX^N$ is already in the \textit{projected} space) we deduce that there exists a constant $C=C(R)$ such that 
\begin{equation*}
    \mathbb{E}\left[\|\ckX^N\|_{\mathcal{S}_{T,2,(k,2)}}\right] \leq C(R)
\end{equation*}
and therefore
\eqref{approxslnestimappendix} holds. 
\end{proof}

\begin{remark}\label{premark}
We provided above the control of all terms in $L^2({\mathbb{T}^2})$. 
When $p\geq 2$ one needs to apply the It\^{o} formula for the map $x\rightarrow x^p$. The same type of control is then needed in $L^2({\mathbb{T}^2})$ but an extra term $\|X^N\|_{k,2}^{p-2}$ needs to be carefully taken into account when applying the Gronwall lemma. However, no other technical differences appear, and therefore we do not redo the calculations for this case. An exact application of the It\^{o} formula  for $x^p$ when $p \geq 2$ can be found for instance in \cite{Krylov}.
\end{remark}
\begin{proposition}(\textit{Cauchy property})\label{lemma:cauchy}
The approximating system $(\check{X}^N)_N$ is Cauchy in 
$$L^2(\Omega, C([0,T], L^2(\T^2)^3) \cap L^2([0,T], \mathcal{W}^{1,2}(\T^2)^3))  $$ and 
 \begin{equation*}\label{eq:cauchyrelation}
   \displaystyle\lim_{N\rightarrow \infty}\displaystyle\sup_{M\geq N}\mathbb{E}\left[\displaystyle\sup_{t\in [0, T]}\|\check{X}_t^N - \check{X}_t^M\|_2^p\right] +\mathbb{E}\left[\left(\int_{0}^{T}\|\check{X}_s^N - \check{X}_s^M\|_{1,2}^2ds\right)^{p/2}\right]= 0
   \end{equation*}
where $\check{X}_t^N = (\check{u}_t^N, \check{h}_t^N)$ is the truncated approximating solution.
\end{proposition}
\begin{proof}[\textbf{Proof of Proposition \ref{lemma:cauchy}}.]
We first show that
\begin{equation}
   \displaystyle\lim_{N\rightarrow \infty}\displaystyle\sup_{M\geq N}\mathbb{E}\left[\displaystyle\sup_{t\in [0, T]}\|\check{X}_t^N - \check{X}_t^M\|_2^2\right] +\mathbb{E}\left[ \int_{0}^{t}\|\check{X}_s^N - \check{X}_s^M\|_{1,2}^2ds \right] = 0
   \end{equation}
and then by Jensen's inequality \eqref{eq:cauchyrelation} holds. 
\noindent To further shorten the formulae we can write
  \begin{equation}
     d\check{X}_t^N = \left(F_t^{N}(\check{X}_t^N)- K_t^{N}(\check{X}_t^N)-L_t^{N}(\check{X}_t^N)\right)dt-W_t^{N}(\check{X}_t^N)dB_t
 \end{equation}
 i.e. 
\begin{subequations}
 \begin{alignat}{2}
& d\check{u}_t^N = (F_t^{N}(\check{u}_t^N)- K_t^{N}(\check{u}_t^N)+L_t^{N}(\check{u}_t^N))dt-W_t^{N}(\check{u}_t^N)dB_t\\
& d\check{h}_t^N = (F_t^{N}(\check{h}_t^N)-K_t^{N}(\check{h}_t^N)+L_t^{N}(\check{h}_t^N))dt-W_t^{N}(\check{h}_t^N)dB_t
\end{alignat}
\end{subequations}
where
\begin{equation*}
    F_t^{N}(\check{X}_t^N):= \gamma\Delta \check{X}_t^N
\end{equation*}
i.e. 
\begin{equation*}
\begin{aligned}
& F_t^{N}(\check{u}^N):= \nu\Delta \check{u}_t^N, \ \ \   F_t^{N}(\check{h}_t^N):=\eta \Delta \check{h}_t^N
\end{aligned}
\end{equation*}
and $\gamma = (\nu, \eta).$
Also
\begin{equation*}
    K_t^{N}(\check{X}_t^N):= K_{t,1}^{N}(\check{X}_t^N)- K_{t,2}^{N}(\check{X}_t^N) 
\end{equation*}
where\footnote{When not necessary, we will omit the subscript $t$ in what follows.}:
\begin{equation*}
\begin{aligned}
    K_1^N(\check{X}^N) & = f_R(\check{X}^N)\left( \mathcal{J}_{N}\left(\mathcal{{L}}_{\mathcal{J}_{N}\check{u}^{N}}+(1+\delta)\cD_{\cJ_N \check{u}^N}\right) \mathcal{J}_{N}\check{X}^{N}\right) \\
    & = f_R(\check{X}^N)\mathcal{J}_N \nabla \cdot (\mathcal{J}_N\check{u}^N\mathcal{J}_N\check{X}^N) + \delta  f_R(\check{X}^N) \cD_{\cJ_N \check{u}^N} \mathcal{J}_{N}\check{X}^{N}
\end{aligned}
\end{equation*}
and $\delta =(\alpha, \beta).$

More precisely: 
\begin{equation*}
\begin{aligned}
 K_1^{N}(\check{u}^N) & :=f_R(\check{u}^N)\left( \mathcal{J}_{N}\left(\mathcal{{L}}_{\mathcal{J}_{N}\check{u}^{N}}+(1+\alpha)\cD_{\cJ_N \check{u}^N}\right) \mathcal{J}_{N}\check{u}^{N}\right) \\
 & = f_R(\check{X}^N)\mathcal{J}_N \nabla \cdot (\mathcal{J}_N\check{u}^N\mathcal{J}_N(\check{u}^N)^{T})  + \alpha f_R(\check{X}^N)\cD_{\cJ_N \check{u}^N} \mathcal{J}_{N}\check{u}^{N}\\
 &= f_R(\check{X}^N)\mathcal{J}_N \nabla \cdot (\mathcal{J}_N\check{u}^N \otimes  \mathcal{J}_N\check{u}^N ) + \alpha f_R(\check{X}^N)\cD_{\cJ_N \check{u}^N} \mathcal{J}_{N}\check{u}^{N}\\
 & =: K_{11}^{N}(\check{u}^N) + \alpha K_{12}^{N}(\check{u}^N)
 \end{aligned}
 \end{equation*}
 \begin{equation*}
 \begin{aligned}
 K_1^{N}(\check{h}) &:=f_R(\check{h}^N)\left(\mathcal{J}_{N}\left(\mathcal{{L}}_{\mathcal{J}_{N}\check{u}^{N}} + (1+\beta)\cD_{\cJ_N \check{u}^N}\right) \mathcal{J}_{N}\check{h}^{N}\right) \\
 & = f_R(\check{X}^N)\mathcal{J}_N \nabla \cdot (\mathcal{J}_N\check{u}^N\mathcal{J}_N\check{h}^N) + \beta f_R(\check{X}^N) \cD_{\cJ_N \check{u}^N} \mathcal{J}_{N}\check{h}^{N} \\
 & =: K_{11}^{N}(\check{h}^N) + \beta K_{12}^{N}(\check{h}^N)
\end{aligned}
\end{equation*}

\begin{equation*}
\begin{aligned}
& K_2^{N}(\check{u}^N):= \mathcal{J}_{N}(\mathcal{{L}}_{\mathcal{J}_{N}u^{S,N}} \mathcal{J}_{N}\check{u}^{N}), \ \ \  K_2^{N}(\check{h}^N):= \mathcal{J}_{N}(\mathcal{{L}}_{\mathcal{J}_{N}u^{S,N}} \mathcal{J}_{N}\check{h}^{N})\ \ \ 
\end{aligned}
\end{equation*}
Finally 
\begin{equation*}
    W^{N}(\check{X}^N):= \mathcal{J}_{N}\left(\mathcal{L}_{\sigma}\mathcal{J}_{N}\check{X}^{N}\right)
\end{equation*}
\begin{equation*}
\begin{aligned}
& W^{N}(\check{u}^N):= \mathcal{J}_{N}\left(\mathcal{L}_{\sigma}\mathcal{J}_{N}\check{u}^{N}\right), \ \ \   W^{N}(\check{h}^N):=\mathcal{J}_{N}\left(\mathcal{L}_{\sigma}\mathcal{J}_{N}\check{h}^{N}\right)
\end{aligned}
\end{equation*}

\begin{equation*}
 L^{N}(\check{X}^N):= - \frac{1}{2}\cJ_N(\nabla \cdot (a\nabla(\cJ_N \check{X}^N)))   + \left( \begin{array}{c}
\mathcal{J}_Ng\nabla (\mathcal{J}_Nh^N ) + fk\times \mathcal{J}_Nu^N \\ 
0
\end{array}%
\right)
\end{equation*}
that is 
\begin{equation*}
\begin{aligned}
& L^{N}(\check{u}^N):= - \frac{1}{2}\cJ_N(\nabla \cdot (a\nabla(\cJ_N \check{u}^N))) + \mathcal{J}_Ng\nabla (\mathcal{J}_Nh^N ) + fk\times \mathcal{J}_Nu^N\\  
& L^{N}(\check{h}^N):= - \frac{1}{2}\cJ_N(\nabla \cdot (a\nabla(\cJ_N \check{h}^N)))
\end{aligned}
\end{equation*}

We need to control the difference between $\check{X}_t^N - \check{X}_t^M$. More precisely, we have:
\begin{equation}\label{basicinequality}
    \begin{aligned}
    d\|\check{X}_t^N - \check{X}_t^M\|_2^2 + 2\gamma\|\check{X}_t^N - \check{X}_t^M\|_{1,2}^2dt  
    &\leq 2|\langle \check{X}_t^N - \check{X}_t^M,K_t^{N}(\check{X}_t^N)- K_t^{M}(\check{X}_t^M) \rangle| dt\\
    & +2|\langle \check{X}_t^N - \check{X}_t^M,L_t^{N}(\check{X}_t^N)- L_t^{M}(\check{X}_t^M) \rangle |dt\\
    & +2|\langle \check{X}_t^N - \check{X}_t^M, W_t^{N}(\check{X}_t^N)- W_t^{M}(\check{X}_t^M) \rangle |dB_t \\
    & + |\langle W_t^{N}(\check{X}_t^N)- W_t^{M}(\check{X}_t^M), W_t^{N}(\check{X}_t^N)- W_t^{M}(\check{X}_t^M)\rangle |dt \\
   &  =: \Gamma^1 + \Gamma^2 + \Gamma^3 + \Gamma^4 
    \end{aligned}
\end{equation}
where $\Gamma^i, i\in \{1,2,3,4\}$, correspond, respectively, to the absolute values of the terms above. Note that $\mathbb{E}[\Gamma^3] = 0$. Based on the lemmas below we have
\begin{equation*}
    \begin{aligned}
        \mathbb{E}\left[ \Gamma^1 \right] & \leq \tilde{C}_1(\epsilon)\mathbb{E}\left[\|\check{X}_t^N - \check{X}_t^M\|_2^2\right] + \tilde{C}_2(\epsilon)\mathbb{E}\left[\|\check{X}_t^N - \check{X}_t^M\|_{k+1,2}^2\right]\\
        & +  \delta \left (C(\epsilon)\mathbb{E}\left[\|\check{X}_t^N - \check{X}_t^M\|_2^2\right] +\epsilon \mathbb{E}\left[\bar{T}_{\ckX^N}^{\cD}\right] \right) \\
        & + C(\epsilon) \mathbb{E}\left[\|\check{X}_t^N - \check{X}_t^M\|_2^2\right] + \frac{C(a, \epsilon)}{M\wedge N}\mathbb{E}\left[\|\ckX^M\|_{k,2}^2 \right] \\
        & \leq  \tilde{C}_1(\epsilon)\mathbb{E}\left[\|\check{X}_t^N - \check{X}_t^M\|_2^2\right] + \tilde{C}_2(\epsilon)\left(\mathbb{E}\left[\|\check{X}_t^N\|_{k+1,2}^2\right] - \mathbb{E}\left[\|\check{X}_t^M\|_{k+1,2}^2\right]\right)\\
        & + C_1(\delta,\epsilon)\mathbb{E}\left[\|\check{X}_t^N - \check{X}_t^M\|_2^2\right] + C_2(\delta, \epsilon)\mathbb{E}\left[\bar{T}_{\ckX^N}^{\cD,1} + \bar{T}_{\ckX^N}^{\cD,2} \right] \\
        & + C(\epsilon) \mathbb{E}\left[\|\check{X}_t^N - \check{X}_t^M\|_2^2\right] + \frac{C(a, \epsilon)}{M\wedge N}\mathbb{E}\left[\|\ckX^M\|_{k,2}^2 \right] \\
        & \leq C(\delta,\epsilon)\mathbb{E}\left[\|\check{X}_t^N - \check{X}_t^M\|_2^2\right] + \frac{C(a,\epsilon)}{M\wedge N}\mathbb{E}\left[\|\ckX^M\|_{k,2}^2\right] \\
        & \leq C(\delta,\epsilon)\mathbb{E}\left[\|\check{X}_t^N - \check{X}_t^M\|_2^2\right] + \frac{C^{\Gamma^1}}{M\wedge N}
    \end{aligned}
\end{equation*}
We used above that, in absolute value,
\begin{equation*}
    K \leq K_{11} + \delta K_{12} + K_2
\end{equation*}
more precisely
\begin{equation*}
    \begin{aligned}
    |K^N(\ckX^N) - K^M(\ckX^M)| \leq |K_{11}^N(\ckX^N) - K_{11}^M(\ckX^M)| + \delta |K_{12}^N(\ckX^N) - K_{12}^M(\ckX^M)| + |K_2^M(\ckX^M) - K_2^N(\ckX^N)|
    \end{aligned}
\end{equation*}
and that $\bar{T}_{\ckX^N}^{\cD} = \bar{T}_{\ckX^N}^{\cD,1} + \bar{T}_{\ckX^N}^{\cD,2}$ is estimated as follows
\begin{equation*}
   \begin{aligned}
    \mathbb{E}\left[\bar{T}_{\ckX^N}^{\cD,1} \right] & \leq C_1(\epsilon)\mathbb{E}\left[ \|\check{X}_t^N - \check{X}_t^M\|_2^2\right] + C_2(\epsilon)\mathbb{E}\left[\|\ckX_t^N - \ckX_t^M\|_{k+1,2}^2\right] \\
    &\leq  C_1(\epsilon)\mathbb{E}\left[ \|\check{X}_t^N - \check{X}_t^M\|_2^2\right] + C_2(\epsilon)\left( \mathbb{E}\left[\|\ckX^N\|_{k+1,2}^2\right] + \mathbb{E}\left[\|\ckX^M\|_{k+1,2}^2\right] \right) \\
    & \leq C_1(\epsilon)\mathbb{E}\left[ \|\check{X}_t^N - \check{X}_t^M\|_2^2\right] + \tilde{C}_2(\epsilon)
    \end{aligned}
\end{equation*}
and 
\begin{equation*}
    \mathbb{E}\left[\bar{T}_{\ckX^N}^{\cD,2} \right] \leq \frac{C}{M\wedge N}\mathbb{E}[\|\ckX^M\|_{k,2}^2] \leq \frac{\tilde{C}}{M\wedge N}
\end{equation*}
as per Lemma \ref{lemma:estimgeneral} below. 
Next we have
\begin{equation*}
    \begin{aligned}
        \mathbb{E}\left[ \Gamma^2 \right] &\leq \frac{\tilde{C}_1(a)}{M\wedge N}\mathbb{E}\left[ \|\ckX^N - \ckX^M\|_2^2 \right] + \frac{\tilde{C}_2(a)}{N\wedge M}\mathbb{E}\left[\|\ckX^N\|_{k,2}^2\right]\\
        & \leq \frac{\tilde{C}(a)}{M\wedge N} \left( \mathbb{E}\left[ \|\ckX^N\|_{k,2}^2\right] + \mathbb{E}\left[ \|\ckX^N\|_{k,2}^2\right] \right)\\ 
        & \leq \frac{C^{\Gamma^2}}{M\wedge N}
    \end{aligned}
\end{equation*}
and 
\begin{equation*}
    \begin{aligned}
        \mathbb{E}\left[ \Gamma^4 \right] & \leq \frac{C(a)}{M\wedge N}\left( \mathbb{E}\left[ \|\ckX^N\|_{k,2}^2\right]  + \mathbb{E}\left[ \|\ckX^M\|_{k,2}^2\right]\right) \\
        & \leq \frac{C^{\Gamma^4}}{M\wedge N}
    \end{aligned}
\end{equation*}

Then 
\begin{equation*}
    \begin{aligned}
     \mathbb{E}\left[\|\check{X}_t^N - \check{X}_t^M\|_2^2\right]  + 2\gamma\mathbb{E}\left[\displaystyle\int_0^t\|\check{X}_s^N - \check{X}_s^M\|_{1,2}^2\right]ds & \leq \|\ckX_0^N - \ckX_0^M\|_2^2+  C_1(\delta,\epsilon)\displaystyle\int_0^t\mathbb{E}\left[\|\check{X}_s^N - \check{X}_s^M\|_2^2\right]ds  + \frac{C^{\Gamma}T}{M\wedge N}
    \end{aligned}
\end{equation*}
where $\Gamma = \sum_{i=1}^4 \Gamma^i$. By Gronwall lemma
\begin{equation}\label{simplecontrol}
    \begin{aligned}
     \mathbb{E}\left[\|\check{X}_t^N - \check{X}_t^M\|_2^2\right] & \leq \left(\|\ckX_0^N - \ckX_0^M\|_2^2  + \frac{C^{\Gamma}T}{M\wedge N}\right) e^{C_1(\delta,\epsilon)T}.
    \end{aligned}
\end{equation}
 By taking the supremum for $t\in [0,T]$ in \ref{basicinequality}, using the Burkholder-Davis-Gundy inequality to control the stochastic integrals, we deduce from (\ref{simplecontrol}) that
\begin{equation*}
    \begin{aligned}
     \displaystyle\lim_{N\rightarrow\infty}\displaystyle\sup_{M\geq N}\mathbb{E}\left[\displaystyle\sup_{t\in[0,T]}\|\check{X}_t^N - \check{X}_t^M\|_2^2\right] & \leq \displaystyle\lim_{N\rightarrow\infty}\displaystyle\sup_{M\geq N}\left(\|\ckX_0^N - \ckX_0^M\|_2^2  + \frac{C^{\Gamma}T}{M\wedge N}\right) e^{C_1(\delta,\epsilon)T} \\
     &=0
    \end{aligned}
\end{equation*}
where we assumed that the initial sequence is convergent. 
   The analysis of each $\Gamma^i$ is quite tedious and we will take care of each of the terms separately. 
   We first look at the nonlinear terms and then at the more simple linear ones. This is all done in the two Lemmata below. \\
\end{proof}

\begin{lemma}\label{lemma:estimgeneral}
With $F^{N,M}, K^{N,M}, L^{N,M}, W^{N,M}$ and $\Gamma^i, i \in\{1,2,3,4\}$, defined as above, the following hold: \\
A. There exist some constants $C(\epsilon, R, \delta),C(\epsilon,\gamma,\delta),C(k,R,a,\epsilon,\delta)$ such that 
\begin{equation*}
    \begin{aligned}
    \mathbb{E}\left[|\langle \ckX_t^N-\ckX_t^M, K_t^N(\ckX_t^N)- K_t^M(\ckX_t^M)\rangle|\right] &\leq C(\epsilon,R,\delta) \mathbb{E}\left[\|\ckX^N-\ckX^M\|_2^2\right] \\
    & + C(\epsilon,\gamma,\delta)\mathbb{E}\left[\|\ckX^N-\ckX^M\|_{1,2}^2 \right] + \frac{C(k,R,a,\epsilon,\delta)}{M \wedge N}.
    \end{aligned}
\end{equation*}
B. There exists a constant $C(a, R)$ such that 
\begin{equation*}
    \begin{aligned}
    \mathbb{E}\left[|\langle \ckX_t^N-\ckX_t^M, L_t^N(\ckX_t^N)- L_t^M(\ckX_t^M)\rangle|\right] \leq \frac{C(a,R)}{M\wedge N}. 
    \end{aligned}
\end{equation*}
C. There exists a constant $C(a, R)$ such that 
\begin{equation*}
    \begin{aligned}
    \mathbb{E}\left[|\langle \ckX_t^N-\ckX_t^M, W_t^N(\ckX_t^N)- W_t^M(\ckX_t^M)\rangle|\right] \leq \frac{C(a,R)}{M\wedge N}. 
    \end{aligned}
\end{equation*}
\end{lemma}
\begin{proof}[\textbf{Proof of Lemma \ref{lemma:estimgeneral}:}] 
 \noindent\textbf{A. We use the following notation:  $\bar{K}^{N,M}(\ckX^N,\ckX^M):=K^N(\ckX^N) - K^M(\ckX^M)$.}\\
 Note that 
 \begin{equation*}
    \begin{aligned}
    |K^N(\ckX^N) - K^M(\ckX^M)|& \leq |K_1^N(\ckX_t^N) - K_1^M(\ckX_t^M)| + |K_2^M(\ckX_t^M) - K_2^N(\ckX_t^N)| \\
    & \leq |K_{11}^N(\ckX^N) - K_{11}^M(\ckX^M)| + \delta |K_{12}^N(\ckX^N) - K_{12}^M(\ckX^M)| + |K_2^M(\ckX^M) - K_2^N(\ckX^N)|
    \end{aligned}
\end{equation*}
then 
 \begin{itemize}
     \item [a.] \textbf{In the equation for $\ckh^N-\ckh^M$ we have $\bar{K}_1^{N,M}(\ckh^N,\ckh^M) := \bar{K}_{11}^{N,M}(\check{h}^N,\ckh^M) + \beta \bar{K}_{12}^{N,M}(\check{h}^N,\ckh^M)$:}
     \begin{itemize}
     \item [a1.] \textbf{For the first term, which contains $\bar{K}_{11}^{N,M}(\ckh^N,\ckh^M)$, we provide the detailed analysis in Lemma \ref{lemma:estimdiv} below. Summarising, we have:}
\begin{equation*}
\begin{aligned}
&\mathbb{E}\left[|\langle \ckh^N-\ckh^M, \bar{K}_{11}^{N,M}(\ckh^N,\ckh^M)\rangle|\right]\\
& =\mathbb{E}\left[|\langle \ckh^N-\ckh^M, \cJ_Nf_R(\check{X}^N) \nabla \cdot (\cJ_N\check{u}^N\cJ_N\check{h}^N)) - \cJ_Mf_R(\check{X}^M)\nabla\cdot(\cJ_M\check{u}^M\cJ_M\check{h}^M) \rangle|\right] \\
    & \leq C(\epsilon,R) \mathbb{E}\left[\|\ckX^N-\ckX^M\|_2^2\right] + C(\epsilon,\gamma)\mathbb{E}\left[\|\ckX^N-\ckX^M\|_{1,2}^2 \right] + \frac{C(k,R)}{M \wedge N}.
\end{aligned}
\end{equation*}
         \item [a2.] \textbf{The second term, which contains $\beta \bar{K}_{12}^{N,M}(\ckh^N,\ckh^M)$, is handled as follows:}
         \begin{equation*}
             \begin{aligned}
          & |\langle \ckh^N-\ckh^M, \bar{K}_{12}^{N,M}(\ckh^N,\ckh^M) \rangle|\\
          & = | \langle \ckh^N - \ckh^M,  f_R(\ckX^N)\cJ_N(\cJ_N\ckh^N \nabla \cdot \cJ_N\cku^N) - f_R(\ckX^M)\cJ_M(\cJ_M\ckh^M \nabla \cdot \cJ_M\cku^M) \rangle | \\
           & \leq C(\epsilon)\|\ckh^N-\ckh^M\|_2^2 + \epsilon \| f_R(\ckX^N)\cJ_N(\cJ_N\ckh^N \nabla \cdot \cJ_N\cku^N) - f_R(\ckX^M)\cJ_M(\cJ_M\ckh^M \nabla \cdot \cJ_M\cku^M)\|_2^2 \\
           & =: C(\epsilon) \|\ckh^N-\ckh^M\|_2^2  + \epsilon \bar{T}_{\ckh^N}^{\cD} \\
           & \leq C(\epsilon) \|\ckX^N-\ckX^M\|_2^2  + \epsilon \bar{T}_{\ckh^N}^{\cD}
             \end{aligned}
         \end{equation*}
         and we want to keep the second term small. We can write
         \begin{equation*}
             \begin{aligned}
             \bar{T}_{\ckh^N}^{\cD} & \leq \| \cJ_N \left(f_R(\ckX^N)\cJ_N\ckh^N \nabla \cdot \cJ_N\cku^N  - f_R(\ckX^M)\cJ_M\ckh^M \nabla \cdot \cJ_M\cku^M \right)\|_2^2\\
             & + \|(\cJ_N-\cJ_M)f_R(\ckX^M)\cJ_M\ckh^M \nabla \cdot \cJ_M\cku^M \|_2^2  \\
             & \leq \bar{T}_{\ckh^N}^{\cD, 1} +  \bar{T}_{\ckh^N}^{\cD, 2}.
             \end{aligned}
         \end{equation*}
         Then we have
         \begin{equation*}
             \begin{aligned}
             \bar{T}_{\ckh^N}^{\cD, 2} &\leq \frac{1}{M\wedge N}\|f_R(\ckX^M)\cJ_M\ckh^M \nabla \cdot \cJ_M\cku^M \|_{1,2}^2 \\
             & \leq \frac{2}{M\wedge N}f_R(\ckX^M)\|\ckh^M\|_{k,2}^2\|\cku^M\|_{k,2}^2 \\
             & \leq \frac{C(R)}{M \wedge N}.
             \end{aligned}
         \end{equation*}
         For $\bar{T}_{\ckh^N}^{\cD, 1}$ we can do the same type of calculations as those from Lemma \ref{lemma:estimdiv} below to obtain\footnote{Note that here and everywhere else, the values of the constants $C$ differ from line to line and we keep track of their dependence on the parameters $R,\epsilon,\gamma,k,\beta$.}
         \begin{equation*}
             \begin{aligned}
                \mathbb{E}\left[ \bar{T}_{\ckh^N}^{\cD, 1}\right] & \leq C(\epsilon,R) \mathbb{E}\left[\|\ckX^N-\ckX^M\|_2^2\right] + C(\epsilon,\gamma)\mathbb{E}\left[\|\ckX^N-\ckX^M\|_{1,2}^2 \right] + \frac{C(k,R)}{M \wedge N}.
             \end{aligned}
         \end{equation*}
         Then
         \begin{equation*}
             \begin{aligned}
                \mathbb{E}\left[ \bar{T}_{\ckh^N}^{\cD}\right] & \leq \mathbb{E}\left[\bar{T}_{\ckh^N}^{\cD, 1}\right] + \mathbb{E}\left[\bar{T}_{\ckh^N}^{\cD, 2}\right]  \\
                & \leq C(\epsilon,R) \mathbb{E}\left[\|\ckX^N-\ckX^M\|_2^2\right] + C(\epsilon,\gamma)\mathbb{E}\left[\|\ckX^N-\ckX^M\|_{1,2}^2 \right] + \frac{C(k,R)}{M \wedge N}.
             \end{aligned}
         \end{equation*}
         and therefore
         \begin{equation*}
             \begin{aligned}
                \mathbb{E}\left[|\langle \ckh^N-\ckh^M, \bar{K}_{12}^{N,M}(\ckh^N,\ckh^M)\rangle|\right] \leq C(\epsilon,R) \mathbb{E}\left[\|\ckX^N-\ckX^M\|_2^2\right] + C(\epsilon,\gamma)\mathbb{E}\left[\|\ckX^N-\ckX^M\|_{1,2}^2 \right] + \frac{C(k,R)}{M \wedge N}.
             \end{aligned}
         \end{equation*}
   \end{itemize}
   In conclusion,
   \begin{equation*}
       \begin{aligned}
          & \mathbb{E}\left[|\langle \ckh^N-\ckh^M, \bar{K}_1^{N,M}(\ckX^N,\ckX^M)\rangle|\right] \\
          & \leq \mathbb{E}\left[|\langle \ckh^N-\ckh^M,\bar{K}_{11}^{N,M}(\check{h}^N,\ckh^M)\rangle|\right] + \beta\mathbb{E}\left[ |\langle \ckh^N-\ckh^M,\bar{K}_{12}^{N,M}(\check{h}^N,\ckh^M)\rangle|\right]\\
          &\leq C(\epsilon,R,\beta) \mathbb{E}\left[\|\ckX^N-\ckX^M\|_2^2\right] + C(\epsilon,\gamma,\beta)\mathbb{E}\left[\|\ckX^N-\ckX^M\|_{1,2}^2 \right] + \frac{C(k,R,\beta)}{M \wedge N}
       \end{aligned}
   \end{equation*}
   again with different values for the constants $C$. 
   \item [b.] \textbf{In the equation for $\ckh^N-\ckh^M$ we also need a control on $\bar{K}_2^{N,M}(\ckh^N, \ckh^M)$:} Using that 
   \begin{equation*}
      \begin{aligned}
       |\langle \check{h}^N-\check{h}^M,K_2^{N}(\check{h}^N) - K_2^{M}(\check{h}^M)  \rangle |&= |\langle \check{h}^N-\check{h}^M, \mathcal{J}_{N}(\mathcal{{L}}_{\mathcal{J}_{N}\check{u}^{S,N}} \mathcal{J}_{N}\check{h}^{N}) - \mathcal{J}_{M}(\mathcal{{L}}_{\mathcal{J}_{M}\check{u}^{S,M}} \mathcal{J}_{M}\check{h}^{M}) \rangle |\\
       & \leq C(\epsilon)\|\ckh^N-\ckh^M\|_2^2 + \frac{C(a,\epsilon)}{M \wedge N} \|\ckh^M\|_{1,2}^2\\
       & \leq C(\epsilon)\|\ckX^N-\ckX^M\|_2^2 + \frac{C(a,\epsilon)}{M \wedge N} \|\ckX^M\|_{1,2}^2
       \end{aligned}
   \end{equation*}
   the control for the corresponding expectation is given by
   \begin{equation*}
       \begin{aligned}
        \mathbb{E}\left[|\langle \ckh^N-\ckh^M,\bar{K}_2^{N,M}(\ckh^N, \ckh^M)\rangle|\right] & \leq C(\epsilon) \mathbb{E}[\|\ckX^N-\ckX^M\|_2^2] + \frac{C(a,\epsilon)}{M \wedge N}\mathbb{E}\left[\|\ckX^M\|_{k,2}^2\right]\\
        & \leq C(\epsilon) \mathbb{E}[\|\ckX^N-\ckX^M\|_2^2] + \frac{C(a,\epsilon,k,R)}{M \wedge N}.
       \end{aligned}
   \end{equation*}
\end{itemize}
We presented the detailed analysis for $\ckh^N-\ckh^M$ but exactly the same type of calculations hold for $\cku^N-\cku^M$. 
Putting all these together we obtain
 \begin{equation*}
     \begin{aligned}
    &\mathbb{E}\left[|\langle \ckX^N-\ckX^M,K^N(\ckX^N) - K^M(\ckX^M)\rangle |\right]\\
    & \leq \mathbb{E}\left[|\langle \ckX^N-\ckX^M, K_1^N(\ckX_t^N) - K_1^M(\ckX_t^M)\rangle|\right] + \mathbb{E}\left[|\langle \ckX^N-\ckX^M,K_2^M(\ckX_t^M) - K_2^N(\ckX_t^N)\rangle |\right] \\
    & \leq \mathbb{E}\left[|\langle \ckX^N-\ckX^M,K_{11}^N(\ckX^N) - K_{11}^M(\ckX^M)\rangle |\right] + \delta \mathbb{E}\left[|\langle \ckX^N-\ckX^M, K_{12}^N(\ckX^N) - K_{12}^M(\ckX^M)\rangle| \right]\\
    & +\mathbb{E} \left[|\langle \ckX^N-\ckX^M, K_2^M(\ckX^M) - K_2^N(\ckX^N)\rangle |\right] \\
    & = \mathbb{E}\left[|\langle \ckX^N-\ckX^M, \bar{K}_{11}^{N,M}(\ckX^N,\ckX^M)\rangle |\right] + \delta \mathbb{E}\left[|\langle \ckX^N-\ckX^M, \bar{K}_{12}^{N,M}(\ckX^N,\ckX^M)\rangle | \right]\\
    & + \mathbb{E} \left[|\langle \ckX^N-\ckX^M, \bar{K}_2^{N,M}(\ckX^N, \ckX^M)\rangle |\right] \\
    & \leq C(\epsilon,R,\delta) \mathbb{E}\left[\|\ckX^N-\ckX^M\|_2^2\right] + C(\epsilon,\gamma,\delta)\mathbb{E}\left[\|\ckX^N-\ckX^M\|_{1,2}^2 \right] + \frac{C(k,R,a,\epsilon,\delta)}{M \wedge N}
     \end{aligned}
 \end{equation*}
which is exactly the control we were looking for. \\

\noindent\textbf{B. For $\bar{L}^{N,M}(\ckX^N, \ckX^M):=L^N(\ckX^N) - L^M(\ckX^M)$:} 


\begin{equation*}
\begin{aligned}
	& \langle \check{h}^N - \check{h}^M, \cJ_N(\nabla \cdot (a\cJ_N\nabla \ckh^N)) - \cJ_M(\nabla \cdot (a\cJ_M\nabla \ckh^M))\rangle  \\
	& = - \langle \nabla(\ckh^N-\ckh^M), a\cJ_N^2 \nabla \ckh^N\rangle + \langle \nabla(\ckh^N-\ckh^M), a\cJ_M^2 \nabla \ckh^M\rangle \\
	& = \langle \nabla(\ckh^N - \ckh^M), a(\cJ_M^2\nabla \ckh^M - \cJ_N^2\nabla \ckh^N)\rangle \\
	& = a \langle \nabla(\ckh^N - \ckh^M), \cJ_M^2\nabla(\ckh^M-\ckh^N) + (\cJ_M^2 - \cJ_N^2)\nabla \ckh^N \rangle \\
	& = -a\langle \cJ_M(\nabla(\ckh^N-\ckh^M)), \cJ_M(\nabla(\ckh^N-\ckh^M)) \rangle -a\langle \nabla(\ckh^N-\ckh^M), (\cJ_N^2 - \cJ_M^2)\nabla \ckh^N\rangle \\
	& = a\langle \ckh^N - \ckh^M, (\cJ_N^2 - \cJ_M^2)\Delta \ckh^N\rangle - a \| \cJ_M(\nabla (\ckh^N - \ckh^M))\|_2^2 \\
	& \leq \frac{C(a)}{N\wedge M}\|\ckh^N\|_{3,2}\|\ckh^N - \ckh^M\|_2 \\
	& \leq \frac{\tilde{C}_1(a)}{N\wedge M}\|\ckh^N - \ckh^M\|_2^2 + \frac{\tilde{C}_2(a)}{N\wedge M}\|\ckh^N\|_{k,2}^2 \\
	& \leq \frac{\tilde{C}_1(a)}{N\wedge M}\|\ckX^N - \ckX^M\|_2^2 + \frac{\tilde{C}_2(a)}{N\wedge M}\|\ckX^N\|_{k,2}^2
\end{aligned}
\end{equation*}
Then 
\begin{equation*}
    \begin{aligned}
       & \mathbb{E}\left[\langle \check{h}^N - \check{h}^M, \cJ_N(\nabla \cdot (a\cJ_N\nabla \ckh^N)) - \cJ_M(\nabla \cdot (a\cJ_M\nabla \ckh^M))\rangle\right] \leq \frac{\tilde{C}_1(a)}{N\wedge M}\mathbb{E}\left[\|\ckX^N - \ckX^M\|_2^2\right] + \frac{\tilde{C}_2(a)}{N\wedge M}\mathbb{E}\left[\|\ckX^N\|_{k,2}^2\right] \\
       & \leq \frac{\tilde{C}_1(a)}{N\wedge M}\left(\mathbb{E}\left[\|\ckX^N\|_2^2\right] + \mathbb{E}\left[\|\ckX^M\|_2^2\right]\right)+\frac{\tilde{C}_2(a)}{N\wedge M}\mathbb{E}\left[\|\ckX^N\|_{k,2}^2\right] \\
       & \leq \frac{C(a,R)}{M\wedge N}.
    \end{aligned}
\end{equation*}
\noindent\textbf{C. For $\bar{W}^{N,M}(\ckX^N,\ckX^M):=W^N(\ckX^N) - W^M(\ckX^M)$:}\\
\begin{itemize}
\item [a.] The stochastic term:
\begin{equation*}
    \begin{aligned}
        |\langle \check{h}_t^N - \check{h}_t^M, W^{N}(\check{h}^N) - W^{M}(\check{h}^M)  \rangle | = |\langle \check{h}_t^N - \check{h}_t^M,  \mathcal{J}_{N}\left(\mathcal{L}_{\sigma}\mathcal{J}_{N}\check{h}^{N}\right) - \mathcal{J}_{M}\left(\mathcal{L}_{\sigma}\mathcal{J}_{N}\check{h}^{M}\right)\rangle |
    \end{aligned}
\end{equation*}
Since $\nabla \cdot \sigma dB_t = 0$,
\begin{equation*}
    \begin{aligned}
       | \langle {\ckX}^{N}-\ckX^M, W^N({\ckX}^{N}) - W^M({\ckX}^{M})\rangle |  = 0 
    \end{aligned}
\end{equation*}
and therefore the stochastic term vanishes.
\item [b.] The It\^{o} correction:
\begin{equation*}
    \begin{aligned}
        |\langle W^N({\ckX}^{N}) - W^M({\ckX}^{M}), W^N({\ckX}^{N}) - W^M({\ckX}^{M}) \rangle | & \leq \frac{C(a)}{M\wedge N}\|\ckX^N-\ckX^M\|_{2,2}^2
    \end{aligned}
\end{equation*}
and then 
\begin{equation*}
    \begin{aligned}
       \mathbb{E}\left[|\langle W^N({\ckX}^{N}) - W^M({\ckX}^{M}), W^N({\ckX}^{N}) - W^M({\ckX}^{M}) \rangle | \right] & \leq \frac{C(a)}{M\wedge N}\mathbb{E}\left[\|\ckX^N-\ckX^M\|_{2,2}^2\right]\\
       & \leq \frac{C(a,R)}{M\wedge N}.
    \end{aligned}
\end{equation*}
\end{itemize}
\end{proof}

\begin{remark}
Note that we moved $F^N - F^M$ to the left hand side above since we have
\begin{equation*}
    \begin{aligned}
         \langle \check{h}_t^N - \check{h}_t^M,F_t^{N}(\check{h}_t^N)- F_t^{M}(\check{h}_t^M) \rangle  &= \eta  \langle \check{h}_t^N - \check{h}_t^M, \Delta\check{h}_t^N - \Delta\check{h}_t^M \rangle   = -\eta \|\nabla(\check{h}_t^N - \check{h}_t^M)\|_2^2
    \end{aligned}
\end{equation*}
and likewise for $u$ so
\begin{equation*}
    \begin{aligned}
         \langle \check{X}_t^N - \check{X}_t^M,F_t^{N}(\check{X}_t^N)- F_t^{M}(\check{X}_t^M) \rangle  &= \gamma  \langle \check{X}_t^N - \check{X}_t^M,\Delta\check{X}_t^N - \Delta\check{X}_t^M \rangle   = -\gamma \|\nabla(\check{X}_t^N - \check{X}_t^M)\|_2^2.
    \end{aligned}
\end{equation*}

\end{remark}   

\vspace{5mm}         
\begin{lemma}\label{lemma:estimdiv}
Let $\epsilon >0$ and $R, \gamma, k$ as above. For any $M,N>0$ there exist some constants which may depend on $\epsilon, R, \gamma, k$ such that the following control holds:
\begin{equation*}
\begin{aligned}
    &\mathbb{E}\left[|\langle \check{h}^N-\check{h}^M, \cJ_Nf_R(\ckX^N) \nabla \cdot (\cJ_N\check{u}^N\cJ_N\check{h}^N)) - \cJ_Mf_R(\check{X}^M)\nabla\cdot(\cJ_M\check{u}^M\cJ_M\check{h}^M) \rangle|\right] \\
    & \leq C(\epsilon,R) \mathbb{E}\left[\|\ckX^N-\ckX^M\|_2^2\right] + C(\epsilon,\gamma)\mathbb{E}\left[\|\ckX^N-\ckX^M\|_{1,2}^2 \right] + \frac{C(k,R)}{M \wedge N}. 
\end{aligned}
\end{equation*}
\end{lemma}
\begin{proof}[\textbf{Proof of Lemma \ref{lemma:estimdiv}:}] 

Using the fact that  
\begin{equation*}
    \begin{aligned}
        & \cJ_Nf_R(\ckX^N)(\cJ_N\cku^N\cJ_N\ckh^N)-\cJ_Mf_R(\ckX^M)(\cJ_M\cku^M\cJ_M\ckh^M) \\
        &= \cJ_N\left( f_R(\ckX^N)\cJ_N\cku^N\cJ_N\ckh^N-f_R(\ckX^M)\cJ_M\cku^M\cJ_M\ckh^M\right) \\
        & + (\cJ_N-\cJ_M)f_R(\ckX^M)\cJ_M\cku^M\cJ_M\ckh^M \\
    \end{aligned}
\end{equation*}
we can write
\begin{equation*}
    \begin{aligned}
        T& := |\langle \ckh^N-\ckh^M, \cJ_Nf_R(\ckX^N) \nabla \cdot (\cJ_N\cku^N\cJ_N\ckh^N)) - \cJ_Mf_R(\ckX^M)\nabla\cdot(\cJ_M\cku^M\cJ_M\ckh^M) \rangle| \\
        & = |\langle \nabla(\ckh^N-\ckh^M), \cJ_N\left( f_R(\ckX^N)\cJ_N\cku^N\cJ_N\ckh^N-f_R(\ckX^M)\cJ_M\cku^M\cJ_M\ckh^M\right) + (\cJ_N-\cJ_M)f_R(\ckX^M)\cJ_M\cku^M\cJ_M\ckh^M \rangle | \\
        & \leq |\langle \cJ_N\nabla(\ckh^N-\ckh^M), f_R(\ckX^N)\cJ_N\cku^N\cJ_N\ckh^N-f_R(\ckX^M)\cJ_M\cku^M\cJ_M\ckh^M\rangle| \\
        & + |\langle (\cJ_N-\cJ_M)\nabla(\ckh^N-\ckh^M),\cJ_M\cku^M\cJ_M\ckh^M\rangle| =: T^1 + T^2.
    \end{aligned}
\end{equation*}
We start with the second term for which the analysis is simpler:
\begin{equation*}
    \begin{aligned}
    \mathbb{E}\left[T^2\right] &= \mathbb{E}\left[ |\langle (\cJ_N-\cJ_M)\nabla(\ckh^N-\ckh^M),\cJ_M\cku^M\cJ_M\ckh^M\rangle|\right] \\
    & \leq C(\epsilon,\gamma)\mathbb{E}\left[\|\ckX^N-\ckX^M\|_{1,2}^2 \right] + \frac{C(k)}{M\wedge N}.
    \end{aligned}
\end{equation*}
This control is similar to (but simpler than) the control obtained for $\Xi^2$ below so we do not re-do all the detailed calculations. 
To control the expectation of the first term (i.e. $\mathbb{E}\left[T^1\right]$) we use the fact that
\begin{equation*}
    \begin{aligned}
       \mathbb{E}\left[ T^1\right] &=  \mathbb{E}\left[T^1\mathds{1}_{\{\ckX^N,\ckX^M \in B(0,R+1)\}} \right] + \mathbb{E}\left[ T^1\mathds{1}_{\{\ckX^N,\ckX^M \notin B(0,R+1)\}} \right] \\
       & + \mathbb{E}\left[ T^1\mathds{1}_{\{\ckX^N \in B(0,R+1) \setminus B(0,R),\ckX^M \notin B(0,R+1)\}} \right] + \mathbb{E}\left[ T^1\mathds{1}_{\{\ckX^N \in B(0,R),\ckX^M \notin B(0,R+1)\}} \right] \\
       & =: E^1 + E^2 + E^3 + E^4.
    \end{aligned}
\end{equation*}
and we control each of these four terms separately below.\\

\noindent\textbf{Analysis for $E^1$:}
Note that we can write 
\begin{equation*}
    \begin{aligned}
        &|\langle \cJ_N(\ckh^N-\ckh^M), f_R(\ckX^N) \nabla \cdot (\cJ_N\cku^N\cJ_N\ckh^N)) - f_R(\ckX^M)\nabla\cdot(\cJ_M\cku^M\cJ_M\ckh^M) \rangle|\mathds{1}_{\{\ckX^N,\ckX^M \in B(0,R+1)\}}  \\
        & = |\langle \cJ_N\nabla(\ckh^N-\ckh^M),f_R(\ckX^N)(\cJ_N\cku^N\cJ_N\ckh^N)-f_R(\ckX^M)(\cJ_M\cku^M\cJ_M\ckh^M)\rangle|\mathds{1}_{\{\ckX^N,\ckX^M \in B(0,R+1)\}}  \\
        & \leq |\langle \cJ_N\nabla(\ckh^N-\ckh^M), f_R(\ckX^N)(\cJ_N\cku^N\cJ_N\ckh^N-\cJ_M\cku^M\cJ_M\ckh^M)\rangle|\mathds{1}_{\{\ckX^N,\ckX^M \in B(0,R+1)\}}  \\
        & + |\langle \cJ_N\nabla(\ckh^N-\ckh^M), |f_R(\ckX^N)-f_R(\ckX^M)|\cJ_M\cku^M\cJ_M\ckh^M\rangle |\mathds{1}_{\{\ckX^N,\ckX^M \in B(0,R+1)\}}  \\
        & =: \Xi^1 + \Xi^2.
    \end{aligned}
\end{equation*}
For $\Xi^1$ we have 
\begin{equation*}
    \begin{aligned}
      \mathbb{E}\left[\Xi^1\right] &= \mathbb{E}\left[|\langle \cJ_N\nabla(\ckh^N-\ckh^M), f_R(\ckX^N)(\cJ_N\cku^N\cJ_N\ckh^N-\cJ_M\cku^M\cJ_M\ckh^M)\rangle|\mathds{1}_{\{\ckX^N,\ckX^M \in B(0,R+1)\}}  \right]\\
       & \leq \epsilon\mathbb{E}\left[\|\nabla(\ckh^N-\ckh^M)\|_2^2\right] + C(\epsilon)\mathbb{E}\left[\|f_R(\ckX^N)(\cJ_N\cku^N\cJ_N\ckh^N-\cJ_M\cku^M\cJ_M\ckh^M)\|_2^2 \mathds{1}_{\{\ckX^N,\ckX^M \in B(0,R+1)\}} \right]\\
       & \leq \epsilon\mathbb{E}\left[\|\nabla(\ckh^N-\ckh^M)\|_2^2\right] + C(\epsilon)\mathbb{E}\left[\|\cJ_N\cku^N\cJ_N\ckh^N-\cJ_M\cku^M\cJ_M\ckh^M\|_2^2\mathds{1}_{\{\ckX^N,\ckX^M \in B(0,R+1)\}}\right]\\
       & \leq \epsilon\mathbb{E}\left[\|\nabla(\ckh^N-\ckh^M)\|_2^2 \right] + C(\epsilon)2(R+1)\mathbb{E}\left[\|\ckX^N-\ckX^M\|_{2}^2 \right]\\
       & \leq C(\epsilon,\gamma)\mathbb{E}\left[\|\ckX^N-\ckX^M\|_{1,2}^2 \right] + C(\epsilon, R)\mathbb{E}\left[\|\ckX^N-\ckX^M\|_{2}^2 \right].
    \end{aligned}
\end{equation*}
We used above the fact that
\begin{equation*}
    |\cku^N\ckh^N-\cku^M\ckh^M| \leq |\cku^N(\ckh^N-\ckh^M)| + |\ckh^M(\cku^N-\cku^M)|
\end{equation*}
and $\cku^N, \ckh^M \in B(0, R+1)$.
For $\Xi^2$ we have, using that $\cJ_N\nabla(\ckh^N-\ckh^M) = \nabla(\ckh^N-\ckh^M)$,
\begin{equation*}
    \begin{aligned}
    \mathbb{E}\left[\Xi^2\right] & = \mathbb{E}\left[|\langle \nabla(\ckh^N-\ckh^M), |f_R(\ckX^N)-f_R(\ckX^M)|\cJ_M\cku^M\cJ_M\ckh^M\rangle | \mathds{1}_{\{\ckX^N,\ckX^M \in B(0,R+1)\}}\right]\\
    & \leq \mathbb{E}\left[|f_R(\ckX^N)-f_R(\ckX^M)|  |\langle \nabla(\ckh^N-\ckh^M), \cJ_M\cku^M\cJ_M\ckh^M\rangle|\mathds{1}_{\{\ckX^N,\ckX^M \in B(0,R+1)\}}\right] \\
    & \leq \mathbb{E}\left[ \|\ckX^N-\ckX^M\|_{k,2}\|\nabla(\ckh^N-\ckh^M)\|_2\|\cJ_M\cku^M\cJ_M\ckh^M\|_2\mathds{1}_{\{\ckX^N,\ckX^M \in B(0,R+1)\}}\right]\\
    & \leq \mathbb{E}\left[ \|\ckX^N-\ckX^M\|_{k,2}\|\nabla(\ckh^N-\ckh^M)\|_2\|\cJ_M\cku^M\|_4\|\cJ_M\ckh^M\|_4\mathds{1}_{\{\ckX^N,\ckX^M \in B(0,R+1)\}}\right] \\
    & \leq \mathbb{E}\left[ \|\ckX^N-\ckX^M\|_{k,2}\|\nabla(\ckh^N-\ckh^M)\|_2\|\cJ_M\cku^M\|_{k-1,2}\|\cJ_M\ckh^M\|_{k-1,2}\mathds{1}_{\{\ckX^N,\ckX^M \in B(0,R+1)\}}\right] \\
    & \leq \mathbb{E}\left[\|\ckX^M\|_{k-1,2}^2 \|\ckX^N-\ckX^M\|_{k,2}\|\nabla(\ckh^N-\ckh^M)\|_2\mathds{1}_{\{\ckX^N,\ckX^M \in B(0,R+1)\}}\right] \\
    & \leq \frac{1}{M \wedge N}\mathbb{E}\left[\|\ckX^M\|_{k,2}^2\|\nabla(\ckh^N-\ckh^M)\|_2 \|\ckX^N-\ckX^M\|_{k,2} \mathds{1}_{\{\ckX^N,\ckX^M \in B(0,R+1)\}}\right]\\
    & \leq \frac{C(k)}{M \wedge N}\mathbb{E}\left[\|\ckX^M\|_{k,2}^2\|\nabla(\ckh^N-\ckh^M)\|_2 \|\ckX^N-\ckX^M\|_{2}^{\frac{1}{k+1}}\|\ckX^N-\ckX^M\|_{k+1,2}^{\frac{k}{k+1}} \mathds{1}_{\{\ckX^N,\ckX^M \in B(0,R+1)\}}\right]\\
    & \leq \frac{C(k)}{M \wedge N}\mathbb{E}\left[\|\ckX^M\|_{k,2}^2\|\ckh^N-\ckh^M\|_2^{\frac{k}{k+1}}\|\ckh^N-\ckh^M\|_{k+1,2}^{\frac{1}{k+1}} \|\ckX^N-\ckX^M\|_{2}^{\frac{1}{k+1}}\|\ckX^N-\ckX^M\|_{k+1,2}^{\frac{k}{k+1}} \mathds{1}_{\{\ckX^N,\ckX^M \in B(0,R+1)\}}\right]\\
    & \leq \frac{C(k)}{M \wedge N}\mathbb{E}\left[\|\ckX^M\|_{k,2}^2\|\ckX^N-\ckX^M\|_2^{\frac{k}{k+1}}\|\ckX^N-\ckX^M\|_{k+1,2}^{\frac{1}{k+1}} \|\ckX^N-\ckX^M\|_{2}^{\frac{1}{k+1}}\|\ckX^N-\ckX^M\|_{k+1,2}^{\frac{k}{k+1}} \mathds{1}_{\{\ckX^N,\ckX^M \in B(0,R+1)\}}\right]\\
    & = \frac{C(k)}{M \wedge N}\mathbb{E}\left[\|\ckX^M\|_{k,2}^2\|\ckX^N-\ckX^M\|_2\|\ckX^N-\ckX^M\|_{k+1,2} \mathds{1}_{\{\ckX^N,\ckX^M \in B(0,R+1)\}}\right]\\
    & \leq \frac{\tilde{C}(k)}{M \wedge N}\mathbb{E}\left[\|\ckX^M\|_{k,2}^2\|\ckX^N-\ckX^M\|_2\|\ckX^N-\ckX^M\|_{k+1,2} \mathds{1}_{\{\ckX^N,\ckX^M \in B(0,R+1)\}}\right]\\
    & \leq \mathbb{E}\left[\|\ckX^M\|_{k,2}^2\|\ckX^N-\ckX^M\|_2\frac{\tilde{C}(k)}{M \wedge N} \|\ckX^N-\ckX^M\|_{k+1,2}\mathds{1}_{\{\ckX^N,\ckX^M \in B(0,R+1)\}}\right]\\
    & \leq \mathbb{E}\left[\|\ckX^M\|_{k,2}^2\mathds{1}_{\{\ckX^N,\ckX^M \in B(0,R+1)\}}\|\ckX^N-\ckX^M\|_2\frac{\tilde{C}(k)}{M \wedge N} \|\ckX^N-\ckX^M\|_{k+1,2} \right]\\
    & \leq \mathbb{E}\left[C(R)\|\ckX^N-\ckX^M\|_2\frac{\tilde{C}(k)}{M \wedge N} \|\ckX^N-\ckX^M\|_{k+1,2} \right]\\
    & \leq \frac{C(R)}{2}\mathbb{E}\left[\|\ckX^N-\ckX^M\|_2^2\right] + \frac{\tilde{C}(k)^2}{2(M \wedge N)^2} \mathbb{E}\left[\|\ckX^N-\ckX^M\|_{k+1,2}^2\right].
    \end{aligned}
\end{equation*}
That is, using the uniform control obtained in Lemma [unifcontrol]:
\begin{equation*}
   \begin{aligned}
   \mathbb{E}\left[\Xi^2\right]  
   & \leq \tilde{C}(R) \mathbb{E}\left[\|\ckX^N-\ckX^M\|_2^2\right] + \frac{\tilde{C}_1(k)}{M \wedge N} \left(\mathbb{E}\left[\|\ckX^N\|_{k+1,2}^2 \right] + \mathbb{E}\left[\|\ckX^M\|_{k+1,2}^2 \right]\right) \\
   & \leq \tilde{C}(R) \mathbb{E}\left[\|\ckX^N-\ckX^M\|_2^2\right] + \frac{\tilde{C}_2(k,R)}{M \wedge N}.
    \end{aligned}
\end{equation*}
In conclusion
\begin{equation*}
    \begin{aligned}
       E^1 \leq C(\epsilon,R)\mathbb{E}\left[\|\ckX^N-\ckX^M\|_2^2 \right] + C(\epsilon,\gamma)\mathbb{E}\left[\|\ckX^N-\ckX^M\|_{1,2}^2 \right] + \frac{C(k,R)}{M\wedge N}.
    \end{aligned}
\end{equation*}

\vspace{3mm}
\noindent\textbf{Analysis of $E^2$:} We have 
\begin{equation*}
    \begin{aligned}
    0 = \|f_R(\ckX^N)\cJ_N\ckh^N\cJ_N\cku^N - f_R(\ckX^M)\cJ_M\ckh^M\cJ_M\cku^M \|_2^2 \leq C\|\ckX^N-\ckX^M\|_2^2.
    \end{aligned}
\end{equation*}
That is
\begin{equation*}
    E^2 \leq C\mathbb{E}\left[\|\ckX^N-\ckX^M\|_2^2\right].
\end{equation*}
\vspace{3mm}
\noindent\textbf{Analysis of $E^3$:} We have 
\begin{equation*}
    \begin{aligned}
    0 = \|f_R(\ckX^N)\cJ_N\ckh^N\cJ_N\cku^N - f_R(\ckX^M)\cJ_M\ckh^M\cJ_M\cku^M \|_2^2 \leq C \|\ckX^N-\ckX^M\|_2^2.
    \end{aligned}
\end{equation*}
That is
\begin{equation*}
    E^3 \leq C\mathbb{E}\left[\|\ckX^N-\ckX^M\|_2^2\right].
\end{equation*}

\noindent\textbf{Analysis of $E^4$:} We have 
\begin{equation*}
    \begin{aligned}
    \|f_R(\ckX^N)\cJ_N\ckh^N\cJ_N\cku^N - f_R(\ckX^M)\cJ_M\ckh^M\cJ_M\cku^M \|_2^2 = \|f_R(\ckX^N)\ckh^N\cku^N\|_2^2 \leq C(R) \leq C(R) \|\ckX^N-\ckX^M\|_2^2.
    \end{aligned}
\end{equation*}
since $\|\ckX^N-\ckX^M\|_2 \geq |\|\ckX^M\|_2 - \|\ckX^N\|_2 |= (R+1)-R=1$ so $\|\ckX^N-\ckX^M\|_2^2 \geq 1$.
That is
\begin{equation*}
    E^4 \leq C(R)\mathbb{E}\left[\|\ckX^N-\ckX^M\|_2^2\right].
\end{equation*}
Summing everything up we have
\begin{equation*}
    \begin{aligned}
       \mathbb{E}\left[T\right] & \leq \mathbb{E}\left[T^1\right] + \mathbb{E}\left[T^2\right] \\
       & \leq C(\epsilon,R)\mathbb{E}\left[\|\ckX^N-\ckX^M\|_2^2 \right] + C(\epsilon,\gamma)\mathbb{E}\left[\|\ckX^N-\ckX^M\|_{1,2}^2 \right] + \frac{C(k,R)}{M\wedge N} 
    \end{aligned}
\end{equation*}
which concludes the proof. 
\end{proof}

\begin{lemma}\label{generallemmauniquenesslu}
Let $(X^1,\tau^1), (X^2,\tau^2)$ be two local solutions and $$\bar{u}:=u^1-u^2, \ \ \ \tau^{1,2} = \tau^1\wedge\tau^2, \ \ \ X^i:=(u^i, h^i), \ \ \ \bar{X}:= (\bar{u}, \bar{h}),$$ 
$$Q({\bar{h}},\bar{u}):=f_R(X^1)\nabla \cdot (h^1 u^1) - f_R(X^2)\nabla \cdot(h^2u^2)+ f_R(X^1)\beta h^{1}(\nabla \cdot u^{1})-f_R(X^2)\beta h^{2}(\nabla \cdot u^{2})$$ where $\bar{h}:=h^1-h^2$. 
Then there exists $\epsilon > 0 $ and $C(\epsilon,\beta),C(\epsilon,R,\beta)$ such that for any $\theta$ with $|\theta|\leq k$
\begin{equation*}
\begin{aligned}
	|\langle \partial^{\theta}\bar{X}, \partial^{\theta}Q({\bar{h}},\bar{u})\rangle| 
	& \leq C(\epsilon,\beta)\|\bar{X}\|_{k+1,2}^2 + C(\epsilon,R, \beta)\|Z\|\|\bar{X}\|_{k,2}^2
\end{aligned}
\end{equation*}
with 
\begin{equation*}
    \|Z\| := \left( \|X^1\|_{k,2}^2 + \|X^2\|_{k,2}^2 \right)^2.
\end{equation*}
\end{lemma}
\begin{proof}[\textbf{Proof.}]
We have $\nabla \cdot (uh) = u \cdot \nabla h + h(\nabla \cdot u) = \mathcal{L}_u h + \mathcal{D}_u h$ for any vector $u$ and scalar $h$. Note that we can write 
\begin{equation*}
    \begin{aligned}
      Q(\bar{h},\bar{u}) =  f_R(X^1)\left( \mathcal{L}_{u^1}h^1 + (1+\beta)\mathcal{D}_{u^1}h^1\right) - f_R(X^2)\left( \mathcal{L}_{u^2}h^2 + (1+\beta)\mathcal{D}_{u^2}h^2\right)
    \end{aligned}
\end{equation*}
and we can decompose this as 
\begin{equation*}
\begin{aligned}
Q(\bar{h},\bar{u}) &= f_R(X^1)(\mathcal{L}_{\bar{u}}h^1 + (1+\beta)\mathcal{D}_{\bar{u}}h^1) + f_R(X^2)(\mathcal{L}_{u^2}\bar{h}+(1+\beta)\mathcal{D}_{u^2}\bar{h}) \\
& + |f_R(X^1) - f_R(X^2)| (\mathcal{L}_{u^2}h^1 + (1+\beta)\mathcal{D}_{u^2}h^1) \\
& := T_1 + T_2 + T_3.
\end{aligned}
\end{equation*}
For $\mathcal{L}_u h$ one can use the fact that
for any multi-index $\theta$ of length $|\theta| \leq k$
\begin{equation*}
\begin{aligned}
|\langle \partial^{\theta}\bar{X}, \partial^{\theta}(u \cdot \nabla h)\rangle| & = |\langle \partial^{\theta + 1}\bar{X}, \partial^{\theta - 1}(u \cdot \nabla h)\rangle|	\\
& \leq C \|\partial^{\theta + 1} \bar{X}\|_2\|\partial^{\theta-1}(u \cdot \nabla h)\|_{2} \\
& \leq C \|\partial^{\theta + 1} \bar{X}\|_2 \displaystyle\sum_{j \leq \theta -1}C\|\partial^{j} u\|_4\|\partial^{\theta-j} h\|_4 \\
& \leq C \|\partial^{\theta + 1} \bar{X}\|_2\|\partial^{j} u\|_2^{1/2}\|\partial^{j +1}u\|_2^{1/2}\|\partial^{\theta-j} h\|_2^{1/2}\|\partial^{\theta - j +1} h\|_2^{1/2}\\
& \leq C \|\bar{X}\|_{k+1,2}\|u\|_{k-1,2}^{1/2}\|u\|_{k,2}^{1/2}\|h\|_{k,2}^{1/2}\|h\|_{k,2}^{1/2} \\
& \leq C \|\bar{X}\|_{k+1,2}\|u\|_{k,2}\|h\|_{k,2} \\
& \leq \frac{\epsilon}{2} \|\bar{X}\|_{k+1,2}^2 + \frac{C_1(\epsilon, R)}{2} \|u\|_{k,2}^2\|h\|_{k,2}^2. 
\end{aligned}	
\end{equation*}
Similarly, for $\mathcal{D}_u h$ we use that 
\begin{equation*}
    \begin{aligned}
   (1+\beta) |\langle \partial^{\theta}\bar{X}, \partial^{\theta}(h (\nabla \cdot u))\rangle| & = (1+\beta)|\langle \partial^{\theta + 1} \bar{X}, \partial^{\theta - 1}(h (\nabla \cdot u))\rangle| \\
    & \leq (1+\beta)C \|\partial^{\theta +1} \bar{X}\|_2 \displaystyle\sum_{j\leq\theta - 1} C \|\partial^{j} h\|_4\|\partial^{\theta - j} u\|_4 \\
    & \leq (1+\beta)\left(\frac{\epsilon}{2} \|\bar{X}\|_{k+1,2}^2 + \frac{C_2(\epsilon, R)}{2} \|h\|_{k,2}^2\| u\|_{k,2}^2\right).
    \end{aligned}
\end{equation*}
Summing up we have
\begin{equation*}
    \begin{aligned}
        |\langle\partial^{\theta}\bar{X}, \partial^{\theta}(\mathcal{L}_uh + (1+\beta)\mathcal{D}_uh) \rangle| \leq C(\epsilon,\beta)\|\bar{X}\|_{k+1,2}^2 + C(\epsilon, R,\beta)\|h\|_{k,2}^2\|u\|_{k,2}^2.
    \end{aligned}
\end{equation*}
Now apply this for $T_i, i=1,2,3:$
\begin{equation*}
\begin{aligned}
|\langle \partial^{\theta}\bar{X}, \partial^{\theta}T_1\rangle| &
& \leq \frac{C(\epsilon,\beta)}{3}\|\bar{X}\|_{k+1,2}^2 + \frac{C(\epsilon, R,\beta)}{3}\|X^1\|_{k,2}^2\|h^1\|_{k,2}^2\|\bar{u}\|_{k,2}^2 
\end{aligned}
\end{equation*}
\begin{equation*}
\begin{aligned}
|\langle \partial^{\theta}\bar{X}, \partial^{\theta}T_2\rangle| 
& \leq \frac{C(\epsilon,\beta)}{3}\|\bar{X}\|_{k+1,2}^2 + \frac{C(\epsilon, R,\beta)}{3}\|X^2\|_{k,2}^2\|u^2\|_{k,2}^2\|\bar{h}\|_{k,2}^2
\end{aligned}
\end{equation*}
\begin{equation*}
\begin{aligned}
|\langle \partial^{\theta}\bar{X}, \partial^{\theta}T_3\rangle| & \leq C|f_R(X^1)-f_R(X^2)|\|\bar{X}\|_{k+1,2}\|h^1\|_{k,2}\|u^2\|_{k,2}\\
& \leq C\|\bar{X}\|_{k,2}\|\bar{X}\|_{k+1,2}\|h^1\|_{k,2}\|u^2\|_{k,2}\\
& \leq \frac{C(\epsilon,\beta)}{3}\|\bar{X}\|_{k+1,2}^2 + \frac{C(\epsilon, R,\beta)}{3}\|h^1\|_{k,2}^2\|u^2\|_{k,2}^2\|\bar{X}\|_{k,2}^2.
\end{aligned}
\end{equation*}
Define
\begin{equation*} 
\tilde{Z}:= \|X^1\|_{k,2}^2\|h^1\|_{k,2}^2 + \|X^2\|_{k,2}^2\|u^2\|_{k,2}^2 + \|h^1\|_{k,2}^2\|u^2\|_{k,2}^2 \leq \left( \|X^1\|_{k,2}^2 + \|X^2\|_{k,2}^2 \right)^2
\end{equation*}
and 
\begin{equation*}
Z := \left( \|X^1\|_{k,2}^2 + \|X^2\|_{k,2}^2 \right)^2.
\end{equation*}
Then 
\begin{equation*}
\begin{aligned}
	|\langle \partial^{\theta}\bar{X}, \partial^{\theta}Q({\bar{h}},\bar{u})\rangle| 
	& \leq C(\epsilon,\beta)\|\bar{X}\|_{k+1,2}^2 + C(\epsilon,R,\beta)\|Z\|\|\bar{X}\|_{k,2}^2.
\end{aligned}
\end{equation*}
\end{proof}

\vspace{3mm}
\noindent\textbf{Funding}\\
\noindent All three authors were partially supported by the European Research Council (ERC) under the European Union’s Horizon 2020 Research and Innovation Programme (ERC, Grant Agreement No 856408). 

\vspace{3mm}
\noindent\textbf{Data availability statement} \\
\noindent Data sharing not applicable to this article as no datasets were generated or analysed during the current study.

\end{document}